\documentclass{amsart}
\usepackage{pinlabel}
\usepackage{amsmath, amsthm, amssymb, amscd, mathrsfs, eucal}
\usepackage{hyperref}
\usepackage{float}
\usepackage{mathtools}
\usepackage{color}

\newtheorem{theorem}{Theorem}[section]
\newtheorem{lemma}[theorem]{Lemma}
\newtheorem{corollary}[theorem]{Corollary}

\newtheorem{proposition}[theorem]{Proposition}

\newtheorem*{claim*}{Claim}

\theoremstyle{definition}
\newtheorem{definition}[theorem]{Definition}
\newtheorem{example}[theorem]{Example}

\newcommand{\rr}{\mathbb{R}}

\newcommand{\nn}{\mathbb{N}}
\newcommand{\zz}{\mathbb{Z}}

\newcommand{\la}{\langle}
\newcommand{\ra}{\rangle}

\numberwithin{equation}{section}

\theoremstyle{theorem}
\newtheorem{theoremA}{Theorem}

\makeatletter                                                   
\newtheorem*{rep@theorem}{\rep@title}
\newcommand{\newreptheorem}[2]{%
	\newenvironment{rep#1}[1]{%
		\def\rep@title{#2 \ref{##1}}%
		\begin{rep@theorem}}%
		{\end{rep@theorem}}}
\makeatother

\theoremstyle{remark}
\newtheorem{remark}[theorem]{Remark}

\usepackage{color}
\usepackage{pdfcolmk}

\newcounter{joecomments}

\newcounter{ccomments}

\def\Z{\mathbb Z}

\def\R{\mathbb R}

\def\cl{\mathrm{cl}}
\def\scl{\mathrm{scl}}
\def\Aut{\mathrm{Aut}}
\def\tl{\mathrm{tl}}
\def\stl{\mathrm{stl}}
\def\Gt{G_{tor}}
\def\Ht{H_{tor}} 
\def\Gat{\Gamma_{tor}} 

\def\defeq{\vcentcolon=}
\def\tcP{\widetilde{\mathcal{P}}}
\newcommand{\RNum}[1]{\uppercase\expandafter{\romannumeral #1\relax}}

\title{Stable torsion length}
\author{Chloe I. Avery}
\address{Department of Mathematics\\ University of Chicago\\ Chicago, Illinois, USA}
\email[C.~Avery]{chloe@math.uchicago.edu}

\author{Lvzhou Chen}
\address{Department of Mathematics\\ University of Texas at Austin\\ Austin, Texas, USA}
\email[L.~Chen]{lvzhou.chen@math.utexas.edu}

\date{\today}

\begin{document}
	
	\begin{abstract}
	    The \emph{stable torsion length} in a group is the stable word length with respect to the set of all torsion elements. We show that the stable torsion length vanishes in crystallographic groups. We then give a linear programming algorithm to compute a lower bound for stable torsion length in free products of groups. Moreover, we obtain an algorithm that exactly computes stable torsion length in free products of finite groups. The nature of the algorithm shows that stable torsion length is rational in this case. As applications, we give the first exact computations of stable torsion length for nontrivial examples. 
		
	\end{abstract}
	
	\maketitle
	
	\setcounter{tocdepth}{1}
	\tableofcontents
	
	\section{Introduction}
	Given a generating set $S$ of a group $G$, the word length $|g|_S$ measures the least number of generators needed to express an element $g\in G$. 
	For finite generating sets, this is widely studied in geometric group theory, and different finite generating sets give equivalent word lengths up to scaling.
	On the other hand, many groups come with interesting and natural \emph{infinite} generating sets, for instance, the set of all commutators in $G$ (generating the commutator subgroup $[G,G]$), the set of torsion elements, and the set of words in a surface group representing simple closed loops. All of these examples are invariant under automorphisms.
	
	Understanding the word length of such infinite generating sets is often difficult, even for the basic question of whether the word length is bounded \cite{Cal:surface,MargalitPutman,BM}. 
	The first main result in this paper establishes boundedness for the word length with respect to the set of all torsion elements in crystallographic groups, namely those acting properly discontinuously and cocompactly on Euclidean spaces.
	
	\begin{theoremA}[Theorem \protect{\ref{thm: vanishing}}]\label{introthm: vanishing}
	    For any crystallographic group generated by torsion, the associated Cayley graph has finite diameter.
	\end{theoremA}
	
	In contrast, in non-elementary word-hyperbolic groups, there is no upper bound on word length with respect to the set of all torsion elements (see Remark \ref{remark: hyperbolic}). When the word length $|\cdot|_S$ with respect to a set $S$ is unbounded, it is interesting to investigate the \emph{stable word length} $\|g\|_S\defeq \lim_n \frac{|g^n|_S}{n}$, which measures the growth of the word length in the direction of $g$. Very little is known about stable word length for an infinite generating set in general, and giving good estimates or computing it is notoriously hard \cite{Cal:surface}. Most known results are about the stable commutator length \cite{Cal:rational,CF:sclhyp,Chen:sclBS}.
	
	In this paper, we use topological methods to study the stable word length with respect to conjugate-invariant generating sets. We focus on the special case of \emph{stable torsion length}, namely the stable word length with respect to the set of torsion elements in $G$, but a large portion of the argument works for other conjugate-invariant generating sets that are closed under taking powers. 
	
	We show that stable torsion length in a free product of finite groups is rational and can be computed by an algorithm.
	\begin{theoremA}[Rationality and Computability]\label{introthm: rationality}
	    If $G$ is a free product of arbitrarily many finite groups, then for any $g\in G$, the stable torsion length of $g$ is rational and computable.
	\end{theoremA}
	We prove a more technical version with weaker assumptions on factor groups in Theorem \ref{thm: rationality} when there are only two factors. The general case with more factors can be done in the exact same way. We do not pursue a fast algorithm here; see Remark \ref{remark: complexity} for a brief discussion on the computational complexity.
	
	For free products with general factor groups, we give a linear programming algorithm that computes an effective lower bound; see Section \ref{subsec: linprog lower bound}.
	
	We apply these methods to give the first exact computations of stable torsion length for nontrivial examples. 
	These formulas hold true in arbitrary free products by an isometric embedding theorem (Theorem \ref{introthm: Isom emb}) that we prove.
	
	\begin{theoremA}[Product formula; Theorem \protect{\ref{thm: product formula}}]\label{introthm: product formula}
	    Let $G=A*B$ be a free product, and let $a\in A$ and $b\in B$ be torsion elements of order $p$ and $q$ respectively such that $2\leq p\leq q$. Then 
		$$\stl_G(ab)=1-\frac{q}{p(q-1)}.$$
	\end{theoremA}
	
	\begin{theoremA}[Commutator formula; Theorem \protect{\ref{thm: compute abAB}}]\label{introthm: compute abAB}
	    Let $G=A*B$ be a free product, and let $a\in A$ and $b\in B$ be torsion elements of orders $p$ and $q$ respectively, where $p,q\ge2$. Then we have
	    $$\stl_G([a,b])=1-\frac{1}{\min(p,q)-1}.$$
	\end{theoremA}
	
	Our results show that the stable torsion length behaves in a way similar to the \emph{stable commutator length}, the stable word length with respect to the set of all commutators. In recent years, the study of stable commutator length has seen many advances \cite{Cal:rational,Chen:sclBS} and interesting applications to the surface subgroup problem \cite{Cal:surfsubgrp,CW:surfsubgrp,Wilton} and the simplicial volume \cite{HL:simpvol}. However, the tools established for stable commutator length are special in an essential way, and thus new tools are required to understand the stable torsion length despite the similarity; see Section \ref{subsec: methods}.
	
	
	In particular, we are not aware of an analog of the Bavard's duality (\cite{Bavard}, \cite[Theorem 2.70]{Cal:sclbook}) in the case of stable torsion length. Due to the lack of this duality, it is not easy to verify whether certain groups have trivial stable torsion length, such as amenable groups, which include crystallographic groups that we consider in Theorem \ref{introthm: vanishing}.
	
	\subsection{Methods}\label{subsec: methods}
	Given a group $G$, let $X$ be a topological space with $\pi_1(X)=G$. 
	Given a conjugate-invariant subset $S$, for a fixed $k\ge1$ there is an expression $g=s_1\cdots s_k$ for some $s_i\in S$ 
	if and only if there is a continuous map $f:\Sigma\to X$, where $\Sigma$ is a disk with $k$ subdisks $D_1,\cdots, D_k$ removed,
	so that each $\partial D_i$ represents a conjugacy class in $S$ and the remaining boundary component $\partial_0 \Sigma$ of $\Sigma$ represents the conjugacy class of $g$; 
	see Figure \ref{fig: admissible} for an illustration. We refer to such a surface as an $S$-admissible surface.
	
	Thus finding the word length of $g$ with respect to $S$ is to find the least complicated connected \emph{planar} surface in $X$ bounding $g$ in the above way.
	Similarly, finding the stable word length is to minimize $\frac{-\chi(\Sigma)+1}{n}$ over all connected planar surfaces $\Sigma$ bounding $g^n$ as above for some $n\in\Z_+$, which turns out to be the same as minimizing $\frac{-\chi(\Sigma)}{n}$ if $S$ is closed under taking powers; see Lemma \ref{lemma: stl via surfaces}. 
	
	\begin{figure}
		\centering
		\labellist
			\small \hair 2pt
			\pinlabel $\Sigma$ at 120 -20
			\pinlabel $X$ at 440 0
			\pinlabel $\partial_0\Sigma$ at 120 50
			\pinlabel $f$ at 300 140
			\pinlabel $\gamma$ at 440 65
			\pinlabel $\partial D_1$ at 20 200
			\pinlabel $\partial D_2$ at 80 220
			\pinlabel $\partial D_3$ at 150 220
			\pinlabel $\partial D_4$ at 220 200
			\endlabellist
		\includegraphics[scale=.4]{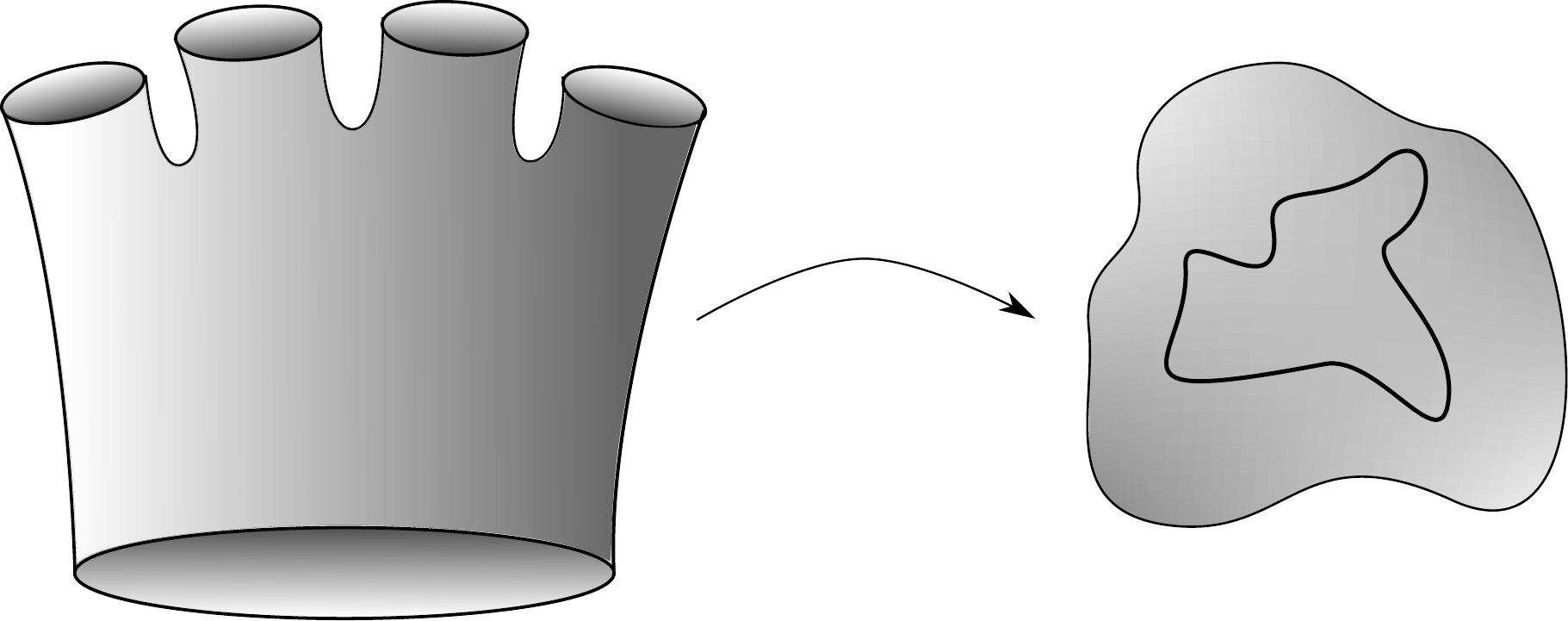}  
		\caption{A planar surface $\Sigma$ bounding a loop $\gamma$ representing the element $g$}
		\label{fig: admissible}
	\end{figure}
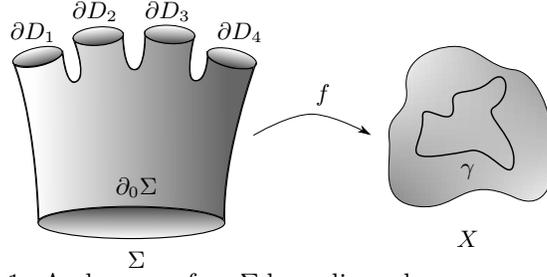
	
	Such a topological interpretation in terms of surfaces makes the problem of computing stable word length more structured since there are nice operations on surfaces: compression, cut-and-paste, and taking finite covers. However, unlike the case of stable commutator length, where an equation $g=[a_1,b_1]\cdots [a_k,b_k]$ represents a surface of genus $k$, here we are restricting our attention to \emph{planar} surfaces. This makes the problem harder since the operations above (e.g. taking finite covers) do not necessarily preserve the class of planar surfaces and so we are forced to use certain subclasses of operations.
	
	If $S$ is the set of torsion elements in $G$, this topological interpretation specializes to the case of stable torsion length of an element $g$, which we denote as $\stl_G(g)$,
	and refer to an $S$-admissible surface bounding $g^n$ as a \emph{torsion-admissible surface} for $g$ of degree $n$. 
	When $G$ is a free product $G=A*B$, we can take $X$ to be a wedge of spaces $X_A$ and $X_B$ with $\pi_1(X_A)=A$ and $\pi_1(X_B)=B$. 
	Then each element of $S$ is represented by a loop which is supported either in $X_A$ or in $X_B$. 
	Using this particular structure, we develop a normal form of torsion-admissible surfaces for a given element $g$ which is not conjugate into $A$ or $B$; see Section \ref{subsec: normal form}.
	
	To further simplify the problem, we introduce and focus on the family of \emph{simple surfaces}. These are surfaces $\Sigma$ made of particular pieces such that each is either a disk or an annulus which is supported either in $X_A$ or in $X_B$. The gluing of pieces are encoded by the \emph{gluing graph} $\Gamma_\Sigma$.
	Any surface in normal form can be simplified into a simple surface whose gluing graph is a tree. 
	
	
	Therefore, we obtain a lower bound of $\stl_G(g)$ by minimizing the complexity $\frac{-\chi(\Sigma)}{n}$ over all connected simple surfaces $\Sigma$ for $g$ with $\chi(\Gamma_\Sigma)=1$; see Lemma \ref{lemma: stl and simple surfaces}. This can be formulated as a linear programming problem when we further relax to the class of not necessarily connected simple surfaces with $\chi(\Gamma_\Sigma)\ge0$, which gives a way to compute a nontrivial
	lower bound of the stable torsion length; see Section \ref{subsec: linprog lower bound}.
	
	When torsion elements in $A$ (resp. $B$) form a subgroup, any simple surface whose gluing graph is a tree is itself a surface in normal form.
	Thus $\stl_G(g)$ is exactly the infimal complexity over all connected simple surfaces $\Sigma$ for $g$ with $\chi(\Gamma_\Sigma)=1$.
	
	This characterization leads to an isometric embedding theorem (Theorem \ref{thm: isometric embedding}), that allows us to compute the stable torsion length in simpler free products to obtain more general results. Here we state a special case:
	
	\begin{theoremA}[Isometric Embedding, weak version]\label{introthm: Isom emb}
	    Suppose that $i_A :A\to A'$ and $i_B: B\to B'$ are injective homomorphisms from finite groups $A$ and $B$. 
	    Then the induced map $i: A*B\to A'*B'$ preserves the stable torsion length, i.e.
		$$\stl_{A*B}(g)=\stl_{A'*B'}(i(g))$$
		for any $g\in A*B$.
	\end{theoremA}
	
		
	To exactly compute the stable torsion length by an algorithm, it is desirable to extend the family of connected simple surfaces $\Sigma$ with $\chi(\Gamma_\Sigma)=1$ to those with $\chi(\Gamma_\Sigma)\ge0$. This relaxation does not affect the computation when $A$ and $B$ are finite groups, as we are able to show that connected simple surfaces $\Sigma$ with $\chi(\Gamma_\Sigma)=0$ can be approximated by those with $\chi(\Gamma_\Sigma)=1$; see Lemma \ref{lemma: approximation}. This is achieved by considering two operations, splitting and rewiring, that we introduce in Section \ref{sec: compute stl}. 
	
	Using these two operations, we further show that any connected simple surface $\Sigma$ with $\chi(\Gamma_\Sigma)\ge0$ can be simplified into a union of \emph{irreducible} ones; see Section \ref{subsec: irreducible}. Moreover, there are only finitely many different irreducible simple surfaces (Proposition \ref{prop: finitely many irreducible}), which can be enumerated. As a result, the stable torsion length is actually equal to $\frac{-\chi(\Sigma)}{n}$ for some irreducible simple surface $\Sigma$, and thus must be a rational number. This yields the Rationality Theorem \ref{introthm: rationality}.
	
	Finally we carry out explicit computations in free products of cyclic groups
	and then use Theorem \ref{introthm: Isom emb} to generalize the formulas to arbitrary free products and prove 
	Theorems \ref{introthm: product formula} and \ref{introthm: compute abAB}.
    
	\subsection{Organization of the paper}
	In Section \ref{sec: background}, we provide some basic properties of stable torsion length and formulate the interpretation via torsion-admissible surfaces.
	
	We show crystallographic groups have trivial stable torsion length due to bounded generation in Section \ref{sec: crystallographic}.
	In Section \ref{sec: free prod}, we develop a normal form of torsion-admissible surfaces in a free product and introduce simple surfaces, 
	using which we prove the Isometric Embedding Theorem \ref{introthm: Isom emb} and give a lower bound estimation via linear programming.
	Then in Section \ref{sec: compute stl}, we specialize to free products of finite groups and introduce the operations of splitting and rewiring to 
	show the Rationality Theorem \ref{introthm: rationality}.
	Finally in Section \ref{sec: examples} we carry out computations in explicit examples and prove Theorems \ref{introthm: product formula} and \ref{introthm: compute abAB}.
	
	\subsection*{Acknowledgement}
	The authors would like to thank Benson Farb for suggesting this problem and the study of $\stl$, as well as for comments on the draft and invaluable support from start to finish. We would also like to thank Danny Calegari, Hannah Hoganson, and Kasia Jankiewicz for many helpful conversations.

	\section{General setup}\label{sec: background}
	In Section \ref{subsec: algebraic}, we give the foundational definitions and deduce fundamental properties of stable word length with respect to a conjugate-invariant set $S$ of a group $G$. 
	Throughout we assume $S$ to be symmetric in the sense that $s\in S$ if and only if $s^{-1}\in S$.
	Some properties are not used in this paper but could be of independent interest.
	In Section \ref{subsec: topological}, we give a topological formulation when $S$ is closed under taking powers, which is crucial in Sections \ref{sec: free prod}--\ref{sec: examples}.
	
	
	\subsection{The algebraic point of view}\label{subsec: algebraic}
	Let $G$ be a group and let $S$ be a (symmetric) conjugate-invariant subset. Let $\langle S\rangle$ be the (normal) subgroup of $G$ generated by $S$.
	When $S$ is the set of commutators, we have $\langle S\rangle=[G,G]$. When $S$ is the set of torsion elements, we denote $\langle S\rangle$ as $\Gt$, the subgroup generated by torsion elements.
	\begin{definition}
		For any element $g$ of $\langle S\rangle$, the word length $|g|_S$ is
		the minimal $k$ such that $g=s_1\cdots s_k$, where each $s_i\in S$.
		When $S$ is the set of torsion elements (resp. commutators), 
		we denote $|g|_S$ by $\tl_G(g)$ (resp. $\cl_G(g)$) and refer to it as the \emph{torsion length} (resp. \emph{commutator length}) of $g$.
	\end{definition}
	The sequence $|g^n|_S$ is subadditive in $n$, thus
	$$\lim_{n\to\infty} \frac{|g^n|_S}{n}=\inf \frac{|g^n|_S}{n},$$
	which is called the \emph{stable word length} of $g$, denoted $\|g\|_S$.
	When $S$ is the set of torsion elements (resp. commutators), we denote $\|g\|_S$ by $\stl_G(g)$ (resp. $\scl_G(g)$) and refer to it as
	the \emph{stable torsion length} (resp. \emph{stable commutator length});
	When the group $G$ is understood we simply denote it as $\stl(g)$ (resp. $\scl(g)$).
	
	The following properties are standard. They are well known in the case of stable commutator length (see \cite[Chapter 2]{Cal:sclbook}).
	
	\begin{lemma}[Monotonicity] \label{lemma: monotone} 
		Let $S\subset G$ and $T\subset H$ be conjugate-invariant subsets.
		Suppose $\varphi:G\to H$ is a group homomorphism such that $\varphi(S)\subset T$. 
		Then
		$$|\varphi(g)|_T \le |g|_S \quad and \quad \|\varphi(g)\|_T \le \|g\|_S$$
		for all $g\in \langle S\rangle$.
		In particular, for an arbitrary homomorphism $\varphi: G\to H$, we have
		$$\tl_H(\varphi(g))\leq \tl_G(g)\quad and \quad \stl_H(\varphi(g))\leq \stl_G(g)$$ for all $g\in \Gt$.
	\end{lemma}
	\begin{proof}
		If $g=s_1\cdots s_n$ for some $s_i\in S$ and $n\in\Z_+$, then $\varphi(g)=\varphi(s_1)\cdots\varphi(s_n)$ where each $\varphi(s_i)\in T$ by the assumption. Thus the inequality $|\varphi(g)|_T \le |g|_S$ easily follows, which implies the stable version. The assumption clearly holds when $S$ and $T$ are the set of torsion elements in $G$ and $H$.
	\end{proof}
	
	\begin{corollary}[Retraction]\label{cor: retract}
		Suppose that $\varphi:G\to H$ and $\psi:H\to G$ are group homomorphisms such that $\psi \circ \varphi: G\to G$ is the identity. Then, 
		$$\stl_H(\phi(g))=\stl_G(g)
		$$
		for all $g\in \Gt$.
	\end{corollary}
	\begin{proof}
		This follows immediately from Lemma \ref{lemma: monotone}. 
	\end{proof}
	
	\begin{lemma}[Characteristic]
		The functions $|\cdot|_S$ and $\|\cdot\|_S$ are constant on conjugate classes.
		If $S$ is also invariant under the action of $\Aut(G)$, then
		$|\cdot|_S$ and $\|\cdot\|_S$ are constant on orbits of $\Aut(G)$. 
	\end{lemma}
	\begin{proof}
		This follows from the definition.
	\end{proof}
	
	Thus, both $\tl_G$ and $\stl_G$ are constant on orbits of $\Aut(G)$. 
	
	\begin{lemma}[Countable subgroup] 
		Let $g$ be an element in $\Gt\le G$. There exists a countable subgroup $H<\Gt$ containing $g$ such that $\stl_H(g)=\stl_G (g)$. 
	\end{lemma}
	\begin{proof}
		For each $n\in \nn$, there exists $\tl(g^n)$ torsion elements whose product is $g^n$. Let $H_n$ be the group generated by those $\tl(g^n)$ torsion elements and let $H$ be the group generated by $\bigcup_n H_n$. Then $H$ is countable and $\tl_{H}(g^n)\leq\tl_{G}(g^n)$ since we have exhibited each $g^n$ as the product of $\tl_{G}(g^n)$ torsion elements in $H$. On the other hand, the inequality $\tl_{H}(g^n)\geq\tl_{G}(g^n)$ follows from Lemma \ref{lemma: monotone}.   
	\end{proof}
	
	A similar statement holds for general stable word length. 
	
	
	Finally, there is an inequality relating the stable commutator length and the stable torsion length. This can be found in \cite{Kotschick}, but we give a conceptually simpler proof using Bavard's duality. Since Bavard's duality and related notions are not used in the rest of the paper, we refer readers to \cite[Chapter 2]{Cal:sclbook}.
	
	\begin{lemma}[Kotschick \cite{Kotschick}] \label{lemma: scl inequality}
		Let $G$ be a group. For any $g\in \Gt \cap [G,G]$, we have
		$$2 \scl(g) \leq \stl(g).$$
	\end{lemma}
	\begin{proof}
		Let $\varphi:G\to \rr$ be a homogeneous quasimorphism. If the defect $D(\varphi)=0$ for every $\varphi$, then $\scl(g)=0$, and the inequality holds trivially. 
		Now assume that $D(\varphi)>0$. If $g=t_1\cdots t_{\tl(g)}$ for some torsion elements $t_i$, then 
		we have 
		$$\left|\varphi(g)-\sum_{i=1}^{\tl(g)}\varphi(t_i)\right|\leq (\tl(g)-1)D(\varphi).$$
		Since each $t_i$ is torsion and $\varphi$ is homogeneous, we know $\varphi(t_i)=0$ for all $i$, and thus
		$$\frac{|\varphi(g)|}{D(\varphi)}\leq \tl(g)-1.$$
		Using that $\varphi$ is homogeneous, applying this to $g^n$ for any $n\in \nn$, we have 
		$$\frac{|\varphi(g)|}{D(\varphi)}\leq \frac{\tl(g^n)-1}{n}.	$$
		The Bavard's duality states that the supremum of the left-hand side over all homogeneous quasimorphisms is $2\scl(g)$.
		As the limit of the right-hand side is $\stl(g)$, this gives the desired inequality.
	\end{proof}
	
	\begin{remark} \label{remark: hyperbolic}
	A theorem of Epstein--Fujiwara \cite{hypbddcohom} implies that any non-elementary hyperbolic group $G$ has an infinite dimensional space of homogeneous quasimorphisms, which implies that $\scl_G(g)>0$ for some $g$. Therefore, by Lemma \ref{lemma: scl inequality}, if $G$ is a non-elementary hyperbolic group generated by torsion, then $\stl_G(g)>0$ for some $g$, and (stable) torsion length is unbounded in $G$.
	
	\end{remark}
	
	
	\subsection{A topological point of view}\label{subsec: topological}
	Fix the group $G$ and conjugate-invariant subset $S$, which we will assume to be closed under taking powers in Lemma \ref{lemma: stl via surfaces} below.
	Let $X$ be a space with fundamental group $G$ and let $\gamma$ be a loop representing $g$.
	\begin{definition}
		Let $\Sigma$ be a compact, oriented, connected \emph{planar} (i.e. genus zero) surface and let $f:\Sigma\to X$. We say that $(\Sigma,f)$ is \emph{$S$-admissible}
		for $g$ of degree $n(\Sigma,f)$ if $\Sigma$ has a specified boundary component $\partial_0 \Sigma$ such that
		$f:\Sigma\to X$ takes $\partial_0 \Sigma$ to $\gamma$, winding around $n(\Sigma,f)$ times, and the image of all other boundary components are loops representing  conjugacy classes in $S$ (see Figure \ref{fig: admissible}).
		
		When $S$ is the set of torsion elements in $G$, we refer to $(\Sigma,f)$ as a \emph{torsion-admissible surface} instead. We often denote a torsion-admissible surface by $\Sigma$ instead of $(\Sigma,f)$ to make $f$ implicit.
	\end{definition}
	
	We refer to the boundary components of $\Sigma$ other than $\partial_0 \Sigma$ as \emph{holes}. Denote by $H(\Sigma)$ the number of holes on $\Sigma$. Note that $-\chi(\Sigma)=H(\Sigma)-1$.
	For a connected surface $\Sigma$, let $\chi^-(\Sigma)$ be $\chi(\Sigma)$ unless $\Sigma$ is a sphere or a disk, in which case we define $\chi^-(\Sigma)$ to be $0$.
	
	\begin{lemma}\label{lemma: stl via surfaces}
		Suppose $s^n\in S$ for any $s\in S$ and $n\in\Z$. For any $g\in\langle S\rangle$ we have
		\begin{equation}\label{eqn: stl via surfaces}
			\|g\|_S=\inf_\Sigma \frac{H(\Sigma)}{n(\Sigma)}=\inf_\Sigma \frac{-\chi^-(\Sigma)}{n(\Sigma)},
		\end{equation}
		where each infimum is taken over all $S$-admissible surfaces $\Sigma$.
	\end{lemma}
	\begin{proof}
		The first equality is simply the topological reformulation of the algebraic definition. 
		For the second, note that if $g$ is torsion then $g^n$ bounds a disk for some $n$ and all three quantities are zero in this case.
		Suppose $g$ is not a torsion element. Then for any $S$-admissible surface $\Sigma$ we have $-\chi^-(\Sigma)=-\chi(\Sigma)$.
		In particular, $-\chi^-(\Sigma)=H(\Sigma)-1\le H(\Sigma)$.
		Thus it suffices to show that
		$$\inf_\Sigma \frac{H(\Sigma)}{n(\Sigma)}\le\inf_\Sigma \frac{-\chi^-(\Sigma)}{n(\Sigma)}.$$
		Note that given any $S$-admissible surface $\Sigma$ and any $N\in \Z_+$, 
		$\Sigma$ has a degree $N$ cover $\Sigma_N$ with genus zero such that the preimage of $\partial_0\Sigma$
		is a single boundary component; see Figure \ref{fig: planar cover}. 
		Note that any other boundary component of $\Sigma_N$ covers a boundary component of $\Sigma$ different from $\partial_0\Sigma$. Since $S$ is closed under taking powers, $\Sigma_N$ is $S$-admissible of degree $n(\Sigma_N)=N\cdot n(\Sigma)$.
		Thus 
		$$\inf_{\Sigma'} \frac{H(\Sigma')}{n(\Sigma')}\le \frac{H(\Sigma_N)}{n(\Sigma_N)}=\frac{-\chi^-(\Sigma_N)+1}{n(\Sigma_N)}=\frac{-N\cdot\chi^-(\Sigma)+1}{N\cdot n(\Sigma)}.$$
		Taking $N\to\infty$ proves the desired inequality.
	\end{proof}	
	
	\begin{figure}
		\centering
		\labellist
		\small \hair 2pt
		\pinlabel $\Sigma$ at 100 320
		\pinlabel $\Sigma_N$ at 350 320
		\pinlabel $\partial_0\Sigma$ at 100 -20
		\endlabellist
		\includegraphics[scale=.3]{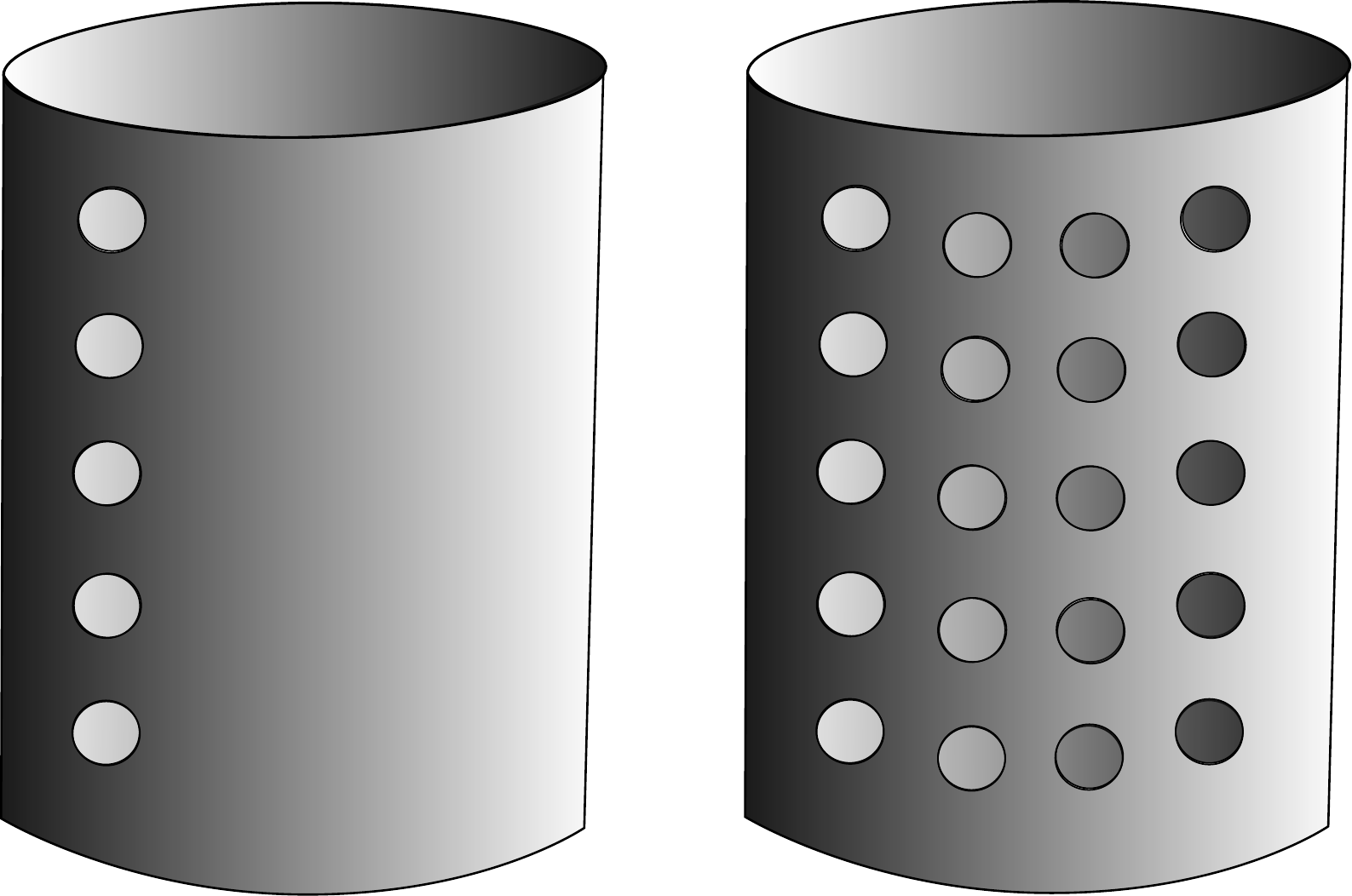}
		\caption{On the left is an $S$-admissible surface $\Sigma$ and on the right is a planar degree $N=4$ cover $\Sigma_N$ of $\Sigma$ 
	    such that the preimage of $\partial_0 \Sigma$ is a single boundary component}
		\label{fig: planar cover}
	\end{figure}
	\begin{remark}
		We can replace $-\chi^-(\Sigma)$ by $-\chi(\Sigma)$ for any $S$-admissible surface $\Sigma$ whenever $g$ is not torsion .
	\end{remark}
	
	\section{Crystallographic groups}\label{sec: crystallographic}
	Let $E(n)$ be the isometry group of the Euclidean space $\R^n$. 
	Each element $\gamma\in E(n)$ acts on $\R^n$ via $\gamma(x)=Ax+v$ for some uniquely determined orthogonal transformation $A\in O(n)$ and vector $v\in\R^n$. 
	We refer to $A$ as the \emph{rotational part} of $\gamma$ and $v$ as the \emph{translational part} of $\gamma$.
	
	A \emph{crystallographic group} $\Gamma$ of dimension $n$ is a cocompact discrete subgroup of $E(n)$. 
	By a theorem of Bieberbach \cite{Bieberbach1912} (see also \cite{Crystbook})
	for any such $\Gamma$, the subgroup $H$ acting by translations is a normal subgroup of $\Gamma$ of finite index and isomorphic to $\Z^n$.
	So we have an exact sequence
	$$1\longrightarrow H \longrightarrow \Gamma \longrightarrow G \longrightarrow 1,$$
	where $G$ is a finite subgroup of $O(n)$, and the map $\Gamma\to G$ takes the rotational part $A\in G$ of any $\gamma\in\Gamma$.
	
	Let $\Gat$ be the subgroup of $\Gamma$ generated by torsion elements.
	We show the following more precise version of Theorem \ref{introthm: vanishing}.
	\begin{theorem}\label{thm: vanishing}
		For any crystallographic group $\Gamma$, the torsion subgroup $\Gat$ is boundedly generated by torsion elements, and thus $\stl_\Gamma\equiv0$.
	\end{theorem}
	
	Let $\Ht\defeq H\cap \Gat$ and let $\Gt$ be the image of $\Gat$ in $G\le O(n)$. Then $\Ht$ is a free abelian subgroup of $H$ and is finite index in $\Gat$. We have
	$$1\longrightarrow \Ht \longrightarrow \Gat \longrightarrow \Gt \longrightarrow 1.$$
	
	Here we think of $H$ both as the translation subgroup of $\Gamma$ and as a lattice in $\R^n$, where each vector $h$ of the lattice corresponds to the translation $T_h:x\mapsto x+h$. To avoid confusion, we use $h$ to represent an element of $H$ when we regard $H$ as a lattice and use $T_h$ when we regard $H$ as the translation subgroup of $\Gamma$. Note that $T_{nh}=T_h^n$ for any $n\in\Z$ and $T_{h+h'}=T_h\cdot T_{h'}$ for all $h,h'\in H$.
	
	We prove Theorem \ref{thm: vanishing} by constructing a finite index subgroup $H_0$ of $\Ht$ that is boundedly generated by torsion in $\Gat$.
	We start by finding elements in $\Ht$ that can be written as a product of few torsion elements.
	
	\begin{lemma}\label{lemma: commutator with torsion}
		Suppose $\gamma\in\Gat$ is a torsion element with rotational part $A\in \Gt$ and let $I$ be the identity element of $O(n)$. Then for any $h\in H$, we have $(A-I)h\in \Ht$, and the corresponding translation $T_{(A-I)h}=[\gamma,T_h]$ is a product of two torsion elements.
	\end{lemma}
	\begin{proof}
		Since $\gamma$ is a torsion element, it must fix some point $p\in\R^n$ and acts by
		$\gamma(x)=A(x-p)+p$ for any $x\in\R^n$. Hence
		$$\gamma T_h \gamma^{-1} (x)= A[A^{-1}(x-p) +p+h-p]+p=x+Ah,$$
		so $[\gamma, T_h](x)=x+Ah-h=x+(A-I)h=T_{(A-I)h}(x)$.
		This shows that $T_{(A-I)h}=[\gamma,T_h]$, which is a product of two torsion elements,
		as $[\gamma,T_h]=\gamma\cdot (T_h\gamma^{-1} T_h^{-1})$. 
		It also shows that $T_{(A-I)h}=[\gamma,T_h]=(\gamma T_h \gamma^{-1})\cdot T_h^{-1}$ is an element of $H$.
		Thus $(A-I)h\in H\cap\Gat=\Ht$.
	\end{proof}
	
	Iterating the previous lemma, we control the torsion length for a larger family of elements in $\Ht$.
	
	\begin{lemma}\label{lemma: iterated commutator}
		For any $m\ge1$ and $1\le i\le m$, let $\gamma_i\in \Gat$ be a torsion element with rotational part $A_i\in \Gt$. Then for any $h\in H$, we have $(A_1\cdots A_m-I)h\in \Ht$, and the corresponding translation $T_{(A_1\cdots A_m-I)h}$ is a product of $4m-2$ torsion elements.
	\end{lemma}
	\begin{proof}
		We prove this by induction on $m$. The base case $m=1$ follows from Lemma \ref{lemma: commutator with torsion}. Suppose this holds for $m-1$. Then
		\begin{align*}
			&[\gamma_1,T_{(A_2\cdots A_m-I)h}]\cdot[\gamma_1,T_h]\cdot T_{(A_2\cdots A_m-I)h}\\
			=&T_{(A_1-I)(A_2\cdots A_m-I)h}\cdot T_{(A_1-I)h}\cdot T_{(A_2\cdots A_m-I)h}\\
			=&T_{(A_1\cdots A_m - I)h}
		\end{align*}
		by Lemma \ref{lemma: commutator with torsion}. Note that the first row is the product of $2+2+(4m-6)=4m-2$ torsion elements by Lemma \ref{lemma: commutator with torsion} and the induction hypothesis.
	\end{proof}
	
	As a consequence, we can uniformly bound the torsion length of all elements in $\Ht$ of the form $(A-I)h$ for any $h\in H$ and any $A\in \Gt$.
	\begin{lemma}\label{lemma: uniform bound}
		With the notation above, there is some $M$ such that for any $h\in H$ and any $A\in \Gt$ we have $(A-I)h\in \Ht$, and the corresponding translation $T_{(A-I)h}$ is a product of at most $M$ torsion elements in $\Gamma$.
	\end{lemma}
	\begin{proof}
		For each $A\in \Gt$, pick an arbitrary lift $\gamma\in\Gat$. Since $\Gt$ is finite, there is some $m$ such that each $\gamma$ can be written as a product of at most $m$ torsion elements. By Lemma \ref{lemma: iterated commutator}, the conclusion holds with $M=4m-2$.
	\end{proof}
	
	We will need the following lemma to ensure that we can pick elements of the form $(A-I)h$ considered above to generate a finite index subgroup $H_0$ in $\Ht$.
	\begin{lemma}\label{lemma: full image}
		With the notation above, let $X\subset \R^n$ be the subspace spanned by the image of $A-I$ for all $A\in \Gt$. 
		Consider $\Ht\le H\le \R^n$ as a discrete subgroup of $\R^n$. Then $X$ is also the $\R$-linear span of $\Ht$.
	\end{lemma}
	\begin{proof}
		By Lemma \ref{lemma: commutator with torsion}, we have $(A-I)h\in \Ht$ for any $h\in H$ and $A\in \Gt$. 
		This shows that the $\R$-linear span of $\Ht$ contains $X$ since $H$ spans the entire space $\R^n$. It remains to show that $\Ht\subset X$.
		
		Any torsion element $\gamma\in \Gamma$ acts on $\R^n$ by $$\gamma(x)=A(x-p)+p=Ax+(I-A)p.$$ 
		The translational part $(I-A)p$ lies in $X$ by definition. We show that this holds for all $\gamma\in \Gat$.
		Since $\Gat$ is generated by torsion, by induction, it suffices to show that
		$\eta\gamma$ has translational part in $X$ if both $\eta,\gamma\in \Gat$ do.
		Indeed, if $\gamma(x)=Ax+u$ and $\eta(x)=Bx+v$ with $A,B\in \Gt$ and $u,v\in X$, then 
		$$\eta\gamma(x)=B(Ax+u)+v=BAx+(B-I)u+u+v$$ has translational part $(B-I)u+u+v\in X$
		since all three terms lie in $X$. Thus 
		any $\gamma\in\Gat$ can be written as $\gamma(x)=Ax+u$ for some $u\in X$.
		In particular, any translation in $\Gat$ takes the form $T_u$ for some $u\in X$.
		This shows $\Ht=H\cap \Gat\subset X$.
	\end{proof}
	
	Now we are in a place to prove Theorem \ref{thm: vanishing}.
	\begin{proof}[Proof of Theorem \ref{thm: vanishing}]
		We use the notation above. Let $d$ be the dimension of the space $X$ as in Lemma \ref{lemma: full image}. 
		Since $H$ spans $\R^n$, by Lemma \ref{lemma: full image} there exists $h_i\in H$ and $A_i\in \Gt$ 
		for $1\le i\le d$ such that $\{(A_i-I)h_i\}_{i=1}^d$ is a basis of $X$. 
		By Lemma \ref{lemma: commutator with torsion}, the subgroup $H_0$ generated by $\{(A_i-I)h_i\}_{i=1}^d$ is a subgroup of $\Ht\le \R^n$, and by construction its $\R$-linear span is $X$.
		As the $\R$-linear span of $\Ht$ is also equal to $X$ by Lemma \ref{lemma: full image}, we observe that $H_0$ is finite index in $\Ht$.
		
		Applying Lemma \ref{lemma: uniform bound} to $A_i\in \Gt$ and $kh_i\in \Ht$ for any $k\in\Z$ and $1\le i\le d$,
		we know there is some uniform $M$ such that $(A_i-I)kh_i$ lies in $\Ht$ and the corresponding translation $T_{(A_i-I)kh_i}$ is a product of at most $M$ torsion elements. 
		By the definition of $H_0$, any element $h\in H_0$ can be written as $\sum_{i=1}^d k_i(A_i-I)h_i$ for some $k_i\in\Z$, and thus the corresponding translation 
		$$T_h=\prod_{i=1}^d T_{k_i(A_i-I)h_i}$$ 
		is a product of at most $dM$ torsion elements.
		
		Since $H_0$ is finite index in $\Ht$, it is also finite index in $\Gat$. By fixing coset representatives of $H_0\le \Gat$ and expressing them as products of torsion elements, the result follows from bounded generation of $H_0$ that we showed above.
	\end{proof}
	
	\begin{example}

		Consider the $(3,3,3)$-triangle group
		$$\Gamma \coloneqq \la a,b,c\;|\; a^2=b^2=c^2=(ab)^3=(bc)^3=(ca)^3=1\ra.$$
		$\Gamma$ acts properly discontinuously and cocompactly on the Euclidean plane with fundamental domain an equilateral triangle $T$, where the generators $a,b$, and $c$ act by reflections about the three lines $\ell_1$, $\ell_2$, and $\ell_3$ containing the three sides of $T$ respectively. This realizes $\Gamma$ as a cocompact discrete subgroup of $E(2)$.

		By Theorem \ref{thm: vanishing}, $\Gamma$ is boundedly generated by torsion elements. Here, we explicitly show that any $\gamma\in\Gamma$ is a product of at most four torsion elements.
		
		The orbit of $T$ under the $\Gamma$ actions gives a tiling of $\R^2$ as in Figure \ref{fig: triangle}.
		As $T$ is a fundamental domain, elements of $\Gamma$ are in one-to-one correspondence to the image of $T$. All lines in the tiling are the image of $\ell_1$, $\ell_2$, or $\ell_3$ under $\Gamma$. Thus any reflection fixing one of these lines is a conjugate of $a$, $b$, or $c$. and hence an element of $\Gamma$. So it suffices to show that one can take $T$ to any other triangle in the tiling using at most four such reflections.
		
		We can take $T$ to any light shaded triangle in Figure \ref{fig: triangle} by at most two reflections about horizontal lines in the tiling. Now we can arrive at any dark shaded triangle by further applying a reflection about a line in the tiling parallel to $\ell_2$. Note that any remaining triangle shares a side with a dark shaded triangle, so we can arrive at any remaining triangle by another reflection. Hence we can reach any triangle in the tiling by at most four reflections.
		
		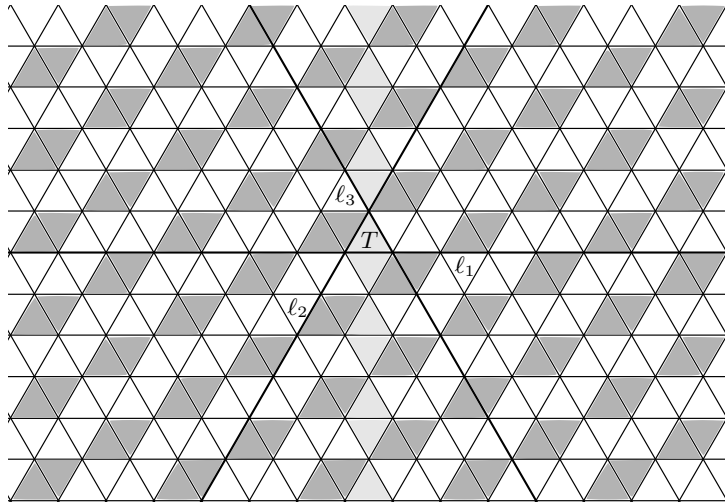
\begin{figure}
		\centering
		\labellist
		\small \hair 2pt
		\pinlabel $T$ at 227 165
		\pinlabel $\ell_1$ at 287 148
		\pinlabel $\ell_2$ at 182 122
		\pinlabel $\ell_3$ at 212 193
		\endlabellist
		\includegraphics[scale=0.6]{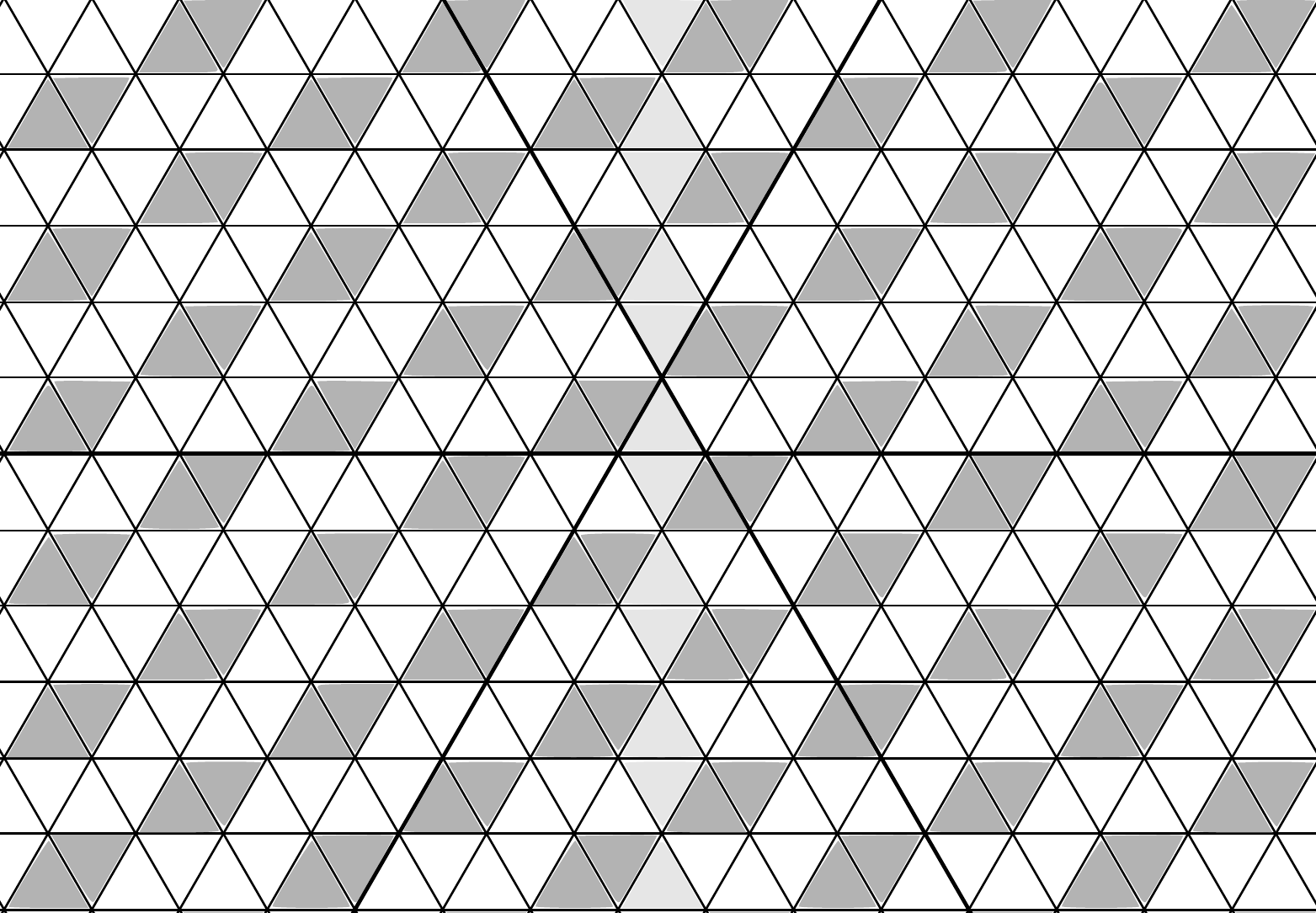}
		\caption{The tiling given by the action of the triangle group.}
		\label{fig: triangle}
	\end{figure}

	\end{example}
	
	\begin{remark}
		For any integers $p,q,r\ge2$, the group $\Gamma_{p,q,r}=\langle a,b,c \;|\; a^2=b^2=c^2=(ab)^p=(bc)^q=(ac)^r=1\ra$ is called the $(p,q,r)$-triangle group. If $\frac{1}{p}+\frac{1}{q}+\frac{1}{r}>1$, then $\Gamma_{p,q,r}$ is finite, so $\stl _{\Gamma_{p,q,r}}\equiv 0$. If $\frac{1}{p}+\frac{1}{q}+\frac{1}{r}=1$, then $\Gamma_{p,q,r}$ is crystallographic, so by Theorem \ref{thm: vanishing}, $\stl_{\Gamma_{p,q,r}}\equiv 0$. 
		
		Finally, if $\frac{1}{p}+\frac{1}{q}+\frac{1}{r}<1$, then $\Gamma_{p,q,r}$ is non-elementary hyperbolic. By remark \ref{remark: hyperbolic}, $\stl_{\Gamma_{p,q,r}}\not\equiv 0$. 
	\end{remark}

	\section{Free products}\label{sec: free prod}
	In the rest of this paper, we focus on stable torsion length in free products.
	
	Let $G=A*B$ be a free product of groups $A$ and $B$. Let $X_A$ be a $K(A,1)$ space, let $X_B$ be a $K(B,1)$ space, and let $X$ be the space obtained by connecting $X_A,X_B$ by a line segment with midpoint $*$. In the sequel, we will really think of $X_A$ as including the half segment up to $*$, and similarly for $X_B$.
	
	We develop a normal form for torsion-admissible surfaces (defined in Section \ref{subsec: topological}) in $X$ in Section \ref{subsec: normal form}.
	The normal form can be further simplified to \emph{simple surfaces} which we introduce in Section \ref{subsec: simple surf}. Describing the stable torsion length in terms of simple surfaces leads to an isometric embedding theorem (Section \ref{subsec: isom embed}) and a linear programming problem which produces an effective lower bound of $\stl_G(g)$ for any element $g\in G$ (Section \ref{subsec: linprog lower bound}).
	Specializing to the case where $A$ and $B$ are finite groups, we will further develop an algorithm that computes $\stl_G(g)$ for any $g$ in Section \ref{sec: compute stl}.
	
	\subsection{A normal form}\label{subsec: normal form}
	Let $g\in A*B$ be an element which does not conjugate into $A$ or $B$. Stable word length is constant on conjugacy classes, so it suffices to consider $g$ as a cyclically reduced word $g=a_1b_1\cdots a_L b_L$ where $a_i\in A\setminus\{id\}$ and $b_i\in B\setminus\{id\}$.
	
	Let $\gamma$ be a loop in $X$ representing $g$ such that $*$ decomposes it as a concatenation of 
	$2L$ arcs $\alpha_1,\beta_1,\alpha_2,\beta_2,\ldots,\alpha_L,\beta_L$ cyclically, 
	where each $\alpha_i$ (resp. $\beta_i$) is supported on
	the $A$-side (resp. $B$-side) and represents $a_i\in A$ (resp. $b_i\in B$) as a loop based at $*$.
	
	Let $f:S\to X$ be a torsion-admissible surface (see Section \ref{subsec: topological}). 
	Recall that by definition there is a specified boundary component $\partial_0 S$ of $S$ whose image represents a power of $g$, and the remaining boundary components are referred to as holes. Since the image of each hole represents a torsion element in $G=A*B$, which must be conjugate to a torsion element in either $A$ or $B$, up to homotopy we may assume the image of each hole is disjoint from $*$.
	Perturb $f$ further by a homotopy to make it transverse\footnote{Here transversality makes sense since $*$ is the midpoint of the segment joining $X_A$ and $X_B$, locally as a submanifold of codimension one.} to $*$ and keep the image of holes disjoint from $*$. Then $F\defeq f^{-1}(*)$ is an embedded proper submanifold of $S$ of codimension one. Thus $F$ is a finite collection of disjoint embedded loops and proper arcs. 
	Moreover, the endpoints of any proper arc in $F$ lie on $\partial_0 S$ since all other boundary components are disjoint from $F$.
	
	\begin{lemma}
		Up to homotopy and compression of $S$ into another	torsion-admissible surface $S'$ of the same degree such that $-\chi^-(S')\le -\chi^-(S)$ and $H(S')\le H(S)$, we can assume that $F=f^{-1}(*)$ only consists of proper arcs.
	\end{lemma}
	\begin{proof}
		If $F$ contains any embedded loop $\tau$ that is essential (possibly boundary parallel) in $S$, then $\tau$ must cut $S$ into two components since $S$ is planar. Since $\tau\subset F$ is disjoint from $\partial_0 S$, only one of the components contains $\partial_0 S$ after cutting.
		Let $S'$ be this component with a disk coning off $\tau$, and extend $f$ by mapping the entire disk to $*$. 
		Further compose $f$ with a homotopy which pushes the image of the disk away from $*$. 
		In this way we obtain a simple torsion-admissible surface and eliminate an embedded essential loop $\tau\subset F$ without changing the map on $\partial_0 S$ (and thus the degree). This deletes an arbitrary embedded essential loop in $F$.
		
		Now suppose $F$ contains any inessential embedded loop $\rho$, i.e. $\rho$ bounds a disk in $S$. Take the inner-most disk $D$ among those bounding such loops. Up to a homotopy, we can modify $f$ on a small neighborhood of $D$ to eliminate an inessential embedded loop $\rho=\partial D\subset F$. 
		
		By applying the above operations finitely many times, we obtain a torsion-admissible surface $S'$ with the desired properties.
	\end{proof}
	
	From now on, assume $F=f^{-1}(*)$ only consists of proper arcs. Denote $S_A\defeq f^{-1}(X_A)$ and $S_B\defeq f^{-1}(X_B)$.
	Then $F$ cuts $S$ into $S_A$ and $S_B$, which are collections of subsurfaces with corners, and map into $X_A$ and $X_B$ respectively. See the example below.
	
	\begin{example}
		Suppose $g=a_1b_1a_2b_2$ where $a_1$ is a product of two torsion elements, $a_2$ is $2$-torsion, and $b_2=b_1^{-1}$. Then a torsion-admissible surface $S$ can be constructed as shown in Figure \ref{fig: S_A},
		where the red (outer) boundary component represents $g^4$ and each blue (inner) boundary represents a torsion element (in $A$).
		Then $F$ is the disjoint union of the (green) dashed arcs, and the subsurface $S_A$ is the union of those pieces in darker grey. 
	\end{example}

	\begin{figure}
		\centering
		\labellist
		\small \hair 2pt
		\pinlabel $\alpha_1$ at -7 101
		\pinlabel $\beta_2$ at 65 125
		\pinlabel $\beta_1$ at 65 72
		\pinlabel $\alpha_2$ at 110 125
		\pinlabel $\alpha_2$ at 110 72
		\pinlabel $\beta_1$ at 155 125
		\pinlabel $\beta_2$ at 155 72
		\pinlabel $\alpha_1$ at 188 129
		\pinlabel $\alpha_1$ at 188 70
		\pinlabel $\alpha_1$ at 240 101
		\pinlabel $\beta_1$ at 255 133
		\pinlabel $\beta_2$ at 206 158
		\pinlabel $\alpha_2$ at 265 203
		\pinlabel $\beta_2$ at 255 69
		\pinlabel $\beta_1$ at 208 44
		\pinlabel $\alpha_2$ at 265 -4
		\pinlabel $S$ at 120 0
		
		\pinlabel $\alpha_1$ at 308 101
		\pinlabel $\alpha_2$ at 395 125
		\pinlabel $\alpha_2$ at 395 72
		\pinlabel $\alpha_1$ at 442 129
		\pinlabel $\alpha_1$ at 442 70
		\pinlabel $\alpha_1$ at 495 101
		\pinlabel $\alpha_2$ at 510 186
		\pinlabel $\alpha_2$ at 510 12
		\pinlabel $S_A$ at 410 0
		\endlabellist
		
		\includegraphics[scale=0.7]{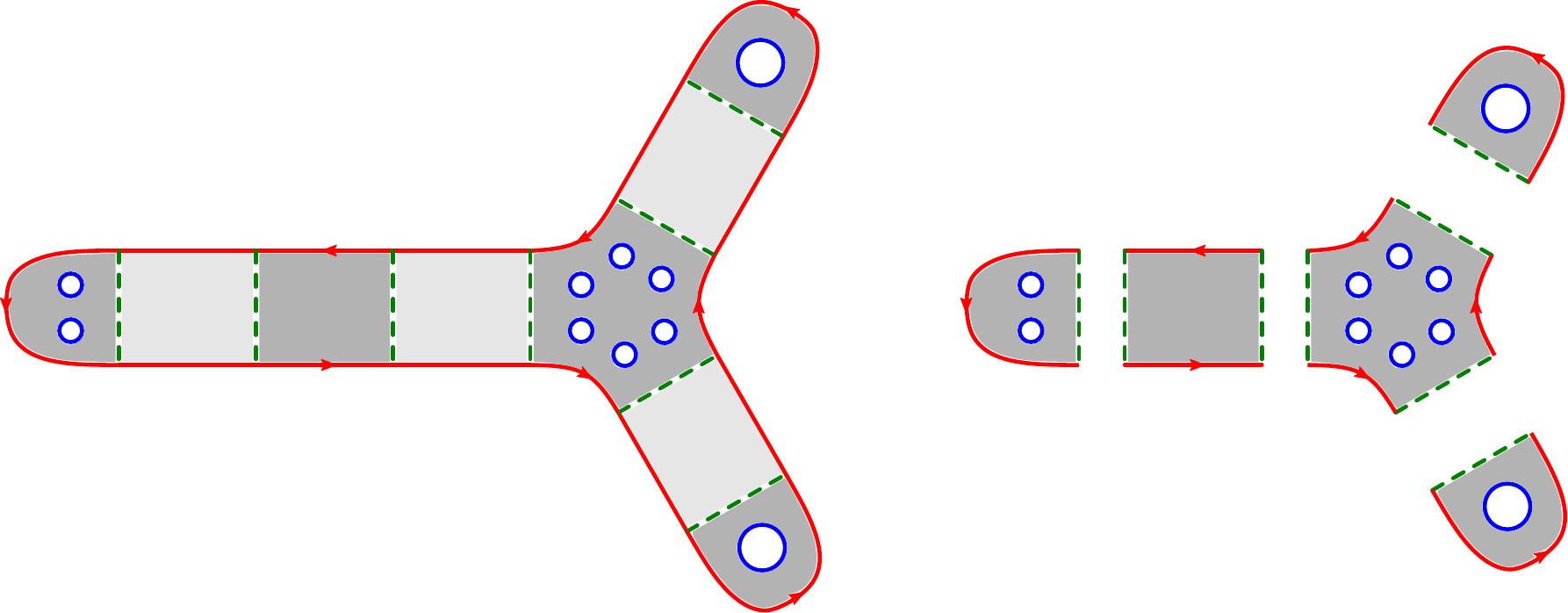}
		\caption{The decomposition of a torsion-admissible surface $S$ by cutting along $F$ (the green dashed arcs). The subsurface $S_A$ is pictured on the right.}
		\label{fig: S_A}
	\end{figure}
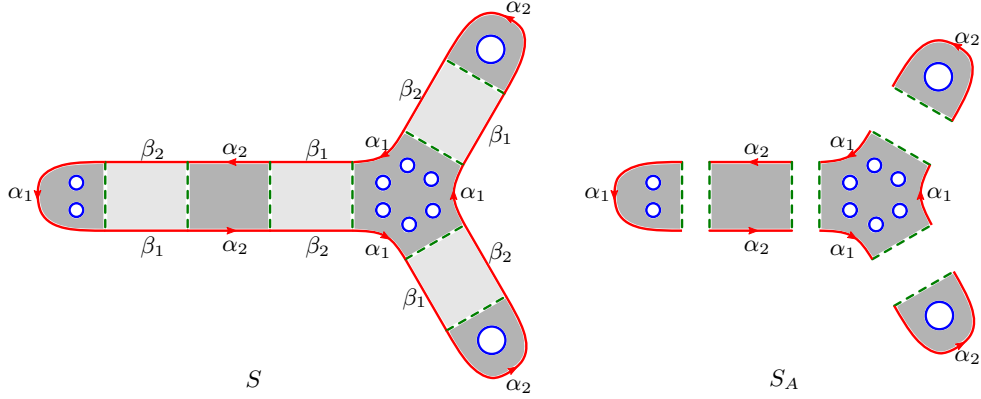


	In general, there are two types of boundary components of $S_A$:
	\begin{enumerate}
		\item A \emph{polygonal boundary} is one that contains corners, arcs in $F$, and arcs in $\partial_0S$. Such a boundary is divided into an even number of sides by corners of $S_A$, where the sides alternate between arcs on $F$ and arcs on $\partial_0 S$ which are mapped to some $\alpha_i$; see the ``outer'' boundary component in Figure \ref{fig: S_A} of each component in $S_A$.
		\item A \emph{hole} is a boundary component that is disjoint from $F$. Then by construction, it must come from a hole of $S$ and represent a torsion element in $A$; see the blue (inner) boundary components of $S_A$ in Figure \ref{fig: S_A}.
	\end{enumerate}
	
	\begin{lemma}
		With the above setup, each component of $S_A$ has exactly one polygonal boundary.
	\end{lemma}
	\begin{proof}
		A component of $S_A$ without any polygonal boundary must be itself a component of $S$ with only holes on the boundary, which is absurd since $S$ is connected. If a component of $S_A$ has at least two polygonal boundaries $C_1$ and $C_2$, then an embedded loop $\ell$ in $S_A$ homotopic to $C_1$ is non-separating in $S$: one can go from one side of $\ell$ to $C_1$, follow $\partial_0 S$ to arrive at $C_2$, and then travel to the other side of $\ell$ in this component; see the arc $\tau$ in Figure \ref{fig: nonsep}.
		This contradicts the fact that $S$ is planar.
		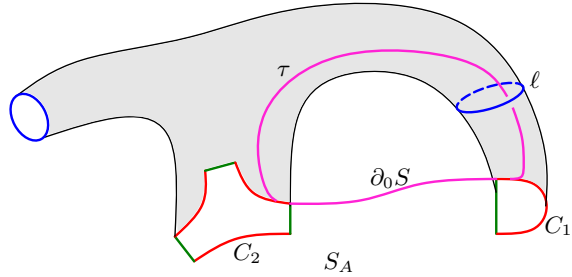
\begin{figure}
			\centering
			\labellist
			\small \hair 2pt
			\pinlabel $\tau$ at 150 105
			\pinlabel $\ell$ at 287 100
			\pinlabel $\partial_0S$ at 208 46
			\pinlabel $C_1$ at 300 20
			\pinlabel $C_2$ at 130 5    
			\pinlabel $S_A$ at 180 0    	
			\endlabellist
			\includegraphics[scale=0.7]{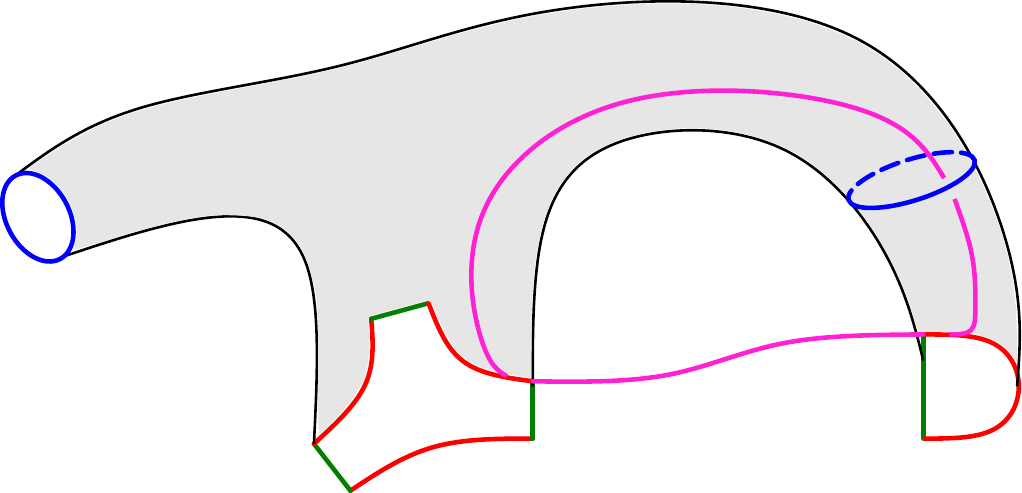}
			\caption{An arc $\tau$ traveling from one side of $\ell$ to the other side using part of the boundary $\partial_0 S$.}
			\label{fig: nonsep}
		\end{figure}
	\end{proof}
	
	Thus every component of $S_A$ is a polygon with $h\geq 0$ holes and a unique polygonal boundary. The same analysis works for components of $S_B$.
	
	\begin{definition}
		A \emph{normal form} of a torsion-admissible surface $S$ is the decomposition of $S$ into $S_A$ and $S_B$ as above, where each component of $S_A$ (resp. $S_B$) is a polygon with $h\ge0$ holes and a unique polygonal boundary.
	\end{definition}
	
	We summarize the discussion above in the following lemma and corollary. 
	\begin{lemma}\label{lemma: put into normal form}
		Any torsion-admissible surface $S$ can be modified into another torsion-admissible surface $S'$ in normal form of the same degree, such that $-\chi(S)\ge-\chi(S')$ and $H(S)\ge H(S')$.
	\end{lemma}
	
	\begin{corollary}
		In equation (\ref{eqn: stl via surfaces}) of Lemma \ref{lemma: stl via surfaces} 
		we can take each infimum over torsion-admissible surfaces in normal form instead.
	\end{corollary}

	\subsection{Simple Surfaces}\label{subsec: simple surf}
	
	The collection of torsion-admissible surfaces for an element $g\in G$ can be further simplified to the collection of simple surfaces, which we now introduce. We use simple surfaces to obtain effective estimates of $\stl_G(g)$ in Section \ref{subsec: linprog lower bound}.
	We push this further in Section \ref{sec: compute stl} when $G$ is a free product of finite groups to compute $\stl_G(g)$.
	
	Roughly speaking, a simple surface $S$ is made of particular pieces either in the $A$-side, or the $B$-side. This is similar to a torsion-admissible surface in normal form, however, one main difference is that each piece now only contains at most one hole.
	Before introducing simple surfaces, we first define the collection of pieces that are allowed in simple surfaces.
	
	Recall that the loop $\gamma$ representing $g$ decomposes into arcs $\alpha_1,\beta_1,\ldots,\alpha_L,\beta_L$,
	representing elements $a_1,b_1,\cdots,a_L,b_L$ in $A$ and $B$.
	
	\begin{definition}
		A \emph{polygonal boundary} is an oriented circle together with a map $f$ into $X_A$, so that the map $f$ naturally divides the loop into an even number of sides alternating between \emph{arcs} and \emph{turns} as follows. The map $f$ collapses each turn to the wedge point $*$ and maps each arc to some $\alpha_i$, called the the \emph{label} of the arc.
		See each ``outer'' boundary component of $S_A$ in Figure \ref{fig: S_A} or \ref{fig: simple_S_A} for examples, where arcs are solid in red and turns are dashed in green.
		
		As a loop in $X_A$, the polygonal boundary represents a conjugacy class in $A$, referred to as the \emph{winding class} of the polygonal boundary. 
		
		If a polygonal boundary has trivial winding class, the map extends to a disk bounding the polygonal boundary. We call such a disk with polygonal boundary a \emph{disk-piece}.
		
		If a polygonal boundary has nontrivial winding class, we require it to lie in $A_{tor}$. We represent such a conjugacy class by a loop in $X_A$ which is away from $*$ and homotopic to the polygonal boundary. Then, there is an annulus bounding the polygonal boundary on one side and this homotopic loop (with opposite induced orientation) on the other side.
        We refer to this annulus as an \emph{annulus-piece}.
		We refer to the non-polygonal boundary as the \emph{hole} in the annulus-piece.
		
		A \emph{piece on the $A$-side} is defined to be either a disk-piece or an annulus-piece. We denote the collection of all possible pieces on the $A$-side as $\tcP_A$.
		Similarly we define \emph{pieces on the $B$-side} and denote the corresponding collection as $\tcP_B$. Let $\tcP\defeq\tcP_A\cup\tcP_B$ be the collection of all pieces.
	\end{definition}
	
	Each turn on a polygonal boundary in $A$ with the given orientation travels from an arc labeled by some $\alpha_i$ to another labeled by some $\alpha_j$. We say such a turn is of \emph{type} $(\alpha_i,\alpha_j)$.
	Similarly each turn on the $B$-side is of type $(\beta_k,\beta_\ell)$ for some $k$ and $\ell$.
	We say a turn of type $(\alpha_i,\alpha_j)$ is \emph{compatible} with a turn of type $(\beta_{j-1},\beta_i)$, where indices are taken mod $L$.
	
	Pieces on the $A$-side can glue to pieces on the $B$-side along compatible turns and the maps on these pieces can be extended continuously in the obvious way; see the left of Figure \ref{fig: simple_S_A} where pieces are glued along compatible turns, the green dashed lines.
	
	\begin{definition}[Simple surfaces]\label{def: simple surf}
		A \emph{simple surface} $S$ is a finite collection of pieces in $\tcP$ together with a pairing on the set of turns of the given pieces so that paired turns are compatible. Geometrically, we think of $S$ as a surface obtained by gluing the given finitely many pieces along turns by the given pairing; see the left of Figure \ref{fig: simple_S_A} for an example where $L=2$.
		It follows from the definition of compatible turns that each boundary component of $S$ containing at least one arc must wind around $\gamma$ by a positive number of times.
		Define the degree $n(S)>0$ of $S$ to be the sum of these positive numbers.
    \end{definition}
    \begin{definition}[Gluing graph]\label{def: gluing graph}
		Associated to each simple surface $S$ is a \emph{gluing graph} $\Gamma_S$, where each vertex represents a piece and each edge connecting two vertices represents two paired compatible turns that the two pieces glue along; see the right of
		Figure \ref{fig: simple_S_A}.
	\end{definition}
	
	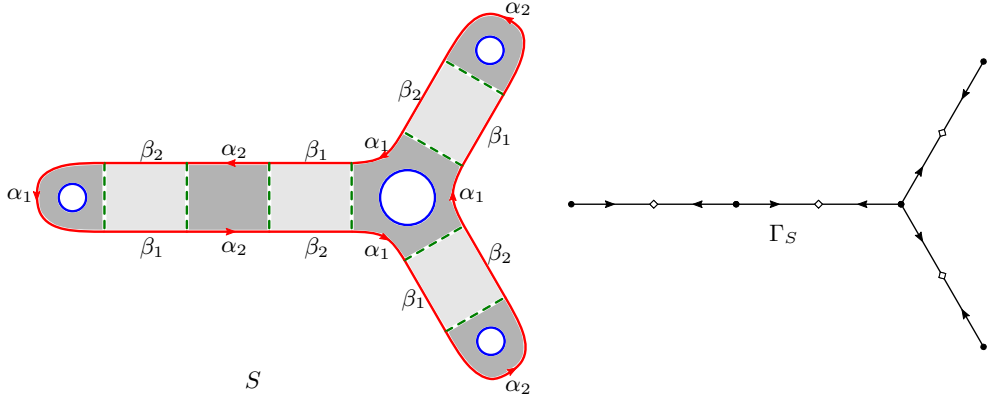
\begin{figure}
		\centering
		\labellist
		\small \hair 2pt
		\pinlabel $\alpha_1$ at -7 101
		\pinlabel $\beta_2$ at 65 125
		\pinlabel $\beta_1$ at 65 72
		\pinlabel $\alpha_2$ at 110 125
		\pinlabel $\alpha_2$ at 110 72
		\pinlabel $\beta_1$ at 155 125
		\pinlabel $\beta_2$ at 155 72
		\pinlabel $\alpha_1$ at 188 129
		\pinlabel $\alpha_1$ at 188 70
		\pinlabel $\alpha_1$ at 240 101
		\pinlabel $\beta_1$ at 255 133
		\pinlabel $\beta_2$ at 206 158
		\pinlabel $\alpha_2$ at 265 203
		\pinlabel $\beta_2$ at 255 69
		\pinlabel $\beta_1$ at 208 44
		\pinlabel $\alpha_2$ at 265 -4
		\pinlabel $S$ at 120 0
		
		\pinlabel $\Gamma_S$ at 410 80
		\endlabellist
		\includegraphics[scale=0.7]{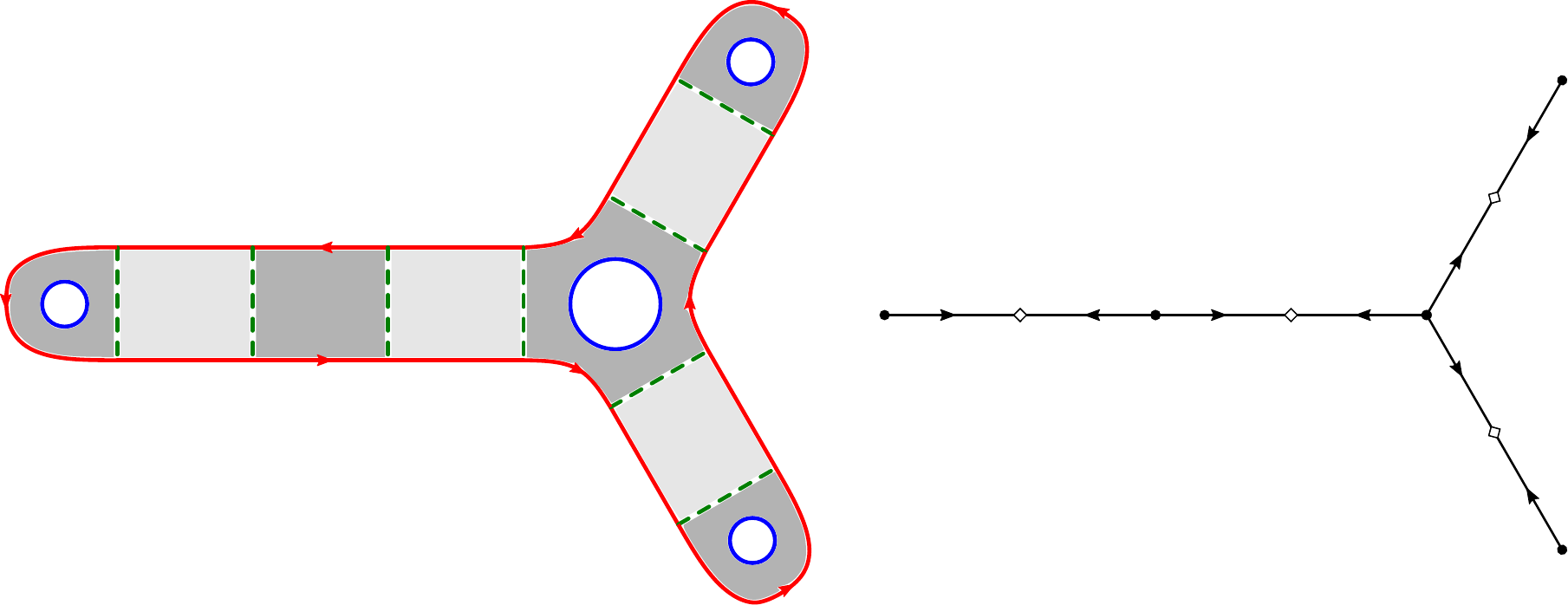}
		\caption{A simple surface $S$ and its gluing graph $\Gamma_S$. Here $S$ is obtained by simplifying the torsion-admissible surface in Figure \ref{fig: S_A}.}
		\label{fig: simple_S_A}
	\end{figure}
	
	\begin{lemma}\label{lemma: chi formula via graph}
		For a simple surface $S$, let $e$ be the number of edges in $\Gamma_S$ and let $d$ be the number of disk-pieces in $S$. Then
		$$-\chi(S)=e-d.$$
	\end{lemma}
	\begin{proof}
		By filling in all holes in annulus-pieces of $S$ (using disks), we obtain a surface $S'$ that deformation retracts to $\Gamma_S$. If $v$ is the number of vertices in $\Gamma_S$, then, 
		$$-\chi(S)=-\chi(S')+(v-d)=-\chi(\Gamma_S)+(v-d)=e-v+v-d=e-d,$$
		and $v-d$ is the number of annulus-pieces in $S$.
	\end{proof}
	
	To allow some desired flexibility, we do not require simple surfaces to be connected. A simple surface $S$ is connected if and only if the gluing graph $\Gamma_S$ is connected. Connected simple surfaces $S$ with $\chi(\Gamma_S)=1$ are closely related to torsion-admissible surfaces in normal form.
	\begin{lemma}\label{lemma: stl and simple surfaces}
		Given an element $g=a_1b_1\cdots a_Lb_L$ in the free product $A*B$. For a connected simple surface $S$ with $\chi(\Gamma_S)=1$, 
		if each hole in an annulus-piece on the $A$-side ($B$-side) represents a torsion element in $A$ (resp. $B$), 
		then $S$ is torsion-admissible for $g$ of degree $n(S)$.
		Conversely, any torsion-admissible surface $S$ for $g$, there exists a connected simple surface $S'$ of the same degree with $\chi(\Gamma_{S'})=1$ such that $-\chi(S')\le -\chi(S)$ and $H(S')\le H(S)$. 
		Thus
		\begin{equation}\label{eqn: stl simple estimate}
			\stl(g)\ge\inf \frac{H(S)}{n(S)}=\inf \frac{-\chi(S)}{n(S)},
		\end{equation}
		where each infimum is taken over all connected simple surfaces $S$ with $\chi(\Gamma_S)=1$.
	\end{lemma}
	\begin{proof}
		For a connected simple surface $S$, if we cap off all the holes in its annulus-pieces, then $S$ deformation retracts to $\Gamma_S$. 
		Thus if $\chi(\Gamma_S)=1$, the capped-off surface is a connected surface with boundary of Euler characteristic $1$, i.e. a disk.
		Therefore a simple surface $S$ with $\chi(\Gamma_S)=1$ has genus zero and a \emph{single} boundary component that represents $g^{n(S)}$, and all other boundary components of $S$
		are holes in its annulus-pieces. Hence if each hole in an annulus-piece in $\tcP_A$ (resp. $\tcP_B$) represents a torsion element in $A$ (resp. $B$), then $S$ is torsion-admissible for $g$ of degree $n(S)$ by definition.
		
		Conversely, given any torsion-admissible surface $S$ for $g$ of degree $n$, we can put it in normal form by Lemma \ref{lemma: put into normal form} without changing the degree. Each component of $S_A$ (resp. $S_B$) gives rise to a piece in $\tcP_A$ (resp. $\tcP_B$) except that it may contain more than one hole. For each component of $S_A$ containing more than one hole, replace it by a piece with the same polygonal boundary, which is a disk or an annulus depending on whether the winding class of the polygonal boundary is trivial or not. 
		Note that the winding class of such a polygonal boundary on the $A$-side (resp. $B$-side) always lies in $A_{tor}$ (resp. $B$-side) since it is the product of those torsion elements corresponding to the holes (up to conjugacy).
		The surface $S'$ obtained this way is connected and planar. Thus $S'$ is a connected simple surface with $\chi(\Gamma_{S'})=1$ and degree $n(S')=n(S)$. Moreover, as $S'$ is obtained from a normal form of $S$ by eliminating some holes, we have $-\chi(S')\le -\chi(S)$ and $H(S')\le H(S)$ using Lemma \ref{lemma: put into normal form}. See Figure \ref{fig: S_A} and Figure \ref{fig: simple_S_A} for an example.
		
		The two infima are equal for a similar reason to that of formula (\ref{eqn: stl via surfaces}) by taking suitable covering spaces; see Figure \ref{fig: branchedcover} for an illustration of a good covering space of a simple surface.
	\end{proof}
	
	It is often convenient to consider a subfamily of simple surfaces, where only a subset of types of pieces are allowed. 
	\begin{definition}\label{def: sufficient collection}
		Given a collection $\mathcal{P}\subset\tcP$ of types of pieces, a simple surface $S$ (with respect to $g$) is called \emph{$\mathcal{P}$-simple} if all pieces used in $S$ have types in $\mathcal{P}$.
		A collection $\mathcal{P}$ is called \emph{sufficient} if
		$$
		\stl(g)\ge\inf \frac{-\chi(S)}{n(S)},
		$$
		where the infimum is taken over all connected $\mathcal{P}$-simple surfaces $S$ with $\chi(\Gamma_S)=1$.
	\end{definition}
	
	Lemma \ref{lemma: stl and simple surfaces} shows that taking $\mathcal{P}=\tcP$ gives a sufficient collection.
	In many cases large pieces can be simplified into several small pieces (see Section \ref{subsec: splitting}), which allows a finite small collection $\mathcal{P}$ to be sufficient.
	
	If each hole in any annulus-piece in a sufficient collection $\mathcal{P}$ is guaranteed to represent a torsion element, 
	then stl is equal to the infimum by Lemma \ref{lemma: stl and simple surfaces}.
	\begin{corollary}\label{cor: simple surface is enough}
		Suppose
		\begin{enumerate}
		    \item both $A_{tor}$ and $B_{tor}$ are torsion groups,
		    \item or $g=a_1b_1\cdots a_Lb_L$, where the subgroup generated by $\{a_1,\cdots, a_L\}$ (resp. $\{b_1,\cdots,b_L\}$) is a torsion group.
		\end{enumerate}
		Then for any fixed sufficient collection $\mathcal{P}$, we have
		$$
		\stl(g)=\inf \frac{-\chi(S)}{n(S)},
		$$
		where the infimum is taken over all connected $\mathcal{P}$-simple surfaces $S$ with $\chi(\Gamma_S)=1$.
	\end{corollary}
	\begin{proof}
		The ``$\ge$'' direction holds by definition. The other direction holds since under both assumptions every connected simple surface $S$ 
		with $\chi(\Gamma_S)=1$ is torsion-admissible by Lemma \ref{lemma: stl and simple surfaces}.
	\end{proof}
	
	\subsection{Isometric embedding}\label{subsec: isom embed}
	As an application of Corollary \ref{cor: simple surface is enough}, injective homomorphisms of factor groups induce an embedding of the free products, which preserves the stable torsion length for generic elements, under suitable assumptions.
	\begin{theorem}\label{thm: isometric embedding}
		Let $i_A :A\to A'$ and $i_B: B\to B'$ be injective homomorphisms. Let $g\in A*B$ be an element that is not conjugate into $A$ or $B$. 
		Suppose
		\begin{enumerate}
		    \item either both $A'_{tor}$ and $B'_{tor}$ are torsion groups,\label{item: assumption one}
		    \item or $g$ is conjugate to a cyclically reduced word $a_1b_1\cdots a_Lb_L$ with each $a_j\in A$ and $b_j\in B$, 
		    where the subgroup generated by $\{a_1,\cdots, a_L\}$ and the subgroup generated by $\{b_1,\cdots,b_L\}$ are torsion groups.\label{item: assumption two}
		\end{enumerate}
		Then the induced map $i: A*B\to A'*B'$ preserves the stable torsion length of $g$, i.e. 
		$$\stl_{A*B}(g)=\stl_{A'*B'}(i(g)).$$
	\end{theorem}
	\begin{proof}
		It suffices to show that $\stl_{A*B}(g)\le\stl_{A'*B'}(i(g))$ since the other direction follows by monotonicity (Lemma \ref{lemma: monotone}). 
		Since $i_A$ is injective, torsion elements in $A$ correspond to torsion elements in the image of $i_A$.
		Thus $A_{tor}$ is a torsion group if $A'_{tor}$ is, and similarly for $B_{tor}$.
		Hence by Corollary \ref{cor: simple surface is enough}, under either assumption we have $\stl_{A*B}(g)=\inf_{S} \frac{-\chi(S)}{n(S)}$,
		where the infimum is taken over connected simple surfaces $S$ with $\chi(\Gamma_S)=1$. 
		
		On the other hand, by Lemma \ref{lemma: stl and simple surfaces}, we know $\inf_{S'} \frac{-\chi(S')}{n(S')}\le \stl_{A'*B'}(i(g))$, where the infimum is taken over all connected simple surfaces $S'$ for $i(g)$ with $\chi(\Gamma_{S'})=1$.
		Thus it suffices to show that $\inf_{S} \frac{-\chi(S)}{n(S)}\le \inf_{S'} \frac{-\chi(S')}{n(S')}$.
		We prove this by showing that every simple surface for $g'\defeq i(g)$ naturally pulls back to a simple surface for $g$.
		
		Up to conjugation, we may write $g=a_1b_1\cdots a_L b_L$, where $a_j\in A\setminus\{id\}$, $b_j\in B\setminus\{id\}$, and $L\ge1$. Then $g'=a'_1b'_1\cdots a'_L b'_L$, where $a'_j=i_A(a_j)$ and $b'_j=i_B(b_j)$.
		
		Let $C'$ be a piece on the $A'$-side in a simple surface $S'$ for $g'$, and suppose its polygonal boundary consists of arcs corresponding to $a'_{j_1}, \cdots, a'_{j_k}$ in the cyclic order, for some $k\in\Z_+$. Then we can construct a corresponding polygonal boundary by simply replacing each label $a'_{j_i}$ to $a_{j_i}$. Suppose its winding class is $w\in A$. 
		
		We claim that $w$ is torsion under both assumptions. Since $w$ is a product of $a_j$'s, this is obvious under assumption (\ref{item: assumption two}). If $A'_{tor}$ is a torsion group as in assumption (\ref{item: assumption one}), then the winding class of the polygonal boundary of $C'$ must be $i_A(w)$, which lies in $A'_{tor}$ and thus must be torsion. Since $i_A$ is injective, we know $w$ must be a torsion element as well in this case.
		
		It follows that the polygonal boundary we construct bounds an $A$-piece $C$. Moreover, since $w=id$ if and only if $i_A(w)=id$, the piece $C$ has the same topological type as the piece $C'$. Similarly we can construct a $B$-piece corresponding to any $B'$-piece. Doing this for all pieces of the simple surface $S'$, we obtain pieces that assemble accordingly to a simple surface $S$ for $g$ that has the same degree as $S'$ and $\chi(S)=\chi(S')$. Thus
		$\inf_{S} \frac{-\chi(S)}{n(S)}\le \inf_{S'} \frac{-\chi(S')}{n(S')}$ as desired, which completes the proof.
	\end{proof}
	
	Now Theorem \ref{introthm: Isom emb} follows as a simple corollary.
	\begin{proof}[Proof of Theorem \ref{introthm: Isom emb}]
	    If $g$ is conjugate to an element in $A$, then $g$ is torsion since $A$ is finite. Then $i(g)$ is also torsion and thus $\stl_{A*B}(g)=0=\stl_{A'*B'}(i(g))$.
	    Similarly the equality holds if $g$ is conjugate to an element in $B$.
	    
	    Now if $g$ is not conjugate into $A$ or $B$, then the assumption (\ref{item: assumption two}) in Theorem \ref{thm: isometric embedding} obviously holds since $A$ and $B$ are finite groups. Hence the equality follows by Theorem \ref{thm: isometric embedding}.
	\end{proof}
	
	\subsection{Lower bounds via linear programming}\label{subsec: linprog lower bound}
	Given a sufficient collection $\mathcal{P}$ of types of pieces (Definition \ref{def: sufficient collection}),
	we describe a (possibly infinite-dimensional) linear programming problem that produces a lower bound of $\stl_G(g)$. 
	This is based on the following observation.
	
	\begin{lemma}\label{lemma: linprog estimate}
		Suppose $\mathcal{P}$ is a sufficient collection for a free product $G=A*B$. 
		For any element $g$ not conjugate into factor groups, we have
		$$\stl_G(g)\ge\inf_S\frac{-\chi(S)}{n(S)},$$
		where the infimum is taken over all (not necessarily connected) $\mathcal{P}$-simple surfaces $S$ with $\chi(\Gamma_S)\ge0$.
	\end{lemma}
	\begin{proof}
		Since the family of surfaces we consider here contains all connected $\mathcal{P}$-simple surfaces $S$ with $\chi(\Gamma_S)=1$, 
		the result follows from the definition of sufficient collections (Definition \ref{def: sufficient collection}).
	\end{proof}
	
	The reason to consider $\mathcal{P}$-simple surfaces that are not necessarily connected with the relaxed constraint $\chi(\Gamma_S)\ge0$ is that this family of surfaces is easier to work with, 
	allowing us to encode such surfaces as vectors in a nice subspace of a vector space as follows.
	
	Given a reduced word $g$ and a sufficient collection $\mathcal{P}$ as above, let $V_\mathcal{P}=\R^{\mathcal{P}}$. 
	Let $\{e_P\mid P\in\mathcal{P}\}$ be the standard basis, where $e_P$ is the positive unit vector in the $P$-direction.
	For any $\mathcal{P}$-simple surface $S$ and any $P\in \mathcal{P}$, let $x_P$ be the number of pieces of type $P$ in $S$. 
	Associate to $S$ a vector $v(S)\in V_\mathcal{P}$ so that the $P$-component of $v(S)$ is $x_P$ for any $P\in \mathcal{P}$.
	
	The vector $v(S)$ is a non-negative integer point in $V_\mathcal{P}$ satisfying some rational linear constraints that we describe below.
	For any turn type $T$ (e.g. $T=(\alpha_i,\alpha_j)$ or $(\beta_i,\beta_j)$), there is a linear function $f_T:V_\mathcal{P}\to\R$ such that for the standard basis $f_T(e_P)$ counts the number of turns of type $T$ in a piece of type $P$ for each $P\in\mathcal{P}$. For any two compatible turn types
	$T,T'$, i.e. $T=(\alpha_i,\alpha_j)$ and $T'=(\beta_{j-1},\beta_i)$, we have 
	$$f_T(v(S))=f_{T'}(v(S))$$
	for any $\mathcal{P}$-simple surface $S$ since each turn is glued to another turn by definition. We refer to this set of equations as the \emph{gluing conditions}.
	
	Let $\chi_\Gamma:V_\mathcal{P}\to\R$ be the linear function determined by $\chi(e_P)=1-\frac{e}{2}$ for each $P\in\mathcal{P}$, 
	where $e$ is the number of turns on the polygonal boundary of the piece $P$.
	Similarly we have a linear function $\chi_o:V_\mathcal{P}\to\R$ with the property that $\chi_o(e_P)=\chi(P)-\frac{e}{2}$ for each $P\in\mathcal{P}$, where $\chi(P)$ is $1$ if $P$ is a disk-piece and is $0$ if $P$ is an annulus-piece.
	Then it is straightforward to see that $\chi_\Gamma(v(S))=\chi(\Gamma_S)$ and $\chi_o(v(S))=\chi(S)$ for any $\mathcal{P}$-simple surface $S$ with gluing graph $\Gamma_S$.
	
	Finally, let $n:V_\mathcal{P}\to\R$ be the linear function such that $n(e_P)$ counts the number of copies of $\alpha_1$ on the polygonal boundary of $P$ for each $P\in\mathcal{P}$. Then for any $\mathcal{P}$-simple surface $S$, its degree is $n(v(S))$.
	
	
	\begin{definition}
		Given the word $g$ and a sufficient collection $\mathcal{P}$, let $C_\mathcal{P}$ be the subspace of $V_\mathcal{P}$ consisting of vectors $x$ with non-negative components such that $x$ satisfies all gluing conditions, $\chi_\Gamma(x)\ge0$, and $n(x)=1$.
	\end{definition}
	Note that $C_\mathcal{P}$ depends on $g$ (since $\mathcal{P}$ does, for instance).	
	
	Summarizing up the discussion above, we have:
	\begin{lemma}\label{lemma: surf to vect}
		For any $\mathcal{P}$-simple surface of degree $n$, the vector $v(S)/n$ is a rational point in $C_\mathcal{P}$ and $\chi_o(v(S)/n)=\chi(S)/n$.
	\end{lemma}
	
	Conversely, we have:
	\begin{lemma}\label{lemma: vect to surf}
		For any rational point $x\in C_\mathcal{P}$, there is some $n\in\Z_+$ and a $\mathcal{P}$-simple surface $S$ of degree $n$, such that $x=v(S)/n$, $\chi(\Gamma_S)\ge0$ and $\chi(S)/n=\chi_o(x)$.
	\end{lemma}
	\begin{proof}
		Choose $n$ so that $nx$ is an integer point in $V_\mathcal{P}$. Then $nx=\sum_P k_P e_P$ for some non-negative integers $k_P$. Take $k_P$ pieces of type $P$ for each $P\in \mathcal{P}$. Since $x$ satisfies the gluing conditions, so does $nx$. Thus by gluing these pieces along compatible pairs of turns, we obtain a $\mathcal{P}$-simple surface $S$ such that $v(S)=nx$. Then we have $\chi(\Gamma_S)=\chi_\Gamma(nx)=n\chi_\Gamma(x)\ge0$ and $\chi(S)/n=\chi_o(nx)/n=\chi_o(x)$ since both $\chi_\Gamma$ and $\chi_o$ are linear on $V_\mathcal{P}$.
	\end{proof}
	
	It follows that we can compute the infimum in Lemma \ref{lemma: linprog estimate} by minimizing the rational linear function $-\chi_o$ on the compact polyhedron $C_\mathcal{P}$ (see the lemma below),  which is a linear programming problem. This gives a way to compute a nontrivial lower bound of $\stl_G(g)$. We compute two explicit examples in Section \ref{sec: examples}, where the lower bounds are actually sharp in both cases.
	
	\begin{lemma}\label{lemma: linprog}
		For a finite sufficient collection $\mathcal{P}$ with respect to a given element $g$, the set $C_\mathcal{P}$ is a rational compact polyhedron. Moreover, it is nonempty if some power of $g$ is a product of torsion elements in $G$. In this case, the infimum of $-\chi(S)/n(S)$ over all $\mathcal{P}$-simple surfaces $S$ with $\chi(\Gamma_S)\ge0$ is achieved.
	\end{lemma}
	\begin{proof}
		By definition, the set $C_\mathcal{P}$ is defined by finitely many rational linear inequalities. The normalizing condition $n(x)=1$ together with gluing conditions implies that each coordinate of $x$ is no more than $1$ for any $x\in C_\mathcal{P}$. In other words, the normalized number of any piece is at most one since the degree is normalized to be one. Thus $C_\mathcal{P}$ is a rational compact polyhedron when $\mathcal{P}$ is finite.
		
		When a power of $g$ is a product of torsion elements in $G$, torsion-admissible surfaces exist. Thus by Lemma \ref{lemma: stl and simple surfaces}, we can reduce any torsion-admissible surface to a simple surface $S$ with $\chi(\Gamma_S)=1$, which yields a rational point in $C_\mathcal{P}$ by Lemma \ref{lemma: surf to vect}. Hence $C_\mathcal{P}$ is nonempty.
		
		Then by Lemmas \ref{lemma: surf to vect} and \ref{lemma: vect to surf}, the infimum of $-\chi(S)/n(S)$ over all $\mathcal{P}$-simple surfaces $S$ with $\chi(\Gamma_S)\ge0$ can be calculated as the infimum of the rational linear function $-\chi_o$ on $C_\mathcal{P}$. Hence by compactness the infimum is achieved by a (rational) vertex $x$ of $C_\mathcal{P}$, which by Lemma \ref{lemma: vect to surf} is of the form $v(S)/n$ for a simple surface $S$ of degree $n$ in the above family.
	\end{proof}

	\section{Free products of finite groups}\label{sec: compute stl}
	In this section we focus on the case of a free product $G=A*B$, where $A$ and $B$ are finite groups.
	We will exhibit an algorithm that computes $\stl_G(g)$ for any given element $g$.
	The nontrivial case is when $g=a_1b_1\cdots a_L b_L$ is a reduced word with $L\ge1$.
	The method applies to the more general case when the subgroups generated by $\{a_1,\cdots, a_L\}$ and $\{b_1,\cdots,b_L\}$ respectively are both finite, which also follows by the isometric embedding Theorem \ref{thm: isometric embedding}. The techniques also work for the free product of arbitrarily many finite groups, but we won't pursue it here. Potential generalizations to the infinite factor groups will be discussed in Section \ref{subsec: generalization}.
	
	We adopt the setup and notation in the previous section.
	Note that by Corollary \ref{cor: simple surface is enough}, to compute $\stl_G(g)$ for a given element $g$, it suffices to consider connected simple surfaces $S$ (Definition \ref{def: simple surf}) with $\chi(\Gamma_S)=1$, where $\Gamma_S$ is the gluing graph (Definition \ref{def: gluing graph}). 
	
	We will introduce two operations that further simplify surfaces: \emph{splitting} and \emph{rewiring}. 
	In terms of the gluing graph, we will use splitting to reduce the valence of vertices, and use rewiring to reduce the diameter of the graph.
	
	However, the family of connected simple surfaces $S$ with $\chi(\Gamma_S)=1$ is not closed under the two operations above. 
	As a remedy, we consider the larger family of simple surfaces with $\chi(\Gamma_\Sigma)\ge0$ for each component $\Sigma$, which is more convenient to work with
	for two reasons:
	\begin{itemize}
		\item The two operations (when applied appropriately) preserve this family, and
		\item The complexity of any connected simple surface with $\chi(\Gamma_S)=0$ can be approximated by a sequence of connected simple surfaces with $\chi(\Gamma_S)=1$, and thus can still be used to compute stl; see Lemma \ref{lemma: approximation}.
	\end{itemize}
	
	For a connected simple surface $S$ with $\chi(\Gamma_S)=0$, the gluing graph $\Gamma_S$ is connected and has a unique embedded loop, which we refer to as the \emph{core}.
	The gluing graph can be thought of as obtained from the core by attaching finitely many (rooted) trees to vertices on the core. We refer to each of such trees as \emph{a decorative tree}, and the vertex it attaches to as the \emph{root}. See Figure \ref{fig: chi0graph}.
	
	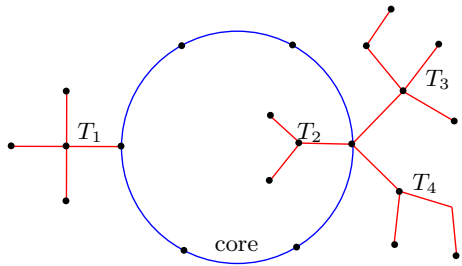
\begin{figure}
		\centering
		\labellist
		\small \hair 2pt
		\pinlabel $\text{core}$ at 125 10
		\pinlabel $T_1$ at 45 72
		\pinlabel $T_2$ at 165 72
		\pinlabel $T_3$ at 235 101
		\pinlabel $T_4$ at 228 44
		\endlabellist
		\includegraphics[scale=0.7]{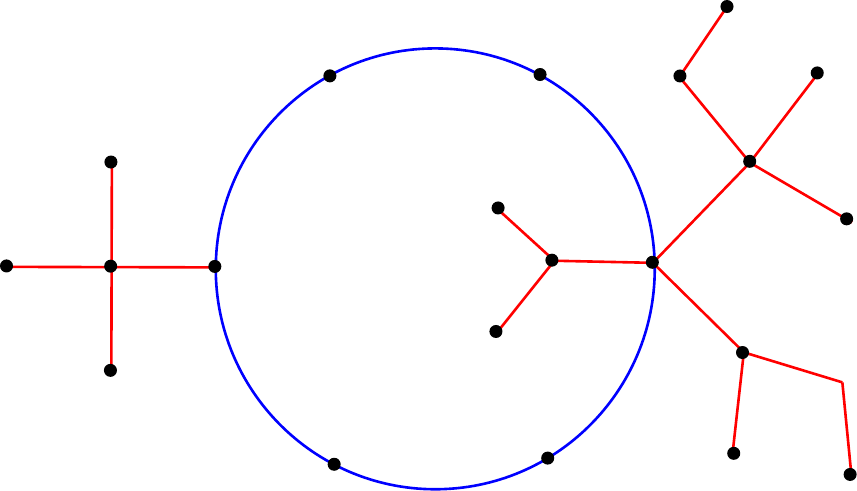}
		\caption{A connected graph with Euler characteristic zero that consists of a core (in blue) and four decorative trees (in red).}
		\label{fig: chi0graph}
	\end{figure}
	
	We first introduce the two operations, splitting and rewiring, in Sections \ref{subsec: splitting} and \ref{subsec: rewiring} respectively. 
	In particular we show the approximation Lemma \ref{lemma: approximation} in Section \ref{subsec: rewiring} using rewiring.
	Then in Section \ref{subsec: irreducible} we define irreducible simple surfaces and show that every connected simple surface $S$ with $\chi(\Gamma_S)\ge0$ decomposes into a union of irreducible ones after applying splitting and rewiring. This yields an algorithm to compute stl and shows that it is rational in free products of finite groups; see Theorem \ref{thm: rationality}.
	
	\subsection{Splitting of a piece}\label{subsec: splitting}
	The first operation that we introduce on simple surfaces is \emph{splitting} of a piece. 
	
	Let $C$ be a piece on the $A$-side with polygonal boundary $P$. Suppose a proper subset of the turns on $P$ can form a polygonal boundary $P_1$ with trivial winding class, and the remaining turns can form another polygonal boundary $P_2$. Then $P_2$ has the same winding class as $P$ if $A$ is abelian.
	
	In this case $P_1$ bounds a disk-piece $C_1$ and $P_2$ bounds a piece $C_2$ that has the same topological type as the original piece $C$. When $C$ is a piece on a simple surface $S$, \emph{splitting} is the operation that we replace the piece $C$ above by the two new pieces $C_1$ and $C_2$ without changing the gluing of turns. This modifies the simple surface without changing the number of holes while splitting one vertex of the gluing graph $\Gamma_S$ into two.
	
	\begin{example}\label{example: splitting}
	    Let $A=\Z/p$ be a cyclic group generated by $a$ and  $B$ be arbitrary. Let $g=ab\bar{a}\bar{b}\in A*B$, where $\bar{a}$ and $\bar{b}$ denote $a^{-1}$ and $b^{-1}$ respectively. 
	    Let $\gamma$ be a loop representing $g$, decomposed into arcs $\alpha_1,\beta_1,\alpha_2,\beta_2$ corresponding to $a,b, \bar{a},\bar{b}$ respectively. 
	    \begin{enumerate}
	        \item On the left of Figure \ref{fig: splitting}, we have disk-piece $C$ on the $A$-side with two copies of each of the four turns $(\alpha_1,\alpha_1)$, $(\alpha_1,\alpha_2)$, $(\alpha_2,\alpha_2)$, $(\alpha_2,\alpha_1)$. The two turns $(\alpha_1,\alpha_2)$ and $(\alpha_2,\alpha_1)$ form a disk-piece $C_1$ and the remaining six turns form another disk-piece $C_2$. The splitting breaks the piece $C$ into the pieces $C_1$ and $C_2$.\label{item: wild splitting}
	        \item There are two consecutive copies of the turn $(\alpha_1,\alpha_1)$ on $C_2$. If $p=2$, these two turns form a disk-piece $C_3$ since $a^2=1$, and the remaining turns form another disk-piece $C_4$, shown on the right of Figure \ref{fig: splitting}. Note that in this case, the turns on the new pieces sit in a cyclic order compatible to the their cyclic order on $C_2$, which is not the case for the previous splitting.\label{item: geom splitting}
	    \end{enumerate}
	    
	    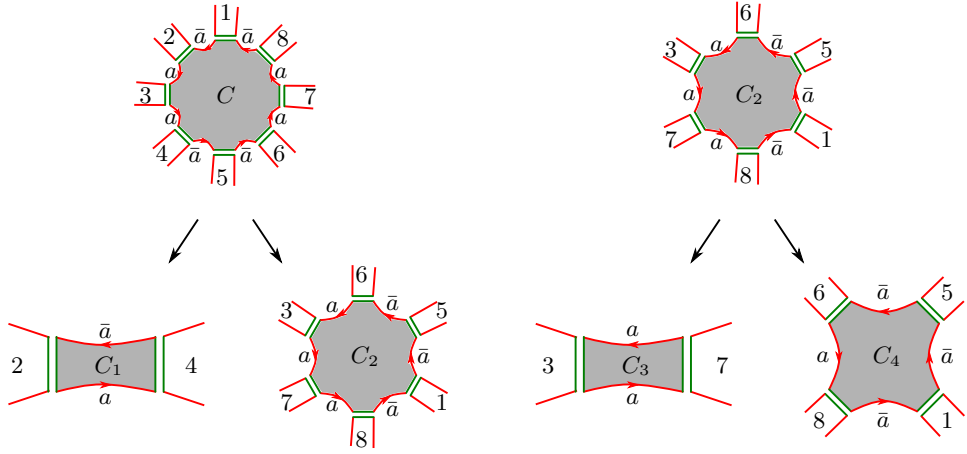
\begin{figure}
		\centering
		\labellist
		\small \hair 2pt
		\pinlabel $C$ at 119 193
		\pinlabel $a$ at 89 205
		\pinlabel $a$ at 89 180
		\pinlabel $\bar{a}$ at 102 160
		\pinlabel $\bar{a}$ at 130 160
		\pinlabel $a$ at 149 205
		\pinlabel $a$ at 149 180
		\pinlabel $\bar{a}$ at 105 225
		\pinlabel $\bar{a}$ at 131 225
		\pinlabel $1$ at 119 238
		\pinlabel $2$ at 88 225
		\pinlabel $3$ at 75 193
		\pinlabel $4$ at 84 161
		\pinlabel $5$ at 117 148
		\pinlabel $6$ at 148 161
		\pinlabel $7$ at 165 192
		\pinlabel $8$ at 150 224
		
		\pinlabel $C_1$ at 55 45
		\pinlabel $a$ at 53 27
		\pinlabel $\bar{a}$ at 53 64
		\pinlabel $2$ at 5 45
		\pinlabel $4$ at 100 45
		
		\pinlabel $C_2$ at 195 50
		\pinlabel $a$ at 177 23
		\pinlabel $a$ at 162 50
		\pinlabel $a$ at 177 77
		\pinlabel $\bar{a}$ at 210 21
		\pinlabel $\bar{a}$ at 227 50
		\pinlabel $\bar{a}$ at 210 79
		\pinlabel $6$ at 193 95
		\pinlabel $3$ at 152 75
		\pinlabel $7$ at 152 25
		\pinlabel $8$ at 193 5
		\pinlabel $1$ at 237 25
		\pinlabel $5$ at 237 75
		
		\pinlabel $C_2$ at 405 193
		\pinlabel $a$ at 387 166
		\pinlabel $a$ at 372 193
		\pinlabel $a$ at 387 220
		\pinlabel $\bar{a}$ at 420 164
		\pinlabel $\bar{a}$ at 437 193
		\pinlabel $\bar{a}$ at 420 222
		\pinlabel $6$ at 403 238
		\pinlabel $3$ at 362 218
		\pinlabel $7$ at 362 169
		\pinlabel $8$ at 403 149
		\pinlabel $1$ at 447 169
		\pinlabel $5$ at 447 218
		
		\pinlabel $C_3$ at 343 45
		\pinlabel $a$ at 341 27
		\pinlabel $a$ at 341 64
		\pinlabel $3$ at 295 45
		\pinlabel $7$ at 390 45
		
		\pinlabel $C_4$ at 480 50
		\pinlabel $a$ at 445 50
		\pinlabel $\bar{a}$ at 478 15
		\pinlabel $\bar{a}$ at 513 50
		\pinlabel $\bar{a}$ at 478 85
		\pinlabel $5$ at 513 85
		\pinlabel $6$ at 443 85
		\pinlabel $8$ at 443 15
		\pinlabel $1$ at 513 15
		\endlabellist
		\includegraphics[scale=0.7]{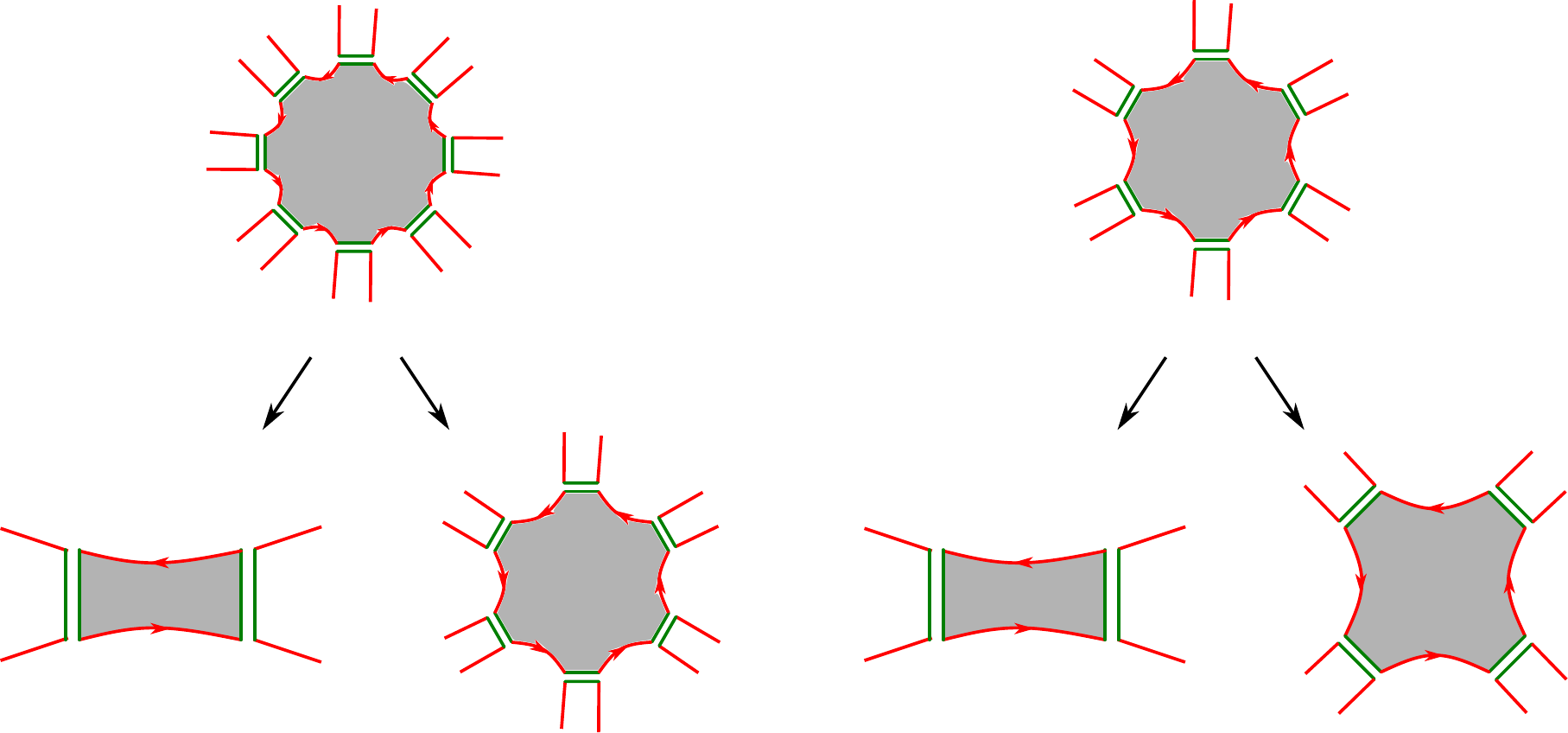}
		\caption{Two examples of splitting, where the one on the right assumes $a^2=1$.}
		\label{fig: splitting}
	\end{figure}
	\end{example}
	
	For the rest of this section, we focus on a special case similar to case (\ref{item: geom splitting}) in Example \ref{example: splitting}, where splitting works without assuming the factor groups to be abelian. Suppose there is a proper subsequence of sides
	$$(T_0,A_1,T_1,\cdots, A_k, T_k)$$
	for some $k\ge1$ on the polygonal boundary $P$ in the positive cyclic order starting and ending at turns $T_0,T_k$ of the \emph{same} type, such that
	the product of the elements represented by the arcs $A_1,\cdots,A_k$ is the identity in $A$.
	Let 
	$$(A_{k+1},T_{k+1},\cdots, A_{n})$$
	be the complementary sequence of sides in the positive cyclic order, where $n>k$.
	Then we obtain two polygonal boundaries $P_1$ and $P_2$, where the sides are
	$(A_1,T_1,\cdots, A_k, T_k)$ and $(A_{k+1},T_{k+1},\cdots, A_n, T_n=T_0)$, respectively.
	Then by the assumption, the winding class of $P_1$ is trivial and thus $P_1$ bounds a disk-piece $C_1$. The winding class of $P_2$ is the same as that of $P$ and thus $P_2$ bounds a piece $C_2$ that has the same topological type as $C$.
	
	Splitting decomposes such a piece $C$ into two pieces $C_1$ and $C_2$ without changing the total number of holes. In addition, it does not affect the gluing of compatible turns.
	Analogously one can perform this for pieces on the $B$-side.
	
	It is helpful to think about the effect of this operation conceptually in terms of the gluing graph $\Gamma_S$. 
	Orient edges so that they go from vertices representing pieces on the $A$-side to those on the $B$-side. Such edges fall into different \emph{types} according to the types of turns.
	Then splitting of a piece applies to a vertex which necessarily have two adjacent edges $e_1,e_2$ of the same type. It replaces such a vertex $v$ by two vertices $v_1,v_2$, where part of the original adjacent edges become edges at $v_1$ and the others are edges at $v_2$, so that $e_i$ is an edge at $v_i$.
	
	Note that $\Gamma_S$ is actually a \emph{fatgraph} in the sense that there is a cyclic order on the edges at each vertex, which is induced from the orientation on the polygonal boundary of each piece. Hence for this special type of splitting, the cyclic order at the vertex $v$ that we split and the position of $e_1,e_2$ in the order determine which edges of $v$ become edges of $v_1$ and $v_2$.
	
	We are able to apply splitting to any vertex with large valence in the gluing graph if the corresponding factor group is finite. 
	\begin{lemma}\label{lemma: splitting applicable}
		Suppose the factor group $A$ is finite. Then for the element $g=a_1b_1\cdots a_L b_L$, we can split any piece on the $A$-side that has more than $|A|\cdot L^2$ turns on its polygonal boundary.
	\end{lemma}
	\begin{proof}
		Note that there are $L^2$ possible types of turns on the $A$-side.
		Suppose there are more than $|A|\cdot L^2$ turns, then by the pigeonhole principle there exist $|A|+1$ turns of the same type. These turns cut the polygonal boundary into $|A|+1$ segments.
		Each segment represents an element in $A$ by taking the product of elements corresponding to the arcs on the segment. Let $x_1,\cdots,x_{|A|+1}\in A$ be the elements corresponding to these segments. Then by the pigeonhole principle, there exist $1\le m<n\le |A|+1$ such that $x^{(m)}=x^{(n)}$, where $x^{(i)}=x_1\cdots x_i$. This implies $x_{m+1}\cdots x_n=id$ and hence we can apply splitting to this piece.
	\end{proof}
	
	\subsection{Rewiring}\label{subsec: rewiring}
	The second operation that we introduce on simple surfaces is \emph{rewiring}. This has been used in a similar setting to understand stable commutator length \cite{Chen:sclBS}. However, in this setting, it is necessary to apply this operation more carefully here to control the Euler characteristic of each component of the gluing graph. We describe this below. 
	
	Suppose that there exist two edges $e_1,e_2$ in the gluing graph of the same type and that each $e_i$ goes from a vertex $u_i$ to a vertex $v_i$, where $i=1,2$. Geometrically, thinking of vertices as pieces, this means that $u_1,v_1$ are glued together along compatible turns, which have the same type as the compatible turns along which we glue $u_2$ and $v_2$.
	Then, we can cut along these two pairs of turns and glue $u_1$ to $v_2$ and glue $u_2$ to $v_1$ instead.
	In terms of the gluing graph, we remove the edges $e_1$ and $e_2$ and construct two new edges connecting $u_1$ to $v_2$ and $u_2$ to $v_1$ instead; see Figure \ref{fig: rewire}.
	
	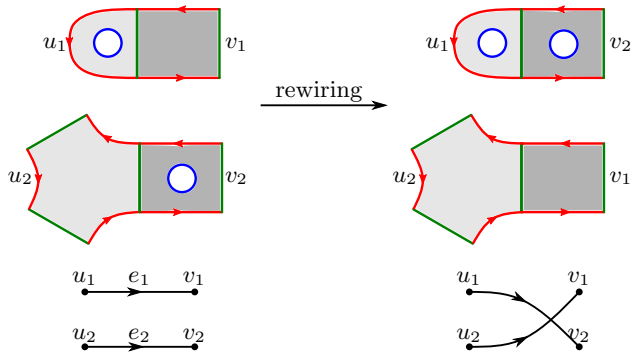
\begin{figure}
		\centering
		\labellist
		\small \hair 2pt
		\pinlabel $u_1$ at 16 168
		\pinlabel $v_1$ at 115 168
		\pinlabel $u_2$ at -3 95
		\pinlabel $v_2$ at 115 95
		\pinlabel $u_1$ at 32 40
		\pinlabel $e_1$ at 62 40
		\pinlabel $v_1$ at 92 40
		\pinlabel $u_2$ at 32 9
		\pinlabel $e_2$ at 62 9
		\pinlabel $v_2$ at 92 9
		
		\pinlabel $\text{rewiring}$ at 160 142
		
		\pinlabel $u_1$ at 225 168
		\pinlabel $v_2$ at 325 168
		\pinlabel $u_2$ at 207 95
		\pinlabel $v_1$ at 325 95
		\pinlabel $u_1$ at 242 40
		\pinlabel $v_1$ at 302 40
		\pinlabel $u_2$ at 242 9
		\pinlabel $v_2$ at 302 9
		\endlabellist
		\includegraphics[scale=0.7]{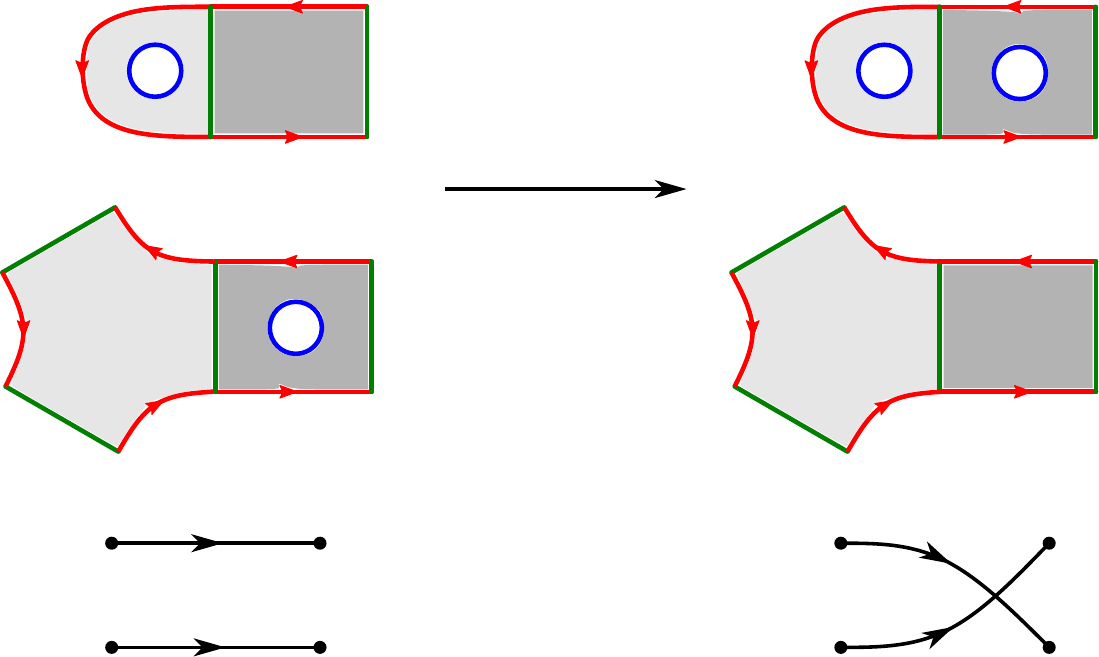}
		\caption{The effect of rewiring on the simple surface and gluing graph}
		\label{fig: rewire}
	\end{figure}
	
	Note that rewiring does not change the types of pieces and preserves the number of vertices and edges of the gluing graph. Thus applying the rewiring operation to a simple surface $S$ results in another simple surface $S'$ with the same Euler characteristic. However, it might change the number of components and the Euler characteristic of individual connected components. For this reason, we will apply rewiring in a restricted way to preserve the family of simple surfaces with $\chi(\Gamma_\Sigma)\ge0$ for each connected component $\Sigma$. 
	
	If $e_1,e_2$ lie in different components of $\Gamma_S$ and at least one of them is non-separating, then the rewiring merges the two components into a single component. This shows that the complexity of any connected simple surface with $\chi(\Gamma_S)=0$ can be approximated by those with $\chi(\Gamma_S)=1$. Before proving this in Lemma \ref{lemma: approximation}, we need the following simple observation.
	
	\begin{lemma}\label{lemma: ample}
		If $A$ and $B$ are finite, then for the given element $g=a_1b_1\cdots a_Lb_L$ and for any turn type $(\alpha_i,\alpha_j)$, there is a connected simple surface $S$ with $\chi(\Gamma_S)=1$ that contains a turn of the given type $(\alpha_i,\alpha_j)$.
	\end{lemma}
	\begin{proof}
		Since $A$ is finite, there is a piece $P^A_{i,j}$ with exactly two turns $(\alpha_i,\alpha_j)$ and $(\alpha_j,\alpha_i)$. Similarly there is a piece $P^B_{i,j}$ with exactly two turns $(\beta_i,\beta_j)$ and $(\beta_j,\beta_i)$.
		So it suffices to put each $P^A_{i,j}$ into a simple surface with the desired properties. 
		Note that there is an arbitrarily long strip of pieces glued together centered at $P^A_{i,j}$, such that on one side we have $P^B_{i,j-1}$, $P^A_{i+1,j-1}$, $P^B_{i+1,j-2}$ and so on, and on the other side we have $P^B_{i-1,j}$, $P^A_{i-1,j+1}$, $P^B_{i-2,j+1}$ and so on; see the top of Figure \ref{fig: forge} for an example with $i=2$ and $j=L=4$. Here the indices are taken mod $L$. 
		
		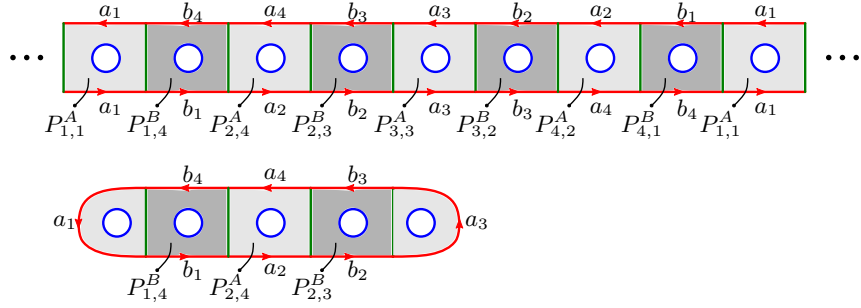
\begin{figure}
			\centering
			\labellist
			\small \hair 2pt
			\pinlabel $a_1$ at 55 142
			\pinlabel $b_4$ at 100 142
			\pinlabel $a_4$ at 145 142
			\pinlabel $b_3$ at 190 142
			\pinlabel $a_3$ at 235 142
			\pinlabel $b_2$ at 280 142
			\pinlabel $a_2$ at 323 142
			\pinlabel $b_1$ at 370 142
			\pinlabel $a_1$ at 413 142
			
			\pinlabel $a_1$ at 55 89
			\pinlabel $b_1$ at 100 89
			\pinlabel $a_2$ at 145 89
			\pinlabel $b_2$ at 190 89
			\pinlabel $a_3$ at 235 89
			\pinlabel $b_3$ at 280 89
			\pinlabel $a_4$ at 323 89
			\pinlabel $b_4$ at 370 89
			\pinlabel $a_1$ at 413 89
			
			\pinlabel $P^A_{1,1}$ at 30 80
			\pinlabel $P^B_{1,4}$ at 75 80
			\pinlabel $P^A_{2,4}$ at 120 80
			\pinlabel $P^B_{2,3}$ at 165 80
			\pinlabel $P^A_{3,3}$ at 210 80
			\pinlabel $P^B_{3,2}$ at 255 80
			\pinlabel $P^A_{4,2}$ at 298 80
			\pinlabel $P^B_{4,1}$ at 345 80
			\pinlabel $P^A_{1,1}$ at 388 80
			
			\pinlabel $b_4$ at 100 53
			\pinlabel $a_4$ at 145 53
			\pinlabel $b_3$ at 190 53
			
			\pinlabel $a_1$ at 30 27
			\pinlabel $b_1$ at 100 1
			\pinlabel $a_2$ at 145 1
			\pinlabel $b_2$ at 190 1
			\pinlabel $a_3$ at 255 27
			
			\pinlabel $P^B_{1,4}$ at 75 -9
			\pinlabel $P^A_{2,4}$ at 120 -9
			\pinlabel $P^B_{2,3}$ at 165 -9
			\endlabellist
			\includegraphics[scale=0.7]{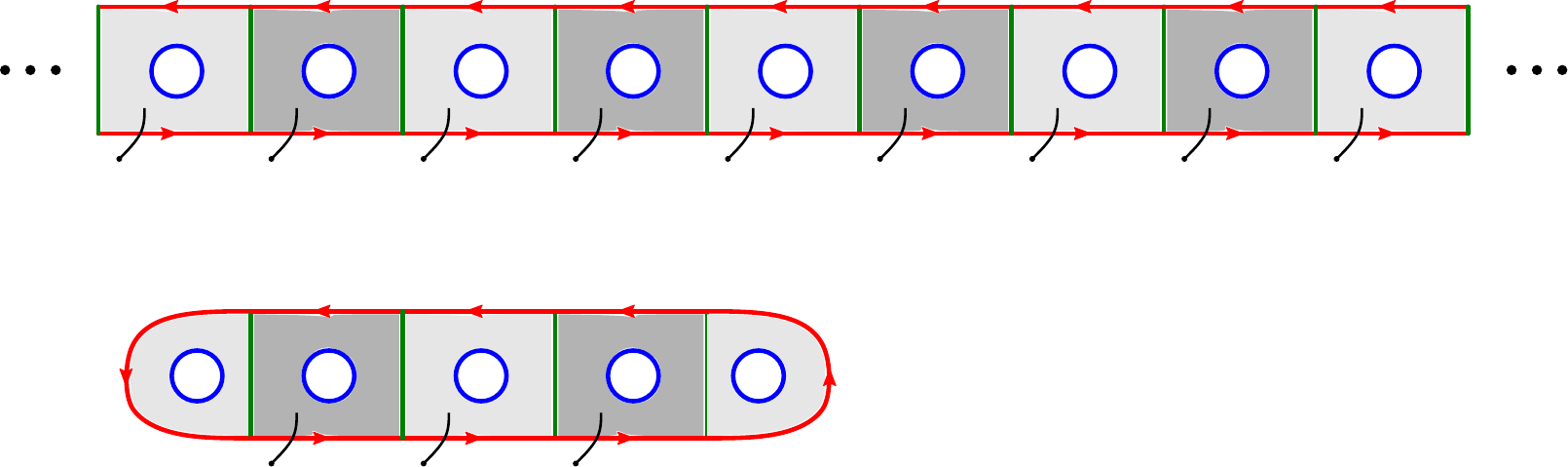}
			\vspace{5pt}
			\caption{An example to include the turn type $(\alpha_2,\alpha_4)$ in a long strip, which can be modified into a simple surface $S$ for $g=a_1b_1\cdots a_4b_4$ with $\chi(\Gamma_S)=1$.}
			\label{fig: forge}
		\end{figure}
		
		For each piece $P^*_{k,\ell}$ in the above sequence, we take the difference $k-\ell$ of the two subindices. 
		Then we observe that the differences between consecutive pieces are consecutive integers and form a monotone sequence.
		Thus on both sides of $P^A_{i,j}$, we can find pieces of the form $P^*_{m,n}$ with $m\equiv n\mod L$ and $*=A$ or $B$. 
		We can cut the strip at such a piece and replace this piece by the piece with only one arc $\alpha_m$ (resp. $\beta_m$) if $*=A$ (resp. $*=B$); 
		see the bottom of Figure \ref{fig: forge}. This constructs a connected simple surface $S$ containing $P^A_{i,j}$ such that $\Gamma_S$ is a tree.
	\end{proof}
	
	\begin{lemma}\label{lemma: approximation}
		Suppose that for the given element $g$ and any turn type $(\alpha_i,\alpha_j)$, there is a connected simple surface $S$ with $\chi(\Gamma_S)=1$.
		Then, for any connected simple surface $S$ of degree $n$ with $\chi(\Gamma_S)=0$, there exists a sequence of connected simple surfaces $S_k$ of degree $n_k$ with $\chi(\Gamma_{S_k})=1$ such that 
		$$\frac{-\chi(S)}{n}=\lim_k \frac{-\chi(S_k)}{n_k}.$$
	\end{lemma}
	\begin{proof}
		Note that for each $k\in\Z_+$ there is a unique degree $k$ cover $\Gamma_{k}$ of $\Gamma_S$. This induces a degree $k$ cover $\Sigma_k$ of $S$, which is again a connected simple surface (of degree $kn$) with gluing graph $\Gamma_k$.
		Now pick any edge $e$ on the core of $\Gamma_S$ corresponding to a turn, say, of type $(\alpha_i,\alpha_j)$.
		By assumption, we can fix a simple surface $S_0$ containing a turn of type $(\alpha_i,\alpha_j)$ such that $\Gamma_{S_0}$ is a tree. Then a lift of $e$ to $\Gamma_k$ provides a non-separating edge on $\Sigma_k$ corresponding to a turn of type $(\alpha_i,\alpha_j)$. Thus we can apply rewiring to $\Sigma_k$ and $S_0$ to obtain a new simple surface $S_k$ where the gluing graph $\Gamma_{S_k}$ is connected. Then $$\chi(\Gamma_{S_k})=\chi(\Gamma_k)+\chi(S_0)=0+1=1.$$
		The degree of $S_k$ is $n_k=kn+n_0$, where $n_0$ is the degree of $S_0$.
		Then, we have
		$$\lim_{k\to\infty} \frac{-\chi(S_k)}{n_k}=\lim_{k\to\infty}\frac{-\chi(\Sigma_k)-\chi(S_0)}{kn+n_0}=\lim_{k\to \infty}\frac{-k\chi(S)-\chi(S_0)}{kn+n_0}= \frac{-\chi(S)}{n}$$
		by construction.
	\end{proof}
	
	\begin{corollary}\label{cor: simple with chi01}
		If $A$ and $B$ are finite, then 
		$$\stl_G(g)=\inf_S \frac{-\chi(S)}{n},$$
		where the infimum is taken over all connected simple surfaces with $\chi(\Gamma_S)\in\{0,1\}$.
	\end{corollary}
	\begin{proof}
		By Lemmas \ref{lemma: ample} and \ref{lemma: approximation}, the complexity of a connected surface with $\chi(\Gamma_S)=0$ can be approximated by those with gluing graph being a tree. Thus the infimum remains the same if we restrict the class of surfaces to connected simple surfaces with $\chi(\Gamma_S)=1$. Then the result follows from Corollary \ref{cor: simple surface is enough}.
	\end{proof}
	
	\begin{remark}
	    The equality still holds if we consider simple surfaces where each component $\Sigma$ has $\chi(\Gamma_\Sigma)\ge0$. Note that this is different from Lemma \ref{lemma: linprog estimate}, which restricts the Euler characteristic of the gluing graph overall instead of component-wise.
	\end{remark}

	For what follows, we will only apply rewiring to two edges in the same component of the gluing graph, particularly in the following three scenarios.
	
	The first scenario is when we have a simple surface $S$ whose gluing graph $\Gamma_S$ is a tree such that there is an embedded oriented path $P$ that starts and ends with two distinct edges of the same type and orientation. 
	Then applying rewiring to these two edges decomposes the simple surface into the union of two connected simple surfaces $S_1$ and $S_2$, where the gluing graph of $S_2$ is still a tree, the gluing graph of $S_1$ has $\chi(\Gamma_{S_1})=0$ and the path $P$ becomes of the core of $\Gamma_{S_1}$. See the left of Figure \ref{fig: rewiretypes}. For later reference, we refer to this as \emph{rewiring of type \RNum{1}}.
	
	The second scenario is when we have a connected simple surface $S$ with $\chi(\Gamma_S)=0$ such that on a decorative tree $T$ there is an embedded oriented path $P$ that starts and ends with two distinct edges $e_1,e_2$ of the same type and orientation, and the unique path connecting $e_2$ to the root of $T$ contains $P$. Then applying rewiring to these two edges decomposes the simple surface into the union of two connected simple surfaces $S_1$ and $S_2$, where $\chi(\Gamma_{S_1})=\chi(\Gamma_{S_2})=0$, the core of $\Gamma_{S_1}$ is inherited from $\Gamma_S$, and the core of $\Gamma_{S_2}$ comes from the path $P$. See the middle of Figure \ref{fig: rewiretypes}. We refer to this as \emph{rewiring of type \RNum{2}}.
	
	The last scenario is when we have a connected simple surface $S$ with $\chi(\Gamma_S)=0$ such that, for a fixed orientation of the core of $\Gamma_S$ as a circle, there are two oriented edges on the core of the same type and orientation. 
	Then applying rewiring to these two edges decomposes the core into two disjoint circles, and accordingly
	breaks the simple surface into the union of two connected simple surfaces $S_1,S_2$ with $\chi(\Gamma_{S_1})=\chi(\Gamma_{S_2})=0$ such that their cores are the two circles above. See the right of Figure \ref{fig: rewiretypes}. We refer to this as \emph{rewiring of type \RNum{3}}.
	
	\begin{figure}
		\centering
		\labellist
		\small \hair 2pt
		\pinlabel $\text{type \RNum{1}}$ at 90 120
		\pinlabel $\text{type \RNum{2}}$ at 290 120
		\pinlabel $\text{type \RNum{3}}$ at 510 120
		\endlabellist
		\includegraphics[scale=0.65]{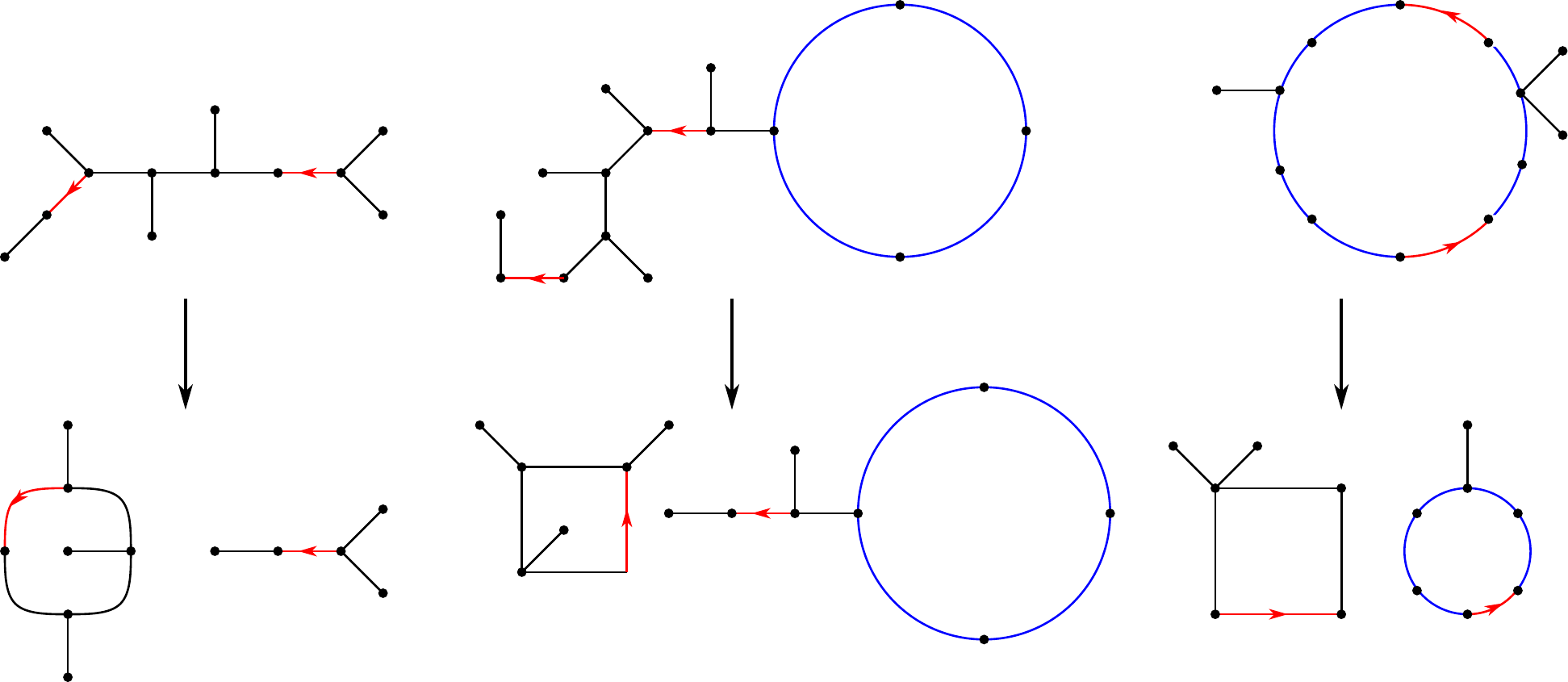}
		\caption{In terms of gluing graphs, this shows the effect of the three types of rewiring applied to the two red edges (marked with arrows) in each graph on top that have the same type and orientation.}
		\label{fig: rewiretypes}
	\end{figure}
	
	\subsection{Irreducible simple surfaces}\label{subsec: irreducible}
	Now we introduce irreducible simple surfaces and show how each connected simple surface $S$ with $\chi(\Gamma_S)\ge0$ decomposes into a disjoint union of irreducible ones by a sequence of splitting and rewiring.
	
	\begin{definition}\label{def: irreducible}
		A connected simple surface $S$ is \emph{irreducible} if $\chi(\Gamma_S)\ge0$, no splitting can be applied to any piece of $S$, and no rewiring of type \RNum{1}, \RNum{2} or \RNum{3} can be applied.
	\end{definition}
	
	\begin{proposition}\label{prop: reduce}
		For any simple surface $S$ with finitely many components such that each component $\Sigma$ has $\chi(\Gamma_\Sigma)\ge0$, there is a sequence of splittings and rewiring of types \RNum{1}, \RNum{2}, or \RNum{3} that modifies $S$ into a disjoint union $S'$ of irreducible simple surfaces. Moreover, 
		there is a component $\Sigma'$ of $S'$ satisfying
		$$\frac{-\chi(\Sigma')}{n(\Sigma')}\le\frac{-\chi(S)}{n(S)},$$
		where $n(\Sigma')$ and $n(S)$ are the degrees of $\Sigma'$ and $S$ respectively.
	\end{proposition}
	\begin{proof}
		Let $\kappa(S)=2e-2c+\ell$, where $e$ is the number of edges in $\Gamma_S$, $c$ is the number of components of $\Gamma_S$, and $\ell$ is the number of embedded loops in $\Gamma_S$. Equivalently, $\ell$ is the number of components of $\Gamma_S$ that have Euler characteristic zero. Note that $\kappa(S)$ is a non-negative integer since each component contains at least one edge.
		
		For the first assertion, it suffices to check that whenever we apply splitting or rewiring of type \RNum{1}, \RNum{2} or \RNum{3} to modify $S$ into another simple surface $S'$, we have $\kappa(S')<\kappa(S)$.
		Note that both splitting and rewiring leave the number of edges invariant. So it comes down to checking how $-2c+\ell$ varies in each situation.
		
		If splitting is applied to a component $\Sigma$ of $S$, either it breaks $\Gamma_\Sigma$ into two components without changing the number of embedded loops, or $\chi(\Gamma_\Sigma)=0$ and it breaks the core of $\Sigma$ without creating new components. Thus for the simple surface $S'$ obtained this way, we have either $\kappa(S')=\kappa(S)-2$ or $\kappa(S')=\kappa(S)-1$ corresponding to these two cases.
		
		If we apply rewiring of type \RNum{1} to a component $\Sigma$, then the tree $\Gamma_\Sigma$ breaks into two components, one of which contains a loop. Thus $\kappa(S')=\kappa(S)-1$. For rewiring of type \RNum{2}, we break the graph $\Gamma_\Sigma$ into two components each containing a loop, where one of loop is inherited from the core of $\Gamma_\Sigma$. Hence we get one more component and one more loop, yielding $\kappa(S')=\kappa(S)-1$. As for rewiring of type \RNum{3}, we also get one more component and one more loop. Thus for all the three types of rewiring we have $\kappa(S')=\kappa(S)-1$.
		
		For the second assertion, by Lemma \ref{lemma: chi formula via graph}, the (total) Euler characteristic of the simple surface does not change when we apply rewiring, and it increases by $1$ every time we apply splitting since we obtain one more vertex representing a disk piece.
		In addition, both operations do not change the total degree.
		Suppose we start with $S$ which has degree $n(S)$, and the irreducible simple surfaces we obtain in the end are $\Sigma'_1,\cdots, \Sigma'_k$ with degrees $n(\Sigma'_1),\cdots,n(\Sigma'_k)$ respectively. Then we have
		$$\frac{-\chi(S)}{n(S)}\ge \frac{\sum_{i=1}^k -\chi(\Sigma'_i)}{\sum_{i=1}^k n(\Sigma'_i)}\ge\min_{1\le i\le k}\frac{-\chi(\Sigma'_i)}{n(\Sigma'_i)},$$
		where the second inequality holds since the term in the middle is a weighted average. Thus the second assertion holds by taking $\Sigma'=\Sigma'_i$  where $\Sigma'_i$ achieves the minimum above.
	\end{proof}
	
	Now we bound the size of the gluing graph of any irreducible simple surface to show that there are only finitely many such surfaces for the given element $g=a_1b_1\cdots a_L b_L$. Note that there are $L^2$ possible types of turns on each side, giving rise to $L^2$ types of edges in gluing graphs.
	\begin{lemma}\label{lemma: valence bound}
		If the factor groups $A$ and $B$ are finite and $S$ is an irreducible simple surface for $g$, then the valence of each vertex of $\Gamma_S$ is at most $L^2\max\{|A|,|B|\}$.
	\end{lemma}
	\begin{proof}
		Suppose there is a vertex with valence greater than $L^2\max\{|A|,|B|\}$. Then we can apply splitting to the corresponding piece by Lemma \ref{lemma: splitting applicable}, which contradicts the assumption that $S$ is irreducible.
	\end{proof}
	
	\begin{lemma}\label{lemma: diameter bound for trees}
		If $S$ is an irreducible simple surface for $g$ with $\chi(\Gamma_S)=1$, then the diameter of $\Gamma_S$ is at most $2L^2$.
	\end{lemma}
	\begin{proof}
		Suppose the diameter of $\Gamma_S$ is greater than $2L^2$. Then there is an embedded path $P$ of length at least $2L^2$, which we orient. Since there are at most $L^2$ types of edges in $\Gamma_S$, each with two possible orientations, there are two edges on $P$ that have the same type and orientation by the pigeonhole principle. Hence rewiring of type \RNum{1} is applicable, contradicting that $S$ is irreducible.
	\end{proof}
	
	\begin{lemma}\label{lemma: diameter bound for rooted trees}
		If $S$ is an irreducible simple surface for $g$ with $\chi(\Gamma_S)=0$, then for any decorative tree $T$ of $\Gamma_S$, the distance from any vertex of $T$ to its root is at most $2L^2$. In particular, the diameter of $T$ is at most $4L^2$.
	\end{lemma}
	\begin{proof}
		If some vertex has distance more than $2L^2$ to the root, the geodesic connecting them contains more than $2L^2$ edges. So by the same argument as in the proof of Lemma \ref{lemma: diameter bound for trees}, rewiring of type \RNum{2} is applicable, contradicting that $S$ is irreducible.
	\end{proof}
	
	\begin{lemma}\label{lemma: diameter bound for core}
		If $S$ is an irreducible simple surface with $\chi(\Gamma_S)=0$, then the core of $\Gamma_S$ has length at most $2L^2$.
	\end{lemma}
	\begin{proof}
		If the core has length greater than $2L^2$, the same pigeonhole principle shows that rewiring of type \RNum{3} is applicable, contradicting that $S$ is irreducible.
	\end{proof}
	
	\begin{proposition}\label{prop: finitely many irreducible}
		If $G=A*B$, where $A$ and $B$ are finite groups, then for any $g$ not conjugate into the factor groups, there are only finitely many irreducible simple surfaces.
	\end{proposition}
	\begin{proof}
		By Lemmas \ref{lemma: diameter bound for trees}--\ref{lemma: diameter bound for core}, the diameter of the gluing graph $\Gamma_S$ of any irreducible simple surface is bounded above (by $5L^2$). Moreover, the valence of each vertex is bounded above by $L^2\cdot \max\{|A|,|B|\}$. Thus there are only finitely many possible gluing graphs. Since there are finitely many types of edges, and the types of edges around a vertex with a chosen cyclic order determines the type of the corresponding piece, we conclude that there are only finitely many possible irreducible simple surfaces.
	\end{proof}
	
	\begin{theorem}\label{thm: rationality}
		For $G=A*B$ and $g=a_1b_1\cdots a_L b_L$ where each $a_i\in A\notin\{id\}$ and $b_i\in B\notin\{id\}$. If the subgroups generated by $\{a_1,\cdots, a_L\}$ and $\{b_1,\cdots,b_L\}$ respectively are both finite, then there is an irreducible simple surface $S$ of some degree $n(S)$ with $\chi(\Gamma_S)=0$ such that
		$$\stl_G(g)=-\frac{\chi(S)}{n(S)}.$$
		As a consequence, $\stl_G(g)$ is rational and computable.
	\end{theorem}
	\begin{proof}
	By the isometric embedding Theorem \ref{thm: isometric embedding} (and its proof), we may replace $A$ and $B$ by the subgroups generated by $\{a_1,\cdots, a_L\}$ and $\{b_1,\cdots,b_L\}$ respectively. Thus we will assume $A$ and $B$ to be finite.
	
		By Corollary \ref{cor: simple with chi01} and Proposition \ref{prop: reduce}, we know 
		$$\stl_G(g)=\inf_S\frac{-\chi(S)}{n(S)},$$
		where the infimum is taken over all irreducible simple surfaces $S$ and $n(S)$ is the degree of $S$.
		By Proposition \ref{prop: finitely many irreducible}, there are only finitely many irreducible simple surfaces. Hence one of them achieves the infimum above. 
		Thus $\stl_G(g)$ is rational and can be computed by enumerating the finitely many irreducible simple surfaces.
		
		It remains to observe that the infimum cannot be achieved by a simple surface where $\Gamma_S$ is a tree.
		For any such simple surface $S$ of degree $n(S)$, we have at least one annulus-piece $C$ since $g$ is not a torsion element.
		Suppose the polygonal boundary of $C$ represents a $k$-torsion element, where $k\ge2$. Then there is a simple surface $S'$ of degree $n(S')=k\cdot n(S)$ such that $\Gamma_{S'}\setminus\{v'\}$ is formed by $k$ disjoint copies of $\Gamma_S\setminus\{v\}$, where $v$ represents the piece $C$ and $v'$ represents a disk-piece $C'$ whose polygonal boundary is a degree $k$ cover of the polygonal boundary of $C$; see Figure \ref{fig: branchedcover}.
		By Lemma \ref{lemma: chi formula via graph}, we see that
		$$\frac{-\chi(S')}{n(S')}=\frac{ke-(kd+1)}{k\cdot n(S)}<\frac{e-d}{n(S)}=\frac{-\chi(S)}{n(S)},$$
		where $e$ is the number of edges in $\Gamma_S$ and $d$ the number of disk-pieces in $S$.
		This shows that no connected simple surface with $\chi(\Gamma_S)=1$ can achieve the minimal complexity.
	\end{proof}
	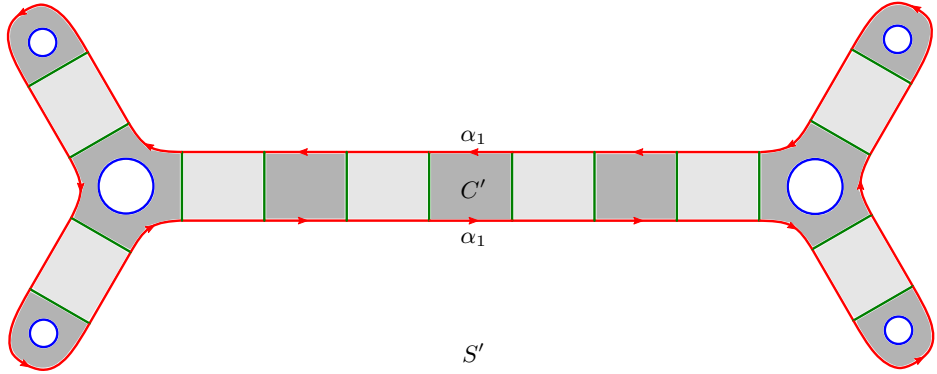
\begin{figure}
		\centering
		\labellist
		\small \hair 2pt
		\pinlabel $\alpha_1$ at 255 128
		\pinlabel $C'$ at 255 100
		\pinlabel $\alpha_1$ at 255 72
		\pinlabel $S'$ at 255 10
		\endlabellist
		\includegraphics[scale=0.7]{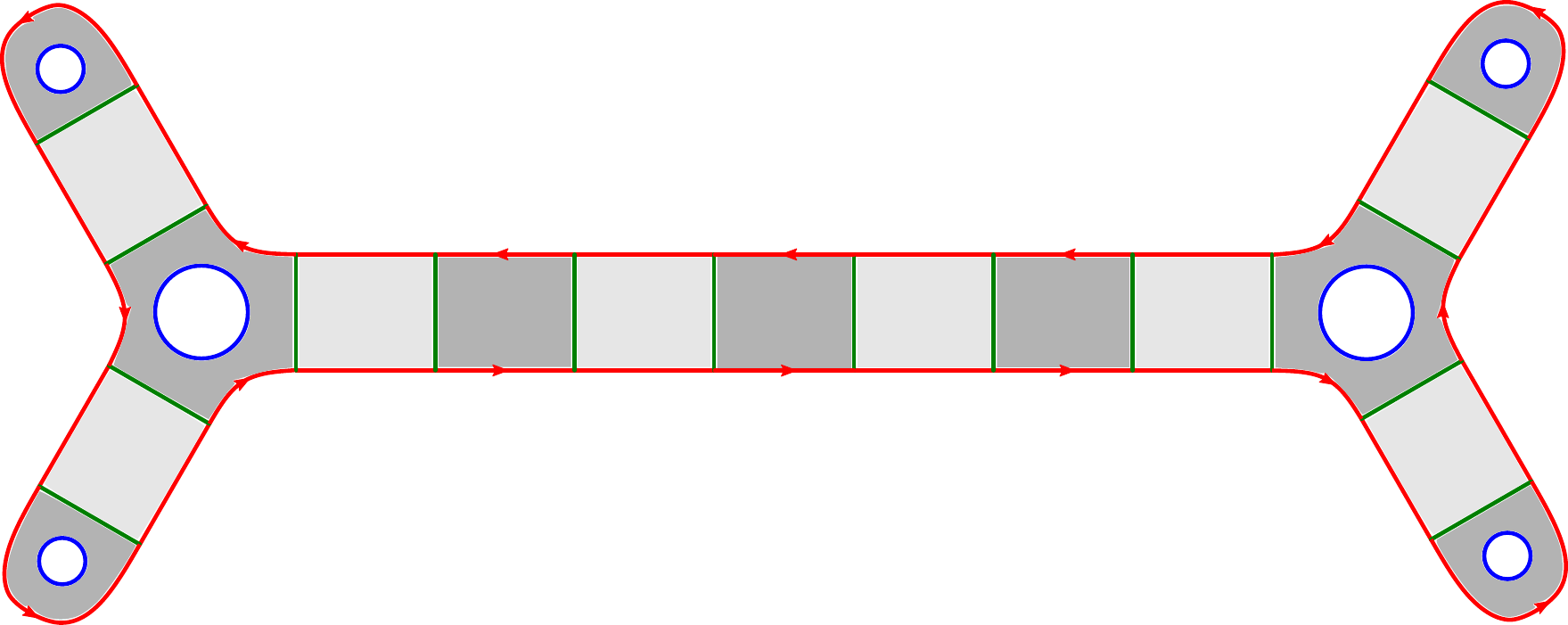}
		\caption{If $\alpha_1$ represents a $2$-torsion in the simple surface in Figure \ref{fig: simple_S_A}, then this is the ``branched'' degree $k=2$ cover that we construct with lower complexity, where we pick the piece $C$ as the leftmost piece in Figure \ref{fig: simple_S_A}.}
		\label{fig: branchedcover}
	\end{figure}

	Theorem \ref{introthm: rationality} follows from the theorem above.
	\begin{proof}[Proof of Theorem \ref{introthm: rationality}]
		Suppose $G=A*B$ with $A$ and $B$ both finite. If $g$ does not conjugate into $A$ or $B$, then by a suitable conjugation the result follows from Theorem \ref{thm: rationality}. If $g$ conjugates into $A$ or $B$, by finiteness of factor groups we know $\stl(g)=0$ since $\tl(g^n)\le 1$ for all $n$.
		
		The same analysis and argument work for free products of arbitrarily many factors (possibly infinitely many) without any difficulty, so the general case of Theorem \ref{introthm: rationality} also holds.

	\end{proof}		
		
	\begin{corollary}\label{cor: intermediate families}
		Let $G=A*B$ be a free product, where $A$ and $B$ are finite groups, and let $g\in G$ be an element not conjugate into the factor groups. 
		Suppose $\mathcal{S}$ is a family of simple surfaces for $g$ with the following properties:
		\begin{enumerate}
			\item For any simple surface in $\mathcal{S}$, each component $S$ has $\chi(\Gamma_S)\in\{0,1\}$.
			\item $\mathcal{S}$ contains all irreducible simple surfaces $S$ with $\chi(\Gamma_S)=0$.
		\end{enumerate}
		 Then
		$$\stl_G(g)=\inf_{S\in \mathcal{S}}\frac{-\chi(S)}{n}.$$
	\end{corollary}
	\begin{proof}
		This is a combination of Corollary \ref{cor: simple with chi01} and Theorem \ref{thm: rationality}.
	\end{proof}
	
	\begin{remark}\label{remark: complexity}
		In Corollary \ref{cor: intermediate families}, one can take $\mathcal{S}$ to be the subfamily of simple surfaces for $g$ where each component satisfies the valence and diameter bounds (of decorative trees) in Lemmas \ref{lemma: diameter bound for trees}--\ref{lemma: diameter bound for core}. The number of simple surfaces in $\mathcal{S}$ is around the order of $L^{{(C L)}^{4L^2}}$, where $C=\max\{|A|,|B|\}$ and $2L$ is the (cyclically reduced) word length of $g$. Thus enumerating	such surfaces gives an algorithm to compute $\stl_G(g)$ with terrible complexity.
		
		However, one should be able to parameterize such surfaces as integer vectors in a rational polyhedral cone, where each variable represents the number of some type of small building blocks used in the surface.
		Projectively, such surfaces are represented by rational points in a compact rational polyhedron.
		This gives a way to use linear programming to compute $\stl_G(g)$ when $G=A*B$ is a free product of finite abelian groups. When $A$ and $B$ are finite cyclic groups, using a setup similar to the one in \cite{Alden}, the numbers of variables and constraints in the linear programming problem are polynomial in $|A|, |B|, |g|$. Thus for a free product $G$ of finite cyclic groups, $\stl_G(g)$ can be computed in polynomial time by \cite{linprog}.
	\end{remark}

	\subsection{Generalizations to other factor groups}\label{subsec: generalization}
	We briefly discuss how one may generalize this method to allow more general factor groups, e.g. infinite ones. We explain where we essentially used the finiteness of the factor groups.
	
	The rewiring operation does not rely on the structure of factor groups at all, and the bounds on the diameter of the gluing graph $\Gamma_S$ of irreducible simple surfaces $S$ (Lemmas \ref{lemma: diameter bound for trees}, \ref{lemma: diameter bound for rooted trees} and \ref{lemma: diameter bound for core}) only depend on the word length of the target element.
	
	We used finiteness of the factor groups in the approximation of connected simple surfaces $S$ with $\chi(\Gamma_S)=0$ by those with $\chi(\Gamma_S)=1$ (Lemma \ref{lemma: approximation}), but it is not essential. For general factor groups, one should modify the definition of simple surfaces by further requiring pieces to only contain \emph{admissible turns}, which are those turns that appear in some connected simple surface $S$ with $\chi(\Gamma_S)=1$.
	Then the assumption of Lemma \ref{lemma: approximation} holds automatically, avoiding Lemma \ref{lemma: ample} (where we used finiteness of factor groups), while $\stl_G(g)$ can still be computed by looking at connected simple surfaces $S$ (using the modified definition) with $\chi(\Gamma_S)\in\{0,1\}$ (i.e. Corollary \ref{cor: simple with chi01}), under either assumptions of Corollary \ref{cor: simple surface is enough}.
	
	Provided that we can uniformly bound the valence of vertices in $\Gamma_S$ for any irreducible simple surface $S$ (Lemma \ref{lemma: valence bound}), we still get an algorithm to compute $\stl_G(g)$ and rationality.
	This requires a replacement or improvement of Lemma \ref{lemma: splitting applicable} that does not require the factor groups to be finite. Even if the factor groups are (infinite) torsion groups, it is not clear if Lemma \ref{lemma: splitting applicable} generalizes to that case. This lemma also appears to fail when the factor groups are infinite abelian groups.
	However, it might be possible to obtain a uniform bound on the valance in a different way, since using a piece with an enormous number of turns seems inefficient and might be ruled out a priori.

	\section{Examples} \label{sec: examples}
	In this section we explicitly compute the stable torsion length in two different examples, for $g=aba^{-1}b^{-1}$ and $g=ab$ in a free product $G=A*B$ with torsion elements $a\in A$ and $b\in B$. In both examples, we first work out a sufficient collection (Definition \ref{def: sufficient collection}) of types of pieces, then use the linear programming problem as in Section \ref{subsec: linprog lower bound} to compute a lower bound of $\stl_G(g)$, and finally use the Approximation Lemma \ref{lemma: approximation} to show that $\stl_G(g)$ actually equals the lower bound.
	
	\subsection{The word $[a,b]$}
	In this section we consider the word $g=[a,b]=aba^{-1}b^{-1}$ in a free product $G=A*B$, where $a\in A$ and $b\in B$ are torsion elements of orders $p,q\ge2$. We will focus on the special case where $A=\Z/p$ is generated by $a$ and $B=\Z/q$ is generated by $b$. The general case will follow from the isometric embedding Theorem \ref{thm: isometric embedding}.
	
	Using the setup in Section \ref{sec: free prod}, the word $g=[a,b]$ is represented by a loop $\gamma$ consisting of arcs $\alpha_1,\beta_1,\alpha_2,\beta_2$, where $\alpha_1,\alpha_2$ represent $a,a^{-1}$ and $\beta_1,\beta_2$ represent $b,b^{-1}$ respectively. Then there are four types of turns on the $A$-side: $(\alpha_1,\alpha_1)$, $(\alpha_1,\alpha_2)$, $(\alpha_2,\alpha_1)$, and $(\alpha_2,\alpha_2)$.
	
	Note that the two turns $(\alpha_1,\alpha_2)$ and $(\alpha_2,\alpha_1)$ form a polygonal boundary that bounds a disk-piece, which we denote by $R_1$. Moreover, on the polygonal boundary of any piece, the number of copies of $(\alpha_1,\alpha_2)$ is equal to that of $(\alpha_2,\alpha_1)$ since the polygonal boundary closes up. Denote this number in a piece $C$ by $N(C)\ge0$.
	
	If $N(C)\ge2$ for a piece $C$, then we can remove a copy of $(\alpha_1,\alpha_2)$ and $(\alpha_2,\alpha_1)$ so that the remaining turns still form a polygonal boundary.
	Thus we can apply splitting (the general form that works for abelian factor groups) to this piece $C$ to obtain a copy of $R_1$ and some piece $C'$ with $N(C')=N(C)-1$; see case (\ref{item: wild splitting}) of Example \ref{example: splitting}.
	
	If a piece $C$ has $N(C)=1$, then the arcs on the boundary of $C$ in the cyclic order must be $m$ copies of $\alpha_1$ followed by $n$ copies of $\alpha_2$ for some $m,n\ge1$. 
	\begin{enumerate}
	    \item If $m=n$ then we have a disk-piece, which we denote by $R_n$. For $n>p$, we have $p$ consecutive copies of the turn $(\alpha_1,\alpha_1)$ (resp. $(\alpha_2,\alpha_2)$) on the boundary of $R_n$. Thus we can apply splitting twice to reduce $R_n$ to $R_{n-p}$ together with two pieces $P_p^+$ and $P_p^-$,
	    where $P_p^+$ (resp. $P_p^-$) is the disk-piece with exactly $p$ copies of $\alpha_1$ (resp. $\alpha_2$) on the boundary. See case (\ref{item: geom splitting}) of Example \ref{example: splitting} for one of the splitting when $p=2$.
	    \item If $m>n$ then we can apply splitting to reduce the piece $C$ into a copy of $R_n$ and a piece with $m-n$ copies of $\alpha_2$ on the boundary. Similarly for the case $m<n$.
	\end{enumerate}
	
	Thus for any simple surface $S$, after applying splittings as above, we may assume that any piece $C$ other than $R_n$ with $1\le n\le p$ has $N(C)=0$. Thus any piece different from $R_n$ only contains one type of turns, either $(\alpha_1,\alpha_1)$ or $(\alpha_2,\alpha_2)$.
	Moreover, since $a$ has order $p$, splitting applies to any piece with more than $p$ copies of the turn $(\alpha_1,\alpha_1)$ (resp. $(\alpha_2,\alpha_2)$) on the boundary.
	
	Thus there are $3p$ types of remaining pieces on the $A$-side, which fall into three classes: 
	\begin{enumerate}
		\item A piece with $1\le n\le p$ arcs $\alpha_1$ on the boundary, which is a disk only when $n=p$. Denote such pieces as $P_n^+$; see the left on the first row of Figure \ref{fig: possible_pieces}.
		\item A piece with $1\le n\le p$ arcs $\alpha_2$ on the boundary, which is a disk only when $n=p$. Denote such pieces by $P_n^-$; see the left on the second row of Figure \ref{fig: possible_pieces}.
		\item A disk-piece $R_n$ for $1\le n\le p$ that has $n$ copies of $\alpha_1$ followed by $n$ copies of $\alpha_2$ on the boundary; see the left on the third row of Figure \ref{fig: possible_pieces}.
	\end{enumerate}
	
	Similarly, we can reduce simple surfaces by splitting on the $B$-side so that there are $3q$ types of remaining pieces on the $B$-side, denoted as $Q_n^+$, $Q_n^-$, and $T_n$ for $1\le n\le q$, where $q$ is the order of $b$.
	
	Let $\mathcal{P}$ be the collection consisting of the above $3p$ pieces on the $A$-side and $3q$ pieces on the $B$-side, depicted in Figure \ref{fig: possible_pieces} for the case $p=4$ and $q=3$.
	
	\begin{figure}
		\centering
		\labellist
		\small \hair 2pt
		
		\pinlabel $a$ at 15 227
		\pinlabel $P_1^+$ at 5 173
		
		\pinlabel $a$ at 95 185
		\pinlabel $a$ at 95 220
		\pinlabel $P_2^+$ at 75 173
		
		\pinlabel $a$ at 177 181
		\pinlabel $a$ at 160 215
		\pinlabel $a$ at 198 215
		\pinlabel $P_3^+$ at 160 173
		
		\pinlabel $a$ at 260 237
		\pinlabel $a$ at 260 173
		\pinlabel $a$ at 228 205
		\pinlabel $a$ at 290 205
		\pinlabel $P_4^+$ at 230 173
		
		\pinlabel $b$ at 325 225
		\pinlabel $Q_1^+$ at 320 173
		
		\pinlabel $b$ at 407 183
		\pinlabel $b$ at 407 220
		\pinlabel $Q_2^+$ at 390 173
		
		\pinlabel $b$ at 490 181
		\pinlabel $b$ at 471 215
		\pinlabel $b$ at 510 215
		\pinlabel $Q_3^+$ at 470 173
		
		\pinlabel $a^{-1}$ at 15 135
		\pinlabel $P_1^-$ at 5 83
		
		\pinlabel $a^{-1}$ at 97 93
		\pinlabel $a^{-1}$ at 95 132
		\pinlabel $P_2^-$ at 75 83
		
		\pinlabel $a^{-1}$ at 177 91
		\pinlabel $a^{-1}$ at 157 128
		\pinlabel $a^{-1}$ at 205 128
		\pinlabel $P_3^-$ at 150 83
		
		\pinlabel $a^{-1}$ at 265 150
		\pinlabel $a^{-1}$ at 262 83
		\pinlabel $a^{-1}$ at 223 118
		\pinlabel $a^{-1}$ at 298 118
		\pinlabel $P_4^-$ at 230 83
		
		\pinlabel $b^{-1}$ at 325 135
		\pinlabel $Q_1^-$ at 315 83
		
		\pinlabel $b^{-1}$ at 407 93
		\pinlabel $b^{-1}$ at 405 132
		\pinlabel $Q_2^-$ at 385 83
		
		\pinlabel $b^{-1}$ at 490 88
		\pinlabel $b^{-1}$ at 467 128
		\pinlabel $b^{-1}$ at 515 128
		\pinlabel $Q_3^-$ at 460 83
		
		\pinlabel $a$ at 3 32
		\pinlabel $a^{-1}$ at 45 35
		\pinlabel $R_1$ at 5 -8
		
		\pinlabel $a$ at 77 52
		\pinlabel $a$ at 77 12
		\pinlabel $a^{-1}$ at 119 55
		\pinlabel $a^{-1}$ at 119 15
		\pinlabel $R_2$ at 85 -8
		
		\pinlabel $a$ at 160 58
		\pinlabel $a$ at 144 32
		\pinlabel $a$ at 160 4
		\pinlabel $a^{-1}$ at 200 61
		\pinlabel $a^{-1}$ at 218 35
		\pinlabel $a^{-1}$ at 200 5
		\pinlabel $R_3$ at 155 -8
		
		\pinlabel $a$ at 247 62
		\pinlabel $a$ at 230 45
		\pinlabel $a$ at 230 19
		\pinlabel $a$ at 247 0
		\pinlabel $a^{-1}$ at 280 63
		\pinlabel $a^{-1}$ at 299 45
		\pinlabel $a^{-1}$ at 299 23
		\pinlabel $a^{-1}$ at 280 1
		\pinlabel $R_4$ at 235 -8
		
		\pinlabel $b$ at 317 32
		\pinlabel $b^{-1}$ at 358 33
		\pinlabel $T_1$ at 320 -8
		
		\pinlabel $b$ at 389 52
		\pinlabel $b$ at 389 12
		\pinlabel $b^{-1}$ at 430 53
		\pinlabel $b^{-1}$ at 430 13
		\pinlabel $T_2$ at 390 -8
		
		\pinlabel $b$ at 474 58
		\pinlabel $b$ at 459 32
		\pinlabel $b$ at 474 3
		\pinlabel $b^{-1}$ at 513 61
		\pinlabel $b^{-1}$ at 530 35
		\pinlabel $b^{-1}$ at 513 4
		\pinlabel $T_3$ at 465 -8
		
		\endlabellist
		\includegraphics[scale=0.7]{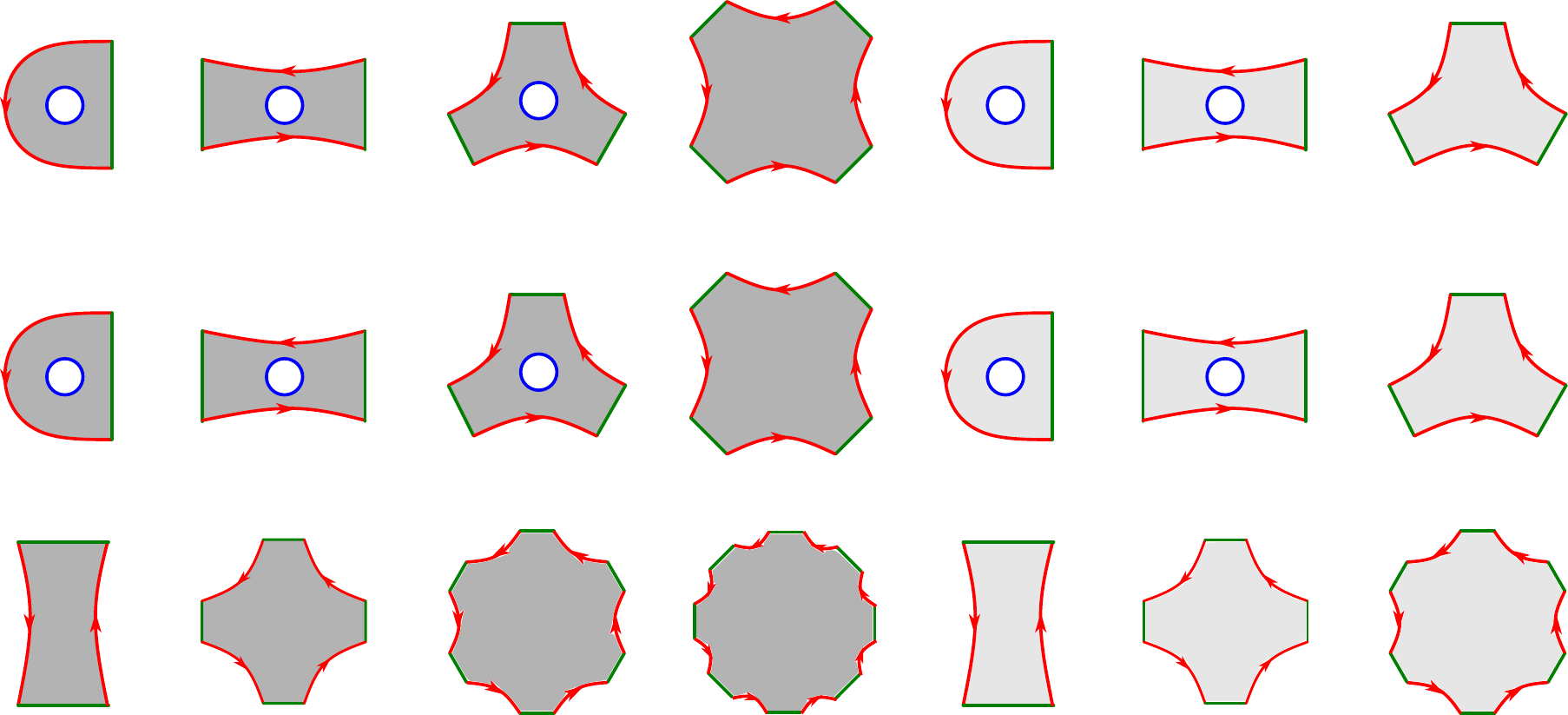}
		\caption{The pieces in the collection $\mathcal{P}$ when $p=4$ and $q=3$.} 
		\label{fig: possible_pieces}
	\end{figure}
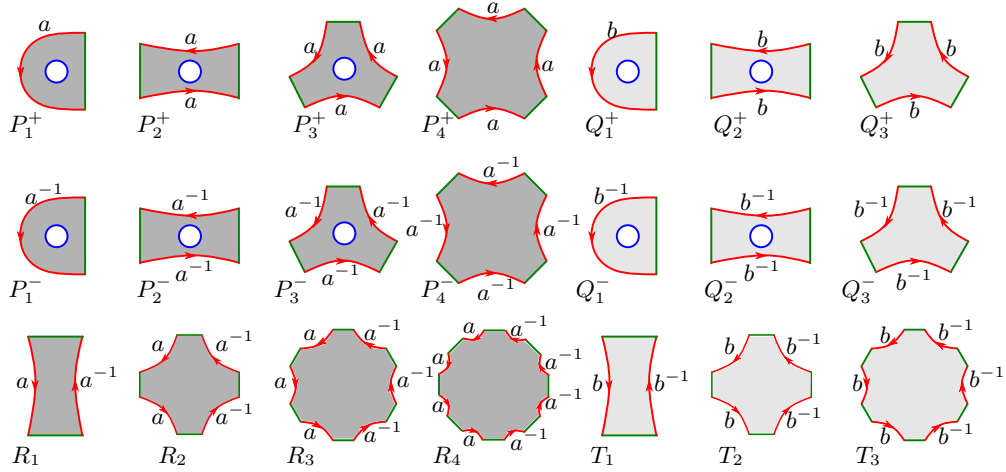
	
	\begin{lemma}\label{lemma: sufficient collection for [a,b]}
		The collection $\mathcal{P}$ is sufficient.
	\end{lemma}
	\begin{proof}
		The discussion above shows that any simple surface can be reduced to a $\mathcal{P}$-simple surface by a sequence of splittings. 
		Note that splitting preserves
		the number of edges in the gluing graph and adds a vertex representing a disk piece, thus it decreases $-\chi(S)$ by Lemma \ref{lemma: chi formula via graph}.
		Thus if we start with a connected simple surface $S$ with $\chi(\Gamma_S)=1$,
		applying splitting once (if applicable) modifies it into a $\mathcal{P}$-simple surface $S'$ of the same degree with $-\chi(S')<-\chi(S)$. Moreover, the resulting gluing graph $\Gamma_{S'}$ has two components, each of which is a tree. Hence one of the two components has lower complexity than the original one.
		Thus the infimum of $-\chi(S)/n(S)$ over all connected simple surfaces with $\chi(\Gamma_S)=1$ does not change as we restrict to connected $\mathcal{P}$-simple surfaces with $\chi(\Gamma_S)=1$. Hence $\mathcal{P}$ is sufficient by definition.
	\end{proof}
	
	Now we can apply the formalism in Section \ref{subsec: linprog lower bound} to compute a lower bound of $\stl_G([a,b])$ by linear programming. The lower bound turns out to be sharp in this case. The following elementary observation is helpful to simplify our computation.
	
	\begin{lemma}\label{lemma: reduction}
		For any $n\ge2$ and any set of numbers $x_k\ge 0$, $1\le k\le n$, there is another set of numbers $x'_k\ge0$ such that 
		\begin{enumerate}
			\item $\sum_{k=1}^n k x'_k=\sum_{k=1}^n k x_k$;\label{item: weighted sum inv}
			\item $\sum_{k=1}^n x'_k = \sum_{k=1}^n x_k$;\label{item: sum inv}
			\item $x'_n\ge x_n$; and\label{item: last term increase}
			\item $x'_k=0$ for all $1<k<n$.
		\end{enumerate}
	\end{lemma}
	\begin{proof}
		If $x_i>0$ for some $1<i<n$, we construct non-negative numbers $\{x'_k\}$ with $x'_i=0$ and satisfying bullets (\ref{item: weighted sum inv}), (\ref{item: sum inv}) and (\ref{item: last term increase}). Let $\lambda=\frac{n-i}{n-1}\in (0,1)$ and $\mu=\frac{i-1}{n-1}\in (0,1)$. Then $\lambda+n\mu=i$ and $\lambda+\mu=1$. So the following set of numbers
		\begin{align*}
			x'_1&=x_1+\lambda x_i, & x'_n&=x_n+\mu x_i, &x'_i&=0, & \text{and } x'_k&=x_k\text{ for all }k\neq 1,i,n,
		\end{align*}
		satisfies (\ref{item: weighted sum inv}), (\ref{item: sum inv}) and (\ref{item: last term increase}). 
		Hence, the conclusion follows by a sequence of such changes by making one $x_i$ zero at a time.
	\end{proof}
	
	\begin{lemma}\label{lemma: compute abAB}
		If $A=\Z/p$ and $B=\Z/q$ are generated by $a,b$, then for $G=A*B$ we have
		$$\stl_G([a,b])=1-\frac{1}{\min(p,q)-1}.$$
	\end{lemma}
	\begin{proof}
		Let $\mathcal{P}$ be the sufficient collection above. Consider the polyhedron $C_{\mathcal{P}}$ defined in Section \ref{subsec: linprog lower bound}.
		Let $x_n^+$ (resp. $x_n^-$) be the coordinate corresponding to the piece $P_n^+$ (resp. $P_n^-$) for each $1\le n\le p$. Let $y_n^+$ (resp. $y_n^-$) be the coordinate corresponding to the piece $Q_n^+$ (resp. $Q_n^-$). Let $z_n$ and $w_n$ be the coordinates corresponding to the pieces $R_n$ and $T_n$ respectively.
		
		Then the gluing conditions in the definition of the polyhedron $C_{\mathcal{P}}$ as in Section \ref{subsec: linprog lower bound} become:
		\begin{align}
			\sum_{k=1}^p k x_k^+ + \sum_{k=1}^p (k-1) z_k = \sum_{k=1}^p k x_k^- + \sum_{k=1}^p (k-1) z_k &= \sum_{k=1}^q w_k, \label{eqn: gluing 1}\\
			\sum_{k=1}^q k y_k^+ + \sum_{k=1}^q (k-1) w_k = \sum_{k=1}^q k y_k^- + \sum_{k=1}^q (k-1) w_k &= \sum_{k=1}^p z_k. \label{eqn: gluing 2}
		\end{align}
		
		By counting the number of copies of the arc $\alpha_1$, the normalizing condition is
		$$\sum_{k=1}^p k x_k^+ + \sum_{k=1}^p k z_k=1.$$
		The left-hand side can be rewritten as $\sum_{k=1}^p k x_k^+ + \sum_{k=1}^p (k-1) z_k + \sum_{k=1}^p z_k$. 
		Thus by the gluing condition (\ref{eqn: gluing 1}) we can express the normalizing condition equivalently as
		\begin{equation}\label{eqn: normalizing}
		    \sum_{k=1}^p z_k+\sum_{k=1}^q w_k=1.    
		\end{equation}
		
		The Euler characteristic constraint $\chi_\Gamma\ge0$ can be written as
		$$\sum_{k=1}^p (1-k) z_k+\sum_{k=1}^q (1-k) w_k+\sum_{k=1}^p (1-\frac{k}{2})(x_k^+ + x_k^-)+\sum_{k=1}^q (1-\frac{k}{2})(y_k^+ + y_k^-)\ge0.$$
		Note that by the gluing condition (\ref{eqn: gluing 1}), we have
		\begin{align*}
		    &\sum_{k=1}^p (1-k) z_k + \sum_{k=1}^p (1-\frac{k}{2})(x_k^+ + x_k^-)\\
		=&\sum_{k=1}^p (x_k^+ + x_k^-) -\frac{1}{2}\left[\sum_{k=1}^p k x_k^+ + \sum_{k=1}^p (k-1) z_k\right] -\frac{1}{2}\left[\sum_{k=1}^p k x_k^- + \sum_{k=1}^p (k-1) z_k\right]\\
		=&\sum_{k=1}^p (x_k^+ + x_k^-) - \sum_{k=1}^q w_k.
		\end{align*}
		Similarly, we have
		$$\sum_{k=1}^q (1-k) w_k+\sum_{k=1}^q (1-\frac{k}{2})(y_k^+ + y_k^-)=\sum_{k=1}^q (y_k^+ + y_k^-) - \sum_{k=1}^p z_k.$$
		Using the normalizing condition (\ref{eqn: normalizing}), the constraint $\chi_\Gamma\ge0$ is equivalent to
		\begin{equation}\label{eqn: chigamma_ineq}
		    \sum_{k=1}^p (x_k^+ + x_k^-) + \sum_{k=1}^q (y_k^+ + y_k^-) \ge 1.
		\end{equation}
		
		The objective is to minimize $-\chi_o$, which is expressed as
		\begin{align*}
			\sum_{k=1}^p& (k-1) z_k+\sum_{k=1}^q (k-1) w_k +\left(\frac{p}{2}-1\right)(x_p^+ + x_p^-)+\left(\frac{q}{2}-1\right)(y_q^+ + y_q^-)\\
			&+\ \sum_{k=1}^{p-1} \frac{k}{2}(x_k^+ + x_k^-)+\sum_{k=1}^{q-1} \frac{k}{2} (y_k^+ + y_k^-)\\
			=\ & \frac{1}{2}\left[\sum_{k=1}^p k x_k^+ + \sum_{k=1}^p (k-1) z_k\right] +\frac{1}{2}\left[\sum_{k=1}^p k x_k^- + \sum_{k=1}^p (k-1) z_k\right]\\
			& +\ \frac{1}{2}\left[\sum_{k=1}^q k y_k^+ + \sum_{k=1}^q (k-1) w_k\right] +\frac{1}{2}\left[\sum_{k=1}^q k y_k^- + \sum_{k=1}^q (k-1) w_k\right]\\
			&-\ (x_p^+ + x_p^- + y_q^+ + y_q^-)\\
			=\ & \sum_{k=1}^q w_k+\sum_{k=1}^p z_k-(x_p^+ + x_p^- + y_q^+ + y_q^-)\\
			=\ & 1-(x_p^+ + x_p^- + y_q^+ + y_q^-),
		\end{align*}
		where we used the gluing conditions (\ref{eqn: gluing 1}) and (\ref{eqn: gluing 2}), and normalizing condition (\ref{eqn: normalizing}) at the last two steps, respectively.
		
		By Lemma \ref{lemma: reduction}, we may assume $x_k^\pm=0$ for all $1<k<p$ and $y_k^\pm=0$ for all $1<k<q$. 
		In addition, if we let $x_k^{+\prime}=x_k^{-\prime}=\frac{1}{2}(x_k^+ + x_k^-)$ for all $k$ without changing $z_k$ and $w_k$, the constraints (\ref{eqn: gluing 1})--(\ref{eqn: chigamma_ineq}) and the objective function are all unaffected.
		Thus we can assume $x_k^+=x_k^-$ for all $k$ and similarly for $y_k^\pm$.
		Hence the linear programming problem reduces to
		\begin{align*}
			\text{minimize:}\quad     &1-2x_p - 2y_q\\
			\text{subject to:}\quad    &x_1 + p x_p + \sum_{k=1}^p (k-1)z_k = \sum_{k=1}^q w_k\\
			&y_1 + q y_q + \sum_{k=1}^q (k-1)w_k = \sum_{k=1}^p z_k\\
			&\sum_{k=1}^q w_k+\sum_{k=1}^p z_k=1\\
			&x_1 + x_p + y_1 + y_q \ge \frac{1}{2}\\
			&x_i,y_i,z_k,w_k\ge0,
		\end{align*}
		where $x_i=x_i^\pm$ ($i=1$ or $p$), $y_i=y_i^\pm$ ($i=1$ or $q$), and the first four constraints correspond to (\ref{eqn: gluing 1})--(\ref{eqn: chigamma_ineq}).
		
		By symmetry, assume $p\le q$. On the one hand, note that by the first two constraints and the fact that $w_k,z_k\ge0$, we have
		$$x_1 + x_p\le \sum_{k=1}^q w_k-(p-1)x_p,\quad\text{and}\quad y_1+y_q\le \sum_{k=1}^p z_k-(q-1)y_q.$$
		Thus using the third and fourth constraints, we get
		$$\frac{1}{2}\le x_1 + x_p + y_1 + y_q\le\sum_{k=1}^q w_k-(p-1)x_p+\sum_{k=1}^p z_k-(q-1)y_q\le 1-(p-1)(x_p+y_q),$$
		which implies $2(x_p+y_q)\le\frac{1}{p-1}$ and thus the objective $1-2(x_p+y_q)\ge 1-\frac{1}{p-1}$.
		On the other hand, this lower bound $1-\frac{1}{p-1}$ is achieved by the feasible solution $x_1=\frac{1}{2}(1-\frac{1}{p-1})$, $x_p=\frac{1}{2(p-1)}$, $y_1=y_q=0$, $w_1=1$, $w_k=0$ for $k>1$, and $z_k=0$ for all $k$.
		Hence we conclude that the minimal value of the linear programming is $1-\frac{1}{\min(p,q)-1}$. Thus $$\stl_G([a,b])\ge1-\frac{1}{\min(p,q)-1}$$
		by Lemma \ref{lemma: linprog estimate}.
		
		Moreover, for the feasible solution above, let $n=2(p-1)$. Take $nx_1^+=nx_1^-=nx_1=p-2$ copies of $P_1^+$ and $P_1^-$, take $nx_p^+=nx_p^-=nx_p=1$ copy of $P_p^+$ and $P_p^-$, and take $nw_1=2(p-1)$ copies of $T_1$. Such pieces can be glued into a $\mathcal{P}$-simple surface $S$ of degree $n=2(p-1)$ that is connected and has $\chi(\Gamma_S)=0$; see Figure \ref{fig: surfex} for an example where $q\ge p=5$. Thus by Lemma \ref{lemma: approximation}, $1-\frac{1}{\min(p,q)-1}=-\chi(S)/n$ is the limit of complexities of a sequence of connected simple surfaces with $\chi_\Gamma=1$. This implies 
		$$\stl_G([a,b])\le 1-\frac{1}{\min(p,q)-1}$$
		by Corollary \ref{cor: simple surface is enough}.
		Combining the two parts we obtain the desired equality.
	\end{proof}
	
	\begin{figure}
		\centering
		\labellist
		\small \hair 2pt
		\pinlabel $a$ at 235 128
		\pinlabel $a^{-1}$ at 190 128
		\pinlabel $b$ at 210 102
		\pinlabel $b^{-1}$ at 208 152
		\pinlabel $a$ at 265 80
		\pinlabel $b$ at 290 40
		\pinlabel $a^{-1}$ at 335 0
		\pinlabel $b^{-1}$ at 337 50
		\pinlabel $a$ at 325 100
		\pinlabel $b$ at 370 108
		\pinlabel $a^{-1}$ at 420 145
		\pinlabel $b^{-1}$ at 378 150
		\pinlabel $a$ at 325 160
		\pinlabel $b$ at 330 205
		\pinlabel $a^{-1}$ at 343 255
		\pinlabel $b^{-1}$ at 290 220
		\pinlabel $a$ at 265 175
		\pinlabel $b$ at 210 195
		\pinlabel $a^{-1}$ at 155 177
		\pinlabel $b^{-1}$ at 133 215
		\pinlabel $a$ at 115 250
		\pinlabel $b$ at 90 205
		\pinlabel $a^{-1}$ at 93 160
		\pinlabel $b^{-1}$ at 55 150
		\pinlabel $a$ at 10 145
		\pinlabel $b$ at 45 107
		\pinlabel $a^{-1}$ at 90 95
		\pinlabel $b^{-1}$ at 87 55
		\pinlabel $a$ at 113 5
		\pinlabel $b$ at 125 40
		\pinlabel $a^{-1}$ at 155 80
		\pinlabel $b^{-1}$ at 213 60
		\endlabellist
		\includegraphics[scale=0.7]{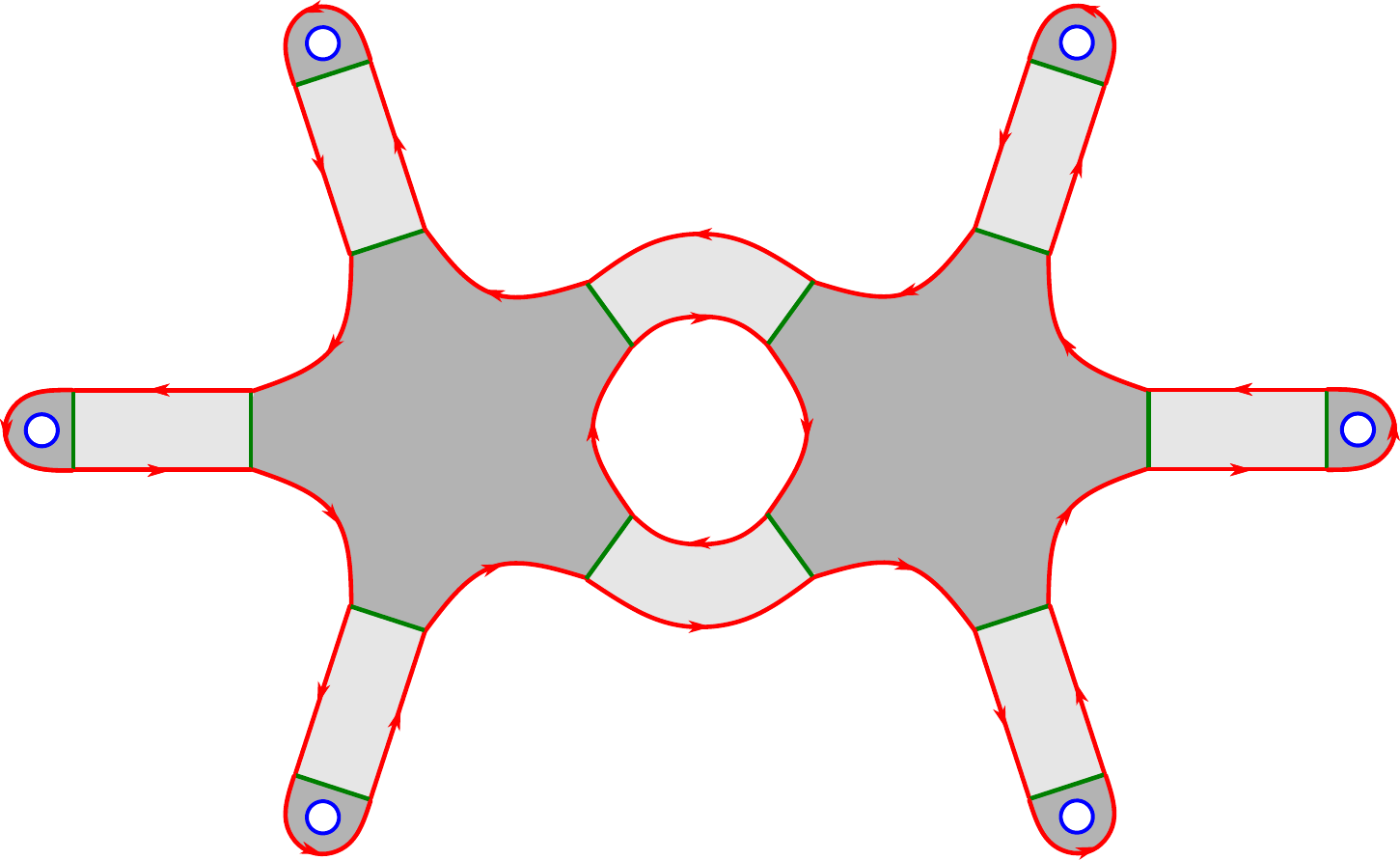}
		\caption{A connected $\mathcal{P}$-simple surface $S$ with the given number of pieces satisfying $\chi(\Gamma_S)=0$ in the case $q\ge p=5$.}
		\label{fig: surfex}
	\end{figure}
	
	Now we generalize this formula using Theorem \ref{thm: isometric embedding}.
	\begin{theorem}\label{thm: compute abAB}
	    Let $G=A*B$ be a free product, and let $a\in A$ and $b\in B$ be torsion elements of orders $p$ and $q$ respectively, where $p,q\ge2$. Then we have
	    $$\stl_G([a,b])=1-\frac{1}{\min(p,q)-1}.$$
	\end{theorem}
	\begin{proof}
	    Let $A'=\Z/p$ and $B'=\Z/q$ be the subgroups generated by $a$ and $b$ respectively. Then the inclusions $A'\to A$ and $B'\to B$ satisfy the assumption (\ref{item: assumption two}) of Theorem \ref{thm: isometric embedding} since $A'$ and $B'$ are finite. 
	    Thus the result directly follows from Lemma \ref{lemma: compute abAB}.
	\end{proof}
	
	\subsection{The word $ab$}
	In this section we consider the word $g=ab$ in a free product $G=A*B$, where $a\in A$ and $b\in B$ are torsion elements of orders $p,q\ge2$. We will focus on the special case where $A=\Z/p$ is generated by $a$ and $B=\Z/q$ is generated by $b$. The general case will follow from the isometric embedding Theorem \ref{thm: isometric embedding}.
	
	Using the setup in Section \ref{sec: free prod}, the word $g=ab$ is represented by a loop $\gamma$ consisting of arcs $\alpha$ and $\beta$, where $\alpha$ represents $a$ and $\beta$ represents $b$. There is exactly one type of turn on the $A$-side: $(\alpha,\alpha)$.
	
	Since $a$ has order $p$, there is a disk-piece with $p$ copies of the arc $\alpha$ on its polygonal boundary. Therefore, we can apply splitting to any $A$-piece with more than $p$ copies of the turn $(\alpha,\alpha)$ on the boundary. So, on the $A$-side, after splitting we are left with pieces with $1\leq n \leq p$ arcs $\alpha$ on the boundary. Furthermore, these pieces are disks only when $n=p$. We denote such pieces by $P_n$.  
	
	Similarly, we can reduce simple surfaces by splitting on the $B$-side to $q$ possible pieces, each with $1\leq n \leq q$ arcs $\beta$. These pieces are disks only when $n=q$. We denote such pieces by $Q_n$.
	
	Let $\mathcal{P}$ be the collection above consisting of these $p$ types of pieces on the $A$-side and these $q$ types of pieces on the $B$-side. The first row of Figure \ref{fig: possible_pieces} depicts such pieces when $p=4$ and $q=3$.
	
	\begin{lemma}
		The collection $\mathcal{P}$ is sufficient. 
	\end{lemma}
	
	\begin{proof}
		Since we reduced the collection of pieces to $\mathcal{P}$ by a series of splittings, the argument in Lemma \ref{lemma: sufficient collection for [a,b]} shows that $\mathcal{P}$ is sufficient.
	\end{proof}
	
	\begin{theorem}[Product formula]\label{thm: product formula}
		Let $a\in A$ and $b\in B$ be torsion elements of order $p$ and $q$ respectively such that $p\leq q$, then 
		$$\stl_G(ab)=\stl_{\Z/p*\Z/q}(ab)=1-\frac{q}{p(q-1)},$$
		where $\Z/p$ and $\Z/q$ are the subgroups generated by $a$ and $b$ respectively.
	\end{theorem}
	\begin{proof}
		The first equality follows from Theorem \ref{thm: isometric embedding}. So it suffices to compute $\stl_{\Z/p*\Z/q}(ab)$.
		
		Let $\mathcal{P}$ be the sufficient collection above. Consider the polyhedron $C_{\mathcal{P}}$ defined in Section \ref{subsec: linprog lower bound}. Let $x_n$ be the coordinate corresponding to the piece $P_n$ for each $1\leq n \leq p$. Let $y_n$ be the coordinate corresponding to the piece $Q_n$ for each $1 \leq n \leq q$. 
		
		The gluing condition in the definition of $C_\mathcal{P}$ becomes:
		$$\sum_{k=1}^p kx_k = \sum_{k=1}^q ky_k.
		$$
		
		The normalizing condition in the definition of $C_{\mathcal{P}}$ is: 
		$$\sum_{k=1}^p kx_k=1.
		$$
		
		The Euler characteristic constraint $\chi_\Gamma \geq 0$ can be written as: 
		
		$$ \sum_{k=1}^{p} x_k \left(1-\frac{k}{2}\right) + \sum_{k=1}^q y_k \left(1-\frac{k}{2}\right) \geq 0.
		$$
		
		Using the gluing and normalizing conditions, this is equivalent to: 
		
		$$\sum_{k=1}^p x_k + \sum_{k=1}^q y_k \geq 1.
		$$
		
		The objective is to minimize $-\chi_o$, which is expressed as: 
		\begin{align*}
			& \sum_{k=1}^{p-1}\frac{k}{2}\cdot x_k + \left(\frac{p}{2}-1\right)x_p + \sum_{k=1}^{q-1}\frac{k}{2}\cdot y_k + \left(\frac{q}{2}-1\right)y_q\\
			& = \frac{1}{2}\left(\sum_{k=1}^p k\cdot x_k +\sum_{k=1}^q k \cdot y_k \right) -x_p-y_q\\
			& = 1-x_p -y_q.
		\end{align*}
		
		Then, by Lemma \ref{lemma: reduction}, it is sufficient to assume that $x_k=0$ for $1<k<p$ and $y_j=0$ for $1<j<q$. This reduces the linear programming problem to: 
		\begin{align*}
			\text{minimize:}\quad     &1-x_p-y_q\\
			\text{subject to:}\quad &x_1+px_p=1\\
			&y_1+qy_q=1\\
			&x_1+x_p+y_1+y_q\geq 1\\
			& x_1, x_p, y_1, y_q \geq 0
		\end{align*}
		
		Then the constraints imply that
		$$2=x_1+y_1+px_p+qy_q\ge 1+(p-1)x_p+(q-1)y_q=1+(q-1)(x_p+y_q)-(q-p)x_p.$$
		Thus 
		$$(q-1)(x_p+y_q)\le 1+(q-p)x_p\le 1+\frac{q-p}{p}=\frac{q}{p},$$
		where we used the assumption that $q\ge p$ and the fact that $x_p\le 1/p$, which is a consequence of the first constraint since $x_1\ge0$.
		Hence it follows that the objective function satisfies
		$$1-(x_p+y_q)\ge 1-\frac{q}{p(q-1)}.$$
		
		
		
		
	
	    This lower bound is achieved by the feasible solution $x_1=0$, $x_p=\frac{1}{p}$, $y_1=\frac{pq-p-q}{p(q-1)}$, and $y_q=\frac{1}{p(q-1)}$. Therefore, by Lemma \ref{lemma: linprog estimate}, $\stl_{\Z/p*\Z/q}(ab)\geq 1-\dfrac{q}{p(q-1)}$. 
	    
	    Moreover, the solution above is (projectively) represented by a connected $\mathcal{P}$-simple surface $S$ of degree $p(q-1)$ with $\chi(\Gamma_S)=0$ in the following way, depicted in Figure \ref{fig: surfex2} in the case where $p=4$ and $q=5$. There is a single piece of type $Q_q$ in $S$, where $2$ out of the $q$ turns are glued with $2$ turns in a piece of type $P_p$, forming the unique embedded loop in $\Gamma_S$. The remaining $p-2$ turns of this piece of type $P_p$ are glued to $p-2$ pieces of type $Q_1$. As for the remaining $q-2$ turns of the unique piece of type $Q_q$, each of them is glued to a new piece of type $P_p$. For each of these $q-2$ new pieces of type $P_p$, the other $p-1$ turns are glued to a piece of type $Q_1$.
	    
	    Therefore, we have $\stl_{\Z/p*\Z/q}(ab)\leq 1-\dfrac{q}{p(q-1)}$ by Corollary \ref{cor: simple surface is enough}. Thus 
	    $$\stl_G(ab)=\stl_{\Z/p*\Z/q}(ab)= 1-\frac{q}{p(q-1)}.$$
	\end{proof}

	\begin{figure}
		\centering
		\labellist
		\small \hair 2pt
		\pinlabel $a$ at 265 178
		\pinlabel $b$ at 157 182
		\pinlabel $a$ at 138 207
		\pinlabel $b$ at 167 233
		\pinlabel $a$ at 127 245
		\pinlabel $b$ at 100 275
		\pinlabel $a$ at 87 233
		\pinlabel $b$ at 60 205
		\pinlabel $a$ at 102 193
		\pinlabel $b$ at 102 163
		\pinlabel $a$ at 73 155
		\pinlabel $b$ at 50 192
		\pinlabel $a$ at 32 155
		\pinlabel $b$ at -5 137
		\pinlabel $a$ at 32 115
		\pinlabel $b$ at 47 77
		\pinlabel $a$ at 73 115
		\pinlabel $b$ at 102 105
		\pinlabel $a$ at 102 75
		\pinlabel $b$ at 60 63
		\pinlabel $a$ at 90 35
		\pinlabel $b$ at 115 5
		\pinlabel $a$ at 125 23
		\pinlabel $b$ at 168 35
		\pinlabel $a$ at 140 60
		\pinlabel $b$ at 155 88
		\pinlabel $a$ at 265 90
		\pinlabel $b$ at 345 94
		\pinlabel $a$ at 320 133
		\pinlabel $b$ at 345 173
		
		\pinlabel $a$ at 242 133
		\pinlabel $b$ at 190 133
		\pinlabel $P_4$ at 282 133
		\pinlabel $Q_5$ at 140 133
		\pinlabel $P_4$ at 55 133
		\pinlabel $P_4$ at 115 45
		\pinlabel $P_4$ at 115 220
		
		\endlabellist
		\includegraphics[scale=0.7]{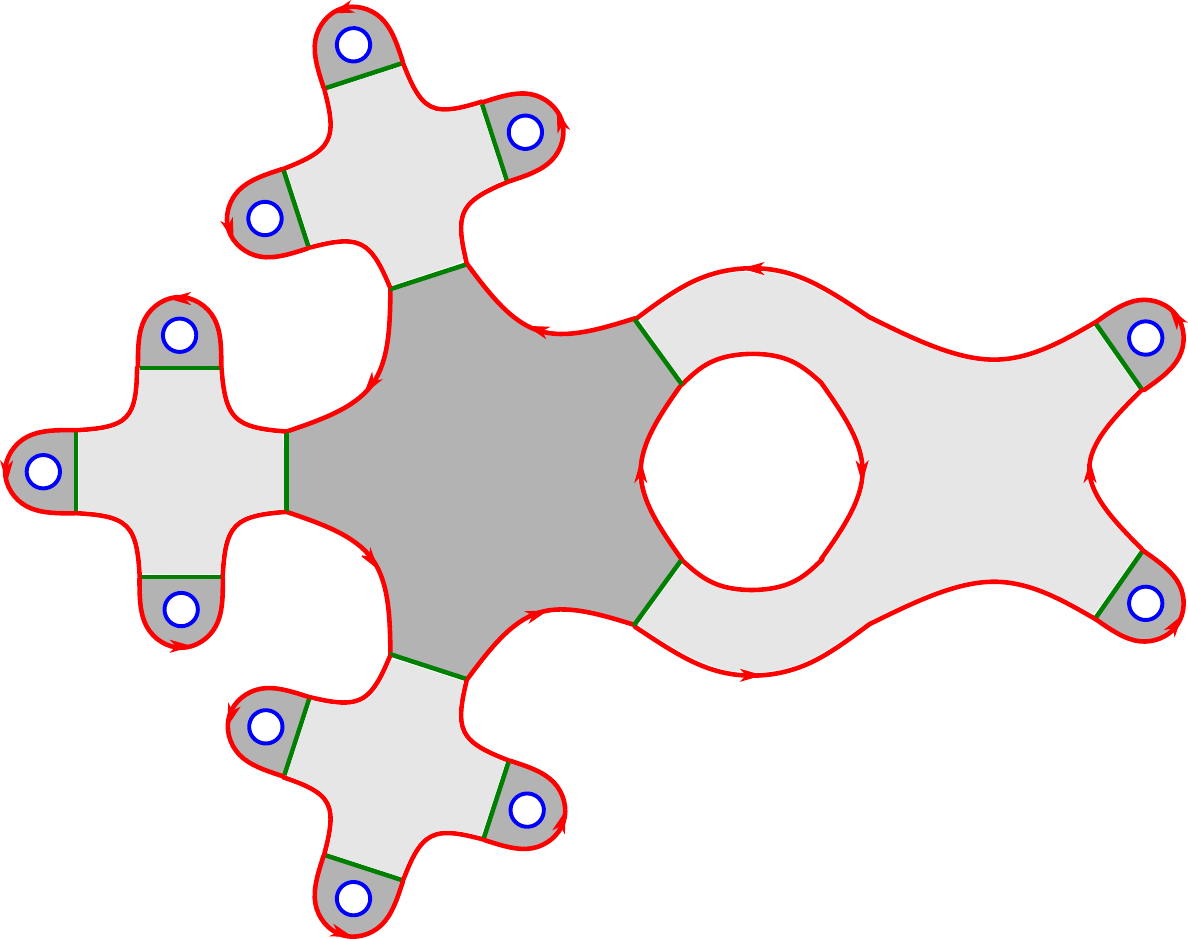}
		\caption{A connected $\mathcal{P}$-simple surface $S$ with the given number of pieces satisfying $\chi(\Gamma_S)=0$ in the case $p=4$ and $q=5$.}
		\label{fig: surfex2}
	\end{figure}
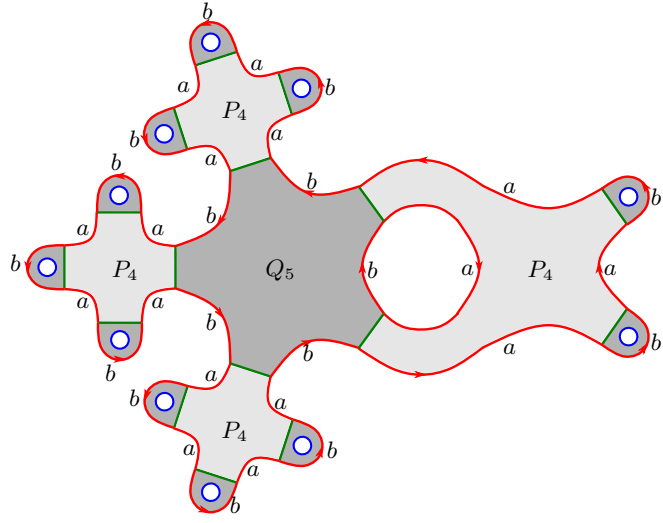
	
	\begin{remark} There is a product formula \cite[Theorem 2.93]{Cal:sclbook} that computes scl of $ab$ in a free product $A*B$ for $a\in A$ and $b\in B$. The result only involves the orders of $a$ and $b$ and $\scl_A(a)$ and $\scl_B(b)$.
	
	It seems unlikely to have such a generalization of Theorem \ref{thm: product formula} for stl when $a$ and $b$ are not necessarily torsion elements. For instance, if $a$ has finite order and $b$ has infinite order, such a formula would express $\stl_G(ab)$ as a function of the order of $a$ and $\stl_B(b)$. This does not seem natural in the following example.
	
	Let $p\leq q \leq r$, and let $\zz /p$, $\zz /q$, and $\zz /r$ be generated by $x$, $y$, and $z$, respectively. The methods in Sections \ref{sec: compute stl} and \ref{sec: examples} generalize to free products of more than two groups. For $G=\Z/p*\Z/q*\Z/r$, a similar calculation as in Theorem \ref{thm: product formula} gives 
	$$\stl_G (xyz)=2-\frac{q}{p(q-1)}.$$
	Consider $G$ as the free product of $A=\Z/p$ and $B=\Z/q*\Z/r$, and let $a=x\in A$ and $b=yz\in B$. It seems unnatural to express $2-\dfrac{q}{p(q-1)}$ as a simple function of $p$ and $\stl_B(b)=1-\frac{r}{q(r-1)}$ since the result depends on $q$ but not on $r$.
	\end{remark}
	
	\bibliographystyle{alpha}
	\bibliography{stl}
	
\end{document}